\documentclass[11pt, oneside, reqno]{amsart}

\usepackage{amsmath, amsthm, amssymb}
\usepackage[usenames,dvipsnames]{color}
\usepackage[all, cmtip]{xy}
\usepackage[pdftex, bookmarks=true, linkbordercolor={0 0 1}]{hyperref}
\usepackage[margin=1in]{geometry}
\usepackage[titletoc,title]{appendix}

\theoremstyle{definition} \newtheorem{definition}{Definition}[section]
\newtheorem{lemma}[definition]{Lemma}
\newtheorem{theorem}[definition]{Theorem}
\newtheorem{prop}[definition]{Proposition}
\newtheorem{conjecture}[definition]{Conjecture}
\newtheorem{corollary}[definition]{Corollary}
\newtheorem{construction}[definition]{Construction}

\newtheorem{claimnum}[definition]{Claim}

\newtheorem*{lemma*}{Lemma}
\newtheorem*{claim}{Claim}

\theoremstyle{definition} \newtheorem{remark}[definition]{Remark}
\theoremstyle{definition} 
\theoremstyle{definition} 
\theoremstyle{definition} 
\theoremstyle{definition} \newtheorem{example}[definition]{Example}
\theoremstyle{definition} \newtheorem{examples}[definition]{Examples}

\newtheorem{pseudoconj}[definition]{Pseudo-Conjecture}

\renewcommand{\gg}{\mathfrak{g}}
\newcommand{\hh}{\mathfrak{h}}
\newcommand{\qq}{\mathfrak{q}}

\newcommand{\bb}[1]{\mathbb{#1}}
\newcommand{\mr}[1]{\mathrm{#1}}
\newcommand{\mc}[1]{\mathcal{#1}}
\newcommand{\mf}[1]{\mathfrak{#1}}
\newcommand{\wt}[1]{\widetilde{#1}}
\newcommand{\bo}[1]{\boldsymbol{#1}}

\newcommand{\inj}{\hookrightarrow}
\newcommand{\bs}{\ \backslash \ }
\newcommand{\dd}{\partial}
\newcommand{\del}{\partial}

\newcommand{\ul}[1]{\underline{#1}}
\newcommand{\ol}[1]{\overline{#1}}

\newcommand{\CC}{\mathbb{C}}
\newcommand{\RR}{\mathbb{R}}
\newcommand{\OO}{\mathcal{O}}
\newcommand{\ZZ}{\mathbb{Z}}

\newcommand{\eps}{\varepsilon}

\newcommand{\SO}{\mathrm{SO}}
\newcommand{\SL}{\mathrm{SL}}
\newcommand{\GL}{\mathrm{GL}}
\newcommand{\SU}{\mathrm{SU}}

\newcommand{\Spin}{\mathrm{Spin}}
\newcommand{\so}{\mathfrak{so}}
\renewcommand{\sl}{\mathfrak{sl}}
\renewcommand{\sp}{\mathfrak{sp}}
\newcommand{\gl}{\mathfrak{gl}}

\newcommand{\sub}{\subseteq}
\newcommand{\iso}{\cong}

\DeclareMathOperator{\tr}{tr}
\DeclareMathOperator{\rank}{rank}
\DeclareMathOperator{\coh}{Coh}
\DeclareMathOperator{\higgs}{Higgs}
\DeclareMathOperator{\bun}{Bun}
\DeclareMathOperator{\Gr}{Gr}
\DeclareMathOperator{\spec}{Spec}
\DeclareMathOperator{\res}{res}
\DeclareMathOperator{\EOM}{EOM}
\DeclareMathOperator{\id}{id}
\DeclareMathOperator{\dvol}{dvol}

\DeclareMathOperator{\sym}{Sym}
\DeclareMathOperator{\Flat}{Flat}
\DeclareMathOperator{\mhiggs}{mHiggs}
\DeclareMathOperator{\mon}{Mon}
\DeclareMathOperator{\diff}{Diff}
\DeclareMathOperator{\Hol}{Hol}
\DeclareMathOperator{\mhitch}{mHitch}
\DeclareMathOperator{\colim}{colim}

\newcommand{\map}{\ul{\mr{Map}}}
\newcommand{\qconn}{q\text{-Conn}}
\newcommand{\conn}{\text{-Conn}}
\newcommand{\epsconn}{\varepsilon\text{-Conn}}
\renewcommand{\d}{\mathrm{d}}
\newcommand{\fr}{\mathrm{fr}}
\newcommand{\red}{\mathrm{red}}
\newcommand{\ad}{\mr{ad}}
\newcommand{\Ad}{\mr{Ad}}
\newcommand{\HT}{\mr{HT}}

\title{Multiplicative Hitchin Systems and Supersymmetric Gauge Theory}
\author{Chris Elliott \and Vasily Pestun}
\date{\today}

\begin{document}

\maketitle 
\begin{abstract}
Multiplicative Hitchin systems are analogues of Hitchin's integrable system based on moduli spaces of $G$-Higgs bundles on a curve $C$ where the Higgs field is group-valued, rather than Lie algebra valued.  We discuss the relationship between several occurences of these moduli spaces in geometry and supersymmetric gauge theory, with a particular focus on the case where $C = \bb{CP}^1$ with a fixed framing at infinity.  In this case we prove that the identification between multiplicative Higgs bundles and periodic monopoles proved by Charbonneau and Hurtubise can be promoted to an equivalence of hyperk\"ahler spaces, and analyze the twistor rotation for the multiplicative Hitchin system.  We also discuss quantization of these moduli spaces, yielding the modules for the Yangian $Y(\gg)$ discovered by Gerasimov, Kharchev, Lebedev and Oblezin.
\end{abstract}

\section{Introduction}
In this paper, we will compare five different perspectives on a single moduli space: three coming from geometry and representation theory, and two coming from supersymmetric gauge theory.  By leveraging these multiple perspectives we'll equip our common moduli space with the structure of a completely integrable system which we call the \emph{multiplicative Hitchin system}.  Many of the structures associated with the ordinary Hitchin system, such as the brane of opers, have parallels in the multiplicative setting.  

Let us begin by presenting these five perspectives, before explaining their inter-relationships.  Let $C$ be a Riemann surface, and let $G$ be a reductive complex Lie group.

\subsection{Multiplicative Higgs Bundles}
We'll first describe the moduli space that motivates the ``multiplicative Hitchin system'' terminology.  Say that a \emph{multiplicative Higgs bundle} on $C$ is a principal $G$-bundle $P$ on $C$ with a meromorphic automorphism $g \colon P \to P$, or
equivalently, a meromorphic section of the group-valued adjoint bundle $\mr{Ad}_P$.  We refer to these by analogy with Higgs bundles, which consist of a meromorphic section of the Lie algebra (co)adjoint
bundle, rather than its Lie group valued analogue \footnote{Usually
  Higgs bundles on a curve are defined to be sections of the coadjoint
  bundle twisted by the canonical bundle.  There isn't an obvious
  replacement for this twist in the multiplicative context, but we'll
  mostly be interested in the case where $C$ is Calabi-Yau, and
  therefore this twist is trivial.}.  Just like for ordinary Higgs
bundles, the space of multiplicative meromorphic Higgs bundles with
arbitrary poles is infinite-dimensional, but we can cut it down to a
finite-dimensional space by fixing the locations of the singularities,
and their local behaviour at each singular point.  For example, for $G = \GL_1$ this local behaviour is given by a degree at each puncture: we restrict to multiplicative Higgs fields $g$ such that, near each singular point $z_i$, $g$ is given by the product of a local holomorphic function and $(z-z_i)^{n_i}$ for some integer $n_i$.  More generally the local ``degrees'' are described by orbits in the
affine Grassmannian $\Gr_G$: restrict the multiplicative Higgs field
to a formal neighborhood of the puncture, to obtain an element of the
algebraic loop group $LG = G(\!(z)\!)$.  This element is well-defined
up to the action of $L^+G = G[[z]]$ on the left and right, and so
determines a well-defined double coset in
$L^+G \!\!\bs\!\! LG/L^+G = L^+G \!\!\bs\!\! \Gr_G$.  Orbits in the
affine Grassmannian are in canonical bijection with dominant coweights
of $\gg$, so fixing a ``degree'' at the punctures means fixing a dominant
coweight $\omega^\vee_{z_i}$ at each puncture $z_i$.
 
 Fix a curve $C$ and consider the \emph{moduli space of multiplicative Higgs bundles} on $C$ with prescribed singularities lying at the points $D = \{z_1, \ldots, z_k\}$, with degrees $\omega_{z_1}^\vee, \ldots, \omega_{z_k}^\vee$ respectively.  We denote this moduli space by $\mhiggs_G(C,D,\omega^\vee)$.
 
 Moduli spaces of multiplicative Higgs bundles have been studied in
 the literature previously, though they've more often been referred to
 simply as ``moduli spaces of $G$-pairs''.  The earliest mathematical  descriptions of these moduli spaces that we're aware of is in the work of Arutyunov, Frolov, and Medvedev \cite{Arutyunov:1996vy,Arutyunov:1996uw} in the integrable systems literature, for $G = \GL(n)$-pairs of zero degree on an elliptic curve, extended by Braden, Chernyakov, Dolgushev, Levin, Olshanetsky
 and Zotov \cite{Chernyakov:2007,Braden:2003gf} to $\GL(n)$-pairs of non-zero degree.  For a general reductive group $G$, moduli spaces of $G$-pairs were studied in the work of Hurtubise and Markman \cite{HurtubiseMarkman, HurtubiseMarkman1}, from a more geometric perspective.  Their work was motivated by the desire to understand the elliptic Skylanin Poisson brackets on
 loop groups, and to that end they too mainly studied the case where the
 underlying curve $C$ is elliptic -- we'll refer back to their work
 throughout the paper.  

 From a different perspective, the moduli space of multiplicative Higgs bundles was also studied by Frenkel and Ng\^o \cite[Section 4]{FrenkelNgo} as part of their geometrization of the trace formula.  Further work was done following their definition by Bouthier \cite{Bouthier2, Bouthier1}, who gave an alternative construction of the moduli space as the space of sections of a vector bundle associated to a fixed singuarity datum defined using the Vinberg semigroup (see Remark \ref{Bouthier_remark}).  
 
 The moduli space of multiplicative Higgs bundles admits a version of the Hitchin fibration, just like the ordinary moduli space of Higgs bundles.  Intuitively, this map is defined using the Chevalley map $\chi \colon G \to T/W$ (for instance if $G = \GL_n$, the map sending a matrix to its characteristic polynomial): the multiplicative Higgs field is, in particular, a meromorphic $G$-valued function, and the multiplicative Hitchin fibration post-composes this function with $\chi$.  See Section \ref{Hitchin_system_section} for a more detailed definition.
 
 While the structure of the multiplicative fibration makes sense for any curve, it only defines a completely integrable system -- in particular there is only a natural symplectic structure on the total space -- in the very special situation where the curve $C$ is Calabi-Yau.  That is, $C$ must be either $\CC$, $\CC^\times$ or an elliptic curve.  In this paper we'll be most interested in the ``rational'' case, where $C = \CC$.  We use a specific boundary condition at infinity: we consider multiplicative Higgs bundles on $\bb{CP}^1$ with a fixed framing at infinity.  In other words we study the moduli space $\mhiggs^{\text{fr}}_G(\bb{CP}^1,D,\omega^\vee)$ of multiplicative Higgs bundles on $\bb{CP}^1$ with prescribed singularities lying at the points $D = \{z_1, \ldots, z_k\}$, with local data at the singuarities encoded by $\omega_{z_1}^\vee, \ldots, \omega_{z_k}^\vee$ respectively, and with a fixed framing at $\infty$.  This moduli space is a smooth algebraic variety with dimension depending on the coweights we chose.  In order to obtain a genuine integrable system we will need to take the Hamiltonian reduction of our moduli space with respect to the adjoint action by the maximal torus $T$ of $G$.  Until we do this the moduli space has the structure of a Lagrangian fibration where the generic fibers look like a product of a compact torus and an affine torus -- a copy of $T$ itself.  We call this Hamiltonian reduction the ``reduced'' moduli space of multiplicative Higgs bundles.
 
 In the case where $C$ is Calabi-Yau the (reduced) moduli space of multiplicative Higgs bundles admits a one-parameter deformation; from the above point of view, the moduli space is hyperk\"ahler so we can rotate the complex structure within the twistor sphere.  The result of this deformation also has a natural interpretation in terms of \emph{$\eps$-connections}.  That is, $G$-bundles $P$ on $C$ equipped with an isomorphism $P \to \eps^*P$ between $P$ and its translate by $\eps \in \CC$.  We can summarize this identification as follows.
 
 \begin{theorem} [See Theorem \ref{HK_rotation_thm}]
 The moduli space $\mhiggs^{\text{red}}_G(\bb{CP}^1,D,\omega^\vee)$, given the complex structure at $\eps \in \CC \sub \bb{CP}^1$ in the twistor sphere can be identified (in the limit where an auxiliary parameter is taken to $\infty$) with the moduli space $\epsconn^{\text{red}}_G(\bb{CP}^1,D,\omega^\vee)$ of framed $\eps$-connections on $\bb{CP}^1$.
 \end{theorem}

\subsection{Moduli Space of Monopoles}

If we focus on the example where a real three-dimensional Riemannian
manifold $M = C \times S^1$ splits as the product of a compact Riemann
surface and a circle, for the rational case $C = \CC$ the moduli space
of monopoles on $M$ was studied by Cherkis and Kapustin
\cite{CherkisKapustin1,CherkisKapustin2,CherkisKapustin3} from the
perspective of the Coulomb branch of vacua of 4d $\mathcal{N}=2$
supersymmetric gauge theory. For general $C$ the moduli space of
monopoles on $M = C \times S^{1}$ was studied by
Charbonneau--Hurtubise \cite{CharbonneauHurtubise}, for
$G_\RR = \mr U(n)$, and by Smith \cite{Smith} for general $G$.  These
moduli spaces were also studied where $C = \CC$ with more general boundary conditions than the framing at infinity
which we consider.  This example has also been studied in the
mathematics literature by Foscolo \cite{FoscoloDef, FoscoloThesis} and
by Mochizuki \cite{Mochizuki} -- note that considering weaker boundary
conditions at infinity means they need to use far more sophisticated
analysis in order to work with hyperk\"ahler structures on their
moduli spaces than we'll consider in this paper.

The connection between periodic monopoles and multiplicative Higgs bundles is provided by the following theorem.
\begin{theorem}[Charbonneau--Hurtubise, Smith]
  There is an analytic isomorphism 
  \begin{equation*}
    H: \mon_G^\fr(C \times S^1,D \times\{0\},\omega^\vee) \to \mhiggs_G^{\text{fr}}(C,D,\omega^\vee)
  \end{equation*}
between the moduli space of (polystable) multiplicative $G$-Higgs bundles on a compact curve $C$ with singularities at points $z_1, \ldots, z_k$ and residues $\omega^\vee_{z_1}, \ldots, \omega^\vee_{z_k}$ and the moduli space of periodic monopoles on $C \times S^1$ with Dirac singularities at $(z_1,0), \ldots, (z_k,0)$ with charges $\omega^\vee_{z_1}, \ldots, \omega^\vee_{z_k}$
\end{theorem}

The morphism $H$,  discussed first by Cherkis and Kapustin in \cite{CherkisKapustin2} and later in
\cite{CharbonneauHurtubise}, \cite{Smith}, \cite{NekrasovPestun}
defines the value of the multiplicative Higgs field at $z \in C$ to be equal to
the holonomy of the monopole connection (complexified by the scalar field) along the fiber circle $\{z\} \times S^{1} \subset M$.

One of our main goals in this paper is to compare -- in the rational case -- the symplectic structure on the moduli space of multiplicative Higgs fields analogous to the symplectic structure defined by Hurtubise and Markman in the elliptic case, and the hyperk\"ahler structure on the moduli space of periodic monopoles, defined by realizing the moduli space of periodic monopoles as a hyperk\"ahler quotient.  We prove that they coincide.

\begin{theorem}[See Theorem \ref{symplectic_comparison_thm}]
  The holomorphic symplectic structure on
  $\mon_G^\fr(\bb{CP}^1 \times S^1,D \times\{0\},\omega^\vee)$ is isomorphic to the pullback of the natural symplectic structure on
  $\mhiggs_G^{\text{fr}}(\bb{CP}^1,D,\omega^\vee)$ under the holonomy morphism $H \colon \mon_G^\fr(\bb{CP}^1 \times S^1,D \times\{0\},\omega^\vee) \to \mhiggs_G^{\text{fr}}(\bb{CP}^1,D,\omega^\vee)$  
\end{theorem}

In particular, after Hamiltonian reduction by the centralizer $T$ of the framing value $g_\infty$, the symplectic structure on the multiplicative Higgs moduli space extends to a hyperk\"ahler structure, as we discussed in point 1 above.

\subsection{Poisson-Lie Groups}
The third perspective that we'll consider connects our moduli spaces directly to the theory of Poisson Lie groups and Lie bialgebras, and leads to an interesting connection to quantum groups upon quantization.  By the \emph{rational Poisson Lie group} we'll mean the group $G_1[[z^{-1}]]$ consisting of $G$-valued power series in $z^{-1}$ with constant term $1$.  More precisely this is an ind-algebraic group: it is expressed as a direct limit of algebraic varieties $G_1[[z^{-1}]]_n$ with an associative multiplication $G_1[[z^{-1}]]_m \times G_1[[z^{-1}]]_n \to G_1[[z^{-1}]]_{m+n}$.  The Poisson structure on this group can be thought of as coming from a Manin triple: specifically the triple $(G(\!(z^{-1})\!), G_1[[z^{-1}]], G[z])$, where $G(\!(z)\!)$ is equipped with the residue pairing (see Section \ref{quantization_section} for more details).  

We claim that our moduli spaces $\mhiggs_G^{\text{fr}}(\bb{CP}^1,D,\omega^\vee)$ correspond to \emph{symplectic leaves} in the rational Poisson Lie group, extending a classification result of Shapiro \cite{Shapiro} for the groups $G = \SL_n$ and $G = \GL_n$.  More specifically, we prove the following.
\begin{theorem}[See Theorem \ref{theorem:symplectic_leaf}]
The map $\mhiggs^\fr_G(\bb{CP}^1,D,\omega^\vee) \to G_1[[z^{-1}]]$, defined by restricting a multiplicative Higgs field to a formal neighborhood of $\infty$, is Poisson.  That is, the pullback of the Poisson structure on $\OO(G_1[[z^{-1}]])$ coincides with the Poisson bracket on $\OO(\mhiggs^\fr_G(\bb{CP}^1,D,\omega^\vee))$. 
\end{theorem}

From this point of view it's natural to try to understand how our moduli spaces behave under deformation quantization.  The quantization of the rational Poisson Lie group is well known: it's modelled by the \emph{Yangian} quantum group $Y(\gg)$.  When we quantize our symplectic leaves, we obtain $Y(\gg)$-modules.  This follows from the work of Gerasimov, Kharchev, Lebedev and Oblezin \cite{GKLO} who constructed the $Y(\gg)$-modules in question and analyzed their classical limits.

\begin{remark}
The article \cite{GKLO}, as well as its generalizations such as the work of Kamnitzer, Webster, Weekes and Yacobi \cite{KWWY}, view these $Y(\gg)$-modules as quantizing certain moduli spaces of monopoles on $\RR^3$, and not of monopoles on $\RR^2 \times S^1$.  These two points of view are expected to be related in the limit where the radius of $S^1$ is sent to infinity, with the positions of the singularities in $S^1$ kept fixed (note that while the holomorphic symplectic structure on the moduli space is independent of this radius, the hyperk\"ahler structure \emph{will} be sensitive to it).  The moduli space of periodic monopoles in the rational case would in this limit be related to the moduli space of monopoles on $\RR^3$ (perhaps with certain restrictions as in the work of Finkelberg, Kuznetsov and Rybnikov on trigonometric Zastava spaces \cite{FKR}).
\end{remark}

\subsection{Moduli Spaces in Supersymmetric Gauge Theory} \label{intro_gauge_section}
Having discussed three points of view on our moduli space in geometric representation theory, let us move on to explain some of the ways in which multiplicative Higgs bundles / periodic monopoles arise naturally in the world of supersymmetric gauge theory.  The general program relating supersymmetric gauge theory and quantum integrable systems relevant for the present work arose from the work of Nekrasov and Shatashvili in \cite{Nekrasov:2009ui,Nekrasov:2009rc}.

Concretely, the moduli spaces of periodic monopoles (and implicitly its quantization) appeared in work of the second author, Nekrasov and Shatashvili \cite{NekrasovPestun, NekrasovPestunShatashvili} concerning the Seiberg-Witten integrable systems of 4d $\mc N=2$ ADE quiver gauge theories.  Let $\gg$ be a simple Lie algebra of ADE type.  One can define an $\mc N=2$ superconformal 4d gauge theory modelled on the Dynkin diagram of $\gg$.  To define this theory one must specify masses $m_{i,j}$ for fundamental hypermultiplets associated to the vertices of the Dynkin diagram, so $i = 1,\ldots,r$ varies over the vertices of the Dynkin diagram and $j=1, \ldots,w_i$ is an index parameterizing the number of hypermultiplets at each vertex.  Let us assume we're in the generic situation, where the masses are all distinct.

The paper \cite{NekrasovPestun} studied the Seiberg-Witten integrable system associated to this theory, so in particular the hyperk\"ahler moduli space $\mf P$, the Coulomb branch of the moduli space of vacua,  occuring as the target of the $\mc N=4$ supersymmetric sigma model obtained by compactifying the 4d theory on a circle.  In particular, it proved the following.

\begin{theorem}[{\cite[Section 8.1]{NekrasovPestun}}]
Let $G$ be a simple Lie group with ADE Lie algebra $\gg$.  The phase space $\mf P$ is isomorphic, as a hyperk\"ahler manifold, to the moduli space $\mon_G^\red(\RR^2 \times S^1,D \times\{0\},\omega^\vee)$ of periodic monopoles, where $D = \{m_{i,j}\}$, and where the residue $\omega^\vee_{m_{i,j}}$ is the fundamental coweight corresponding to the vertex $i$ of the Dynkin diagram of $\gg$.  The notation ``red'' indicates that we fix the value of the holonomy around $S^1$ at infinity, then perform Hamiltonian reduction by the adjoint action of the maximal torus.
\end{theorem}

\begin{remark}
If one considers gauge theories that are not necessarily conformal, but only asymptotically free, the phase space is modelled by periodic monopoles that are permitted to have a singularity at infinity.  If $G_\RR = \SU(2)$ then this moduli space fits into the framework studied by Cherkis--Kapustin and by Foscolo.  We won't discuss this more complicated setting further in this paper.
\end{remark}

\begin{remark}[Elliptic fiber version]
  The analysis of \cite{NekrasovPestun} suggests an elliptic fiber
  generalization of our multiplicative story, corresponding to the
  phase space of a theory modelled on an \emph{affine} ADE quiver.
  The phase space of these theories is identified with a moduli space
  of doubly periodic instantons.  In our setting, the idea is the
  following.  We can think of the moduli space of multiplicative Higgs
  bundles as a moduli space of meromorphic functions into the adjoint
  quotient stack $G/G$, just like we can think of the moduli space of
  ordinary Higgs bundles as a moduli space of meromorphic functions
  into the adjoint quotient stack $\gg/G$ for the Lie algebra.  There
  is a rational/trigonometric/elliptic trichotomy extending these two
  examples: one studies the moduli stack of semistable $G$-bundles on
  a cuspidal, nodal or smooth elliptic curve.  That is:
\begin{align*}
\bun^{\mr{ss}}_G(E^{\mr{cusp}}) &\iso \gg/G \\
\bun^{\mr{ss}}_G(E^{\mr{nod}}) &\iso G/G \\
\bun^{\mr{ss}}_G(E_q) &\iso LG/_q LG.
\end{align*}
The last statement is a theorem of Looijenga (see e.g. \cite{Laszlo}): $LG/_q LG$ denotes the $q$-twisted adjoint quotient of the loop group $LG$: so $g(z)$ acts by $h(z) \mapsto g(qz)^{-1}h(z)g(z)$.  

Without singularities, therefore, the elliptic analogue of our multiplicative Hitchin system should be the space of maps from $C$
into $\bun^{\mr{ss}}_G(E_q)$: this moduli space is closely related to the moduli space of instantons on $C \times E_q$, i.e. doubly periodic instantons.  Moduli spaces of this form for the group $\SL_2$, including their hyperk\"ahler structures, have been studied by Biquard and Jardim \cite{BiquardJardim}.  In particular, the relationship between the elliptic version of the Hitchin system for $G = \SL_2$ and the space of doubly periodic $\SU(2)$-instantons is provided by \cite[Theorem 0.2]{BiquardJardim}.
\end{remark}

\subsection{Multiplicative $q$-Geometric Langlands Correspondence}
The final perspective we'll consider is perhaps the most interesting from our point of view, because it suggests that the categorical geometric Langlands conjecture of Beilinson and Drinfeld might admit a multiplicative analogue, built from the multiplicative Hitchin system instead of the ordinary Hitchin system.  Again, we'll describe a moduli space coming from supersymmetric gauge theory, but in a quite different setting to that of perspective 4 above.

We consider now a five-dimensional $\mc N=2$ supersymmetric gauge theory.  In order to extract interesting moduli spaces for geometric representation theory, we won't study its moduli space of vacua, but instead we'll \emph{twist} the theory, then look at the entire moduli space of solutions to the equations of motion in the twisted theory.  We'll review the idea behind twisting at the beginning of Section \ref{twist_section}, but very informally, we'll choose a supersymmetry $Q_\HT$ that squares to zero, and study the derived space of $Q_\HT$-invariant solutions to the equations of motion.  In Section \ref{twist_section} we'll sketch the following, which also follows from forthcoming work of Butson \cite{Butson}.

\begin{claimnum}
The twist by $Q_\HT$ of $\mc N=2$ super Yang-Mills theory with gauge group $G$, on the 5-manifold $C \times S^1 \times \RR^2$ where $C$ is a Calabi-Yau curve,and  with monopole surface operators placed at the points $(z_1,0), \ldots, (z_k,0)$ in $C \times S^1$ with charges $\omega^\vee_{z_1}, \ldots, \omega^\vee_{z_k}$ respectively, has the following stack of solutions to the equations of motion:
\[\mr{EOM}(C \times S^1 \times \RR^2) = T^*[1]\mhiggs(C,D,\omega^\vee)\]
where $T^*[1]X$ denotes the ``1-shifted cotangent space'' of $X$.
\end{claimnum}

This story leads us to a ``multiplicative'' analogue of the approach to geometric Langlands introduced by Kapustin and Witten \cite{KapustinWitten}.  When we take the radius of $S^1$ to zero we recover Kapustin's partially topological twist \cite{KapustinHolo} of 4d $\mc N=4$ super Yang-Mills, which degenerates to Kapustin-Witten's A- and B-topologically twisted theories.  By ``multiplicative geometric Langlands'' we mean a version of the geometric Langlands conjecture modelled on the multiplicative Hitchin system instead of the ordinary Hitchin system, where the radius of this circle is kept positive.  For more details on what this means, see Section \ref{Langlands_section}.

\begin{remark}[Derived geometry]
  We'll use the language of derived algebraic geometry in several places in this paper, mainly in Sections \ref{mhiggs_def_section} and \ref{twist_section}.  While we believe this perspective provides a clear way of thinking about the moduli spaces we're interested in studying, the derived language is not necessary when we state and prove the main results of this paper.  We hope that the majority of the paper is understandable without a derived geometry background.  There are several clear and references explaining the point of view of derived algebraic geometry, for instance the survey articles of To\"en \cite{ToenOverview,ToenSurvey}.  A comprehensive account of the theory of derived geometry can be found in the book of Gaitsgory and Rozenblyum \cite{GRvol1, GRvol2}.  We'll also refer a few times to the theory of shifted symplectic structures.  This theory was developed by Pantev, To\"en, Vaqui\'e and Vezzosi \cite{PTVV}, and we also recommend the explanations in \cite{Calaque}.  From the physical point of view, the ideas of derived geometry appear in classical field theory through the Batalin-Vilkovisky formalism, and through Pantev, To\"en, Vaqui\'e and Vezzosi's derived version of the AKSZ construction of symplectic structures on mapping spaces \cite{AKSZ}. 
\end{remark}

\subsection{Outline of the Paper}
The paper is divided into two parts.  In Part \ref{part1} we will investigate the symplectic geometry of the moduli space of multiplicative Higgs bundles, and in Part \ref{part2} we will study the connection between multiplicative Higgs bundles and periodic monopoles, and the hyperk\"ahler geometry of these moduli spaces.

We begin in Section \ref{mhiggs_def_section} by introducing and defining the moduli spaces we'll be studying in this paper: the moduli spaces of multiplicative Higgs bundles.  We discuss the trichotomy of rational, trigonometric and elliptic examples that are most relevant to us: where the moduli space has the structure of an integrable system.  We include a discussion of how the symplectic structures we're studying in this paper are expected to arise from the theory of derived symplectic geometry.

The following section, Section \ref{twist_section}, stands alone from the rest of the paper.  In this section we explain how the moduli space of multiplicative Higgs bundles appears when one studies a certain partially topological twist of 5d $\mc N=2$ supersymmetric gauge theory.  As a consequence, we can speculate on the existence of a multiplicative version of the geometric Langlands conjecture, using the work of Kapustin and Witten.

In Section \ref{symp_section} we construct a symplectic structure on the moduli space of multiplicative Higgs bundles in the rational case.  We do this by realizing the moduli space as a symplectic leaf in the infinite-dimensional Poisson Lie group $G_1[[z^{-1}]]$ of $G$-valued Taylor series with constant leading coefficient.

We conclude Part \ref{part1} with Section \ref{quantization_section}, in which we discuss the quantization of our moduli spaces.  We identify the moduli spaces of multiplicative Higgs bundles with symplectic leaves in the rational Poisson Lie group, and then after quantization we identify the quantized algebra of functions on our moduli spaces with modules for the Yangian as first constructed by Gerasimov, Kharchev, Lebedev and Oblezin. 

We begin Part \ref{part2} of the paper with Section \ref{periodic_monopole_section}, in which we introduce the other main player of the talk, the moduli space of periodic monopoles.  We explain how these moduli spaces arise as a hyperk\"ahler quotient, and discuss the relationship, following Charbonneau--Hurtubise \cite{CharbonneauHurtubise} and Smith \cite{Smith} with the multiplicative Higgs moduli space.

In Section \ref{symp_comparison_section} we extend this result to an equivalence of holomorphic symplectic manifolds, using the construction of the symplectic structure from the previous part of the paper.  Then in Section \ref{hyperkahler_section} we investigate the twistor family of holomorphic symplectic structures on the two equivalent moduli spaces, and prove that twistor rotation of the multiplicative Higgs moduli space is equivalent to the deformation from Higgs fields to $q$-difference connections.

Finally, in Section \ref{qchar_section} we discuss the relationship between the multiplicative analogue of the space of opers and the $q$-character, motivated by the origin of the multiplicative Hitchin system in quiver gauge theory.

The preliminary results of this work were announced by the second author at String-Math 2017 \cite{PestunStringMath}.

\subsection{Acknowledgements}
We would like to thank David Jordan, Davide Gaiotto, Dennis Gaitsgory, Nikita Nekrasov, Kevin Costello and especially Pavel Safronov for helpful conversations about this work. The calculation of twists of 5d and 6d supersymmetric gauge theories was performed independently by Dylan Butson, and we would like to thank him for sharing his forthcoming manuscript.  We are also very grateful to Takuro Mochizuki for pointing out an error in an earlier version of the article. CE would like to thank the Perimeter Institute for Theoretical Physics for supporting research on this project. Research at Perimeter Institute is supported by the Government of Canada through Industry Canada and by the Province of Ontario through the Ministry of Economic Development \& Innovation. We acknowledge the support of IH\'ES.  This project has received funding from the European Research Council (ERC) under the European Union's Horizon 2020 research and innovation programme (QUASIFT grant agreement 677368).

\renewcommand{\thepart}{\Alph{part}}
\part{Symplectic Structures} \label{part1}

\section{Multiplicative Higgs Bundles and $q$-Connections} \label{mhiggs_def_section}
We'll begin with an abstract definition of moduli spaces of multiplicative Higgs bundles using the language of derived algebraic geometry.  We note, however, that once we specialize to our main, rational, example, the moduli spaces we'll end up studying are actually smooth algebraic varieties, not derived stacks.  However, the derived point of view gives us a useful and concise definition of these moduli spaces in full generality.

Let $G$ be a reductive complex algebraic group, let $C$ be a smooth complex algebraic curve and fix a finite set $D = \{z_i, \ldots, z_k\}$ of closed points in $C$.  We write $\bun_G(C)$ for the moduli stack of $G$-bundles on $C$, which we view as a mapping stack $\map(C, BG)$ into the classifying stack of $G$.

\begin{definition}
The moduli stack of \emph{multiplicative $G$-Higgs fields} on $C$ with singularities at $D$ is the fiber product
\[\mhiggs_G(C,D) = \bun_G(C) \times_{\bun_G(C \! \bs \! D)} \map(C \! \bs \! D, G/G)\]
where $G/G$ is the adjoint quotient stack.
\end{definition}

\begin{remark}
A closed point of $\mhiggs_G(C,D)$ consists of a principal $G$-bundle $P$ on $C$ along with an automorphism of the restriction $P|_{C \! \bs \! D}$, i.e. a section of $\Ad_P$ away from $D$.
\end{remark}

The adjoint quotient stack can also be described as the derived loop space $\map(S^1_B, BG)$ of the classifying stack, where $S^1_B$ is the ``Betti stack'' of $S^1$, i.e. the constant derived stack at the simplicial set $S^1$.  We can therefore view $\map(C \! \bs \! D, G/G)$ instead as the mapping stack $\map((C \! \bs \! D) \times S^1_B, BG)$, and the moduli stack of multiplicative Higgs bundles instead as
\[\mhiggs_G(C,D) = \map((C \times S^1_B) \bs (D \times \{0\}), BG).\]  
The source of this mapping stack can be $q$-deformed.  Indeed, let $q$ denote an automorphism of the curve $C$.  Write $C \times_q S^1_B$ for the \emph{mapping torus} of $q$, i.e the derived fiber product
\[C \times_q S^1_B = C \times_{C \times C} C\]
where the two maps $C \to C \times C$ are given by the diagonal and the $q$-twisted diagonal $x \mapsto (x,q(x))$ respectively.

\begin{definition}
The moduli stack of \emph{$q$-difference connections} for the group $G$ on $C$ with singularities at $D$ is the mapping space
\[\qconn_G(C,D) = \map((C \times_q S^1_B) \bs (D \times \{0\}), BG).\] 
In particular when $q=1$ this recovers the moduli stack of multiplicative Higgs bundles.
\end{definition}

\begin{remark}
  A closed point of $\qconn_G(C,D)$ consists of a principal $G$-bundle
  $P$ on $C$ along with a \emph{$q$ difference connection}: an
  isomorphism of $G$-bundles
  $P|_{C \! \bs \! D} \to q^*P|_{C \! \bs \! D}$ away from the divisor
  $D$.  For an introduction and review of the classical theory of $q$-difference
  connections we refer the reader to \cite{STSSevostyanov} and \cite{Sauloy}.
\end{remark}

\subsection{Local Conditions at the Singularities}
These moduli stacks are typically of infinite type.  In order to obtain finite type stacks, and later in order to define symplectic rather than only Poisson structures, we can fix the behaviour of our multiplicative Higgs fields and $q$-difference connections near the punctures $D \sub C$.

We'll write $\bb D$ to denote the \emph{formal disk} $\spec \CC[[z]]$.  Likewise we'll write $\bb D^\times$ for the \emph{formal punctured disk} $\spec \CC(\!(z)\!)$.  We'll then write $\bb B$ for the derived pushout $\bb D \sqcup_{\bb D^\times} \bb D$.  Let $LG = \map(\bb D^\times, G)$ and let $L^+G = \map(\bb D, G)$.

There is a canonical inclusion $\bb B^{\sqcup k} \to (C \times_q S^1_B) \!\!\bs\!\! (D \times \{0\})$, the inclusion of the formal punctured neighborhood of $D \times \{0\}$.  This induces a restriction map on mapping spaces
\[\res_D \colon \qconn_G(C, D) \to \bun_G(\bb B)^k.\]

One can identify $\bun_G(\bb B)$ with the double quotient stack $L^+G \!\bs\! LG / L^+G$, or equivalently with the quotient $L^+G \!\bs\! \Gr_G$ of the affine Grassmannian.  The following is well-known (see e.g. the expository article \cite{Zhu}).

\begin{lemma}
The set of closed points of $\bun_G(\bb B)$ is in canonical bijection with the set of dominant coweights of $G$.
\end{lemma}

\begin{definition}
Choose a map from $D$ to the set of dominant coweights and denote it by $\omega^\vee \colon z_i \mapsto \omega^\vee_{z_i}$.  Write $\Lambda_i$ for the isotropy group of the point $\omega^\vee_{z_i}$ in $\bun_G(\bb B)$. The moduli stack of $q$ difference connections on $C$ with singularities at $D$ and fixed local data given by $\omega^\vee$ is defined to be the fiber product
\[\qconn_G(C,D, \omega^\vee) = \qconn_G(C,D) \times_{\bun_G(\bb B)^k} (B\Lambda_1 \times \cdots \times B\Lambda_k).\]
\end{definition}

\begin{remark}\label{ind_structure_remark}
  In a similar way we can define a filtration on the moduli stack of
  $q$-connections.  Recall that the affine Grassmannian $\Gr_G$ is
  stratified by dominant coweights $\omega^\vee$, and there is an
  inclusion $\Gr_G^{\omega^\vee_1}\sub \ol{\Gr_G^{\omega^\vee_2}}$ of
  one stratum into the closure of another stratum if and only if
  $\omega^\vee_1 \preceq \omega^\vee_2$ with respect to the standard
  partial order on dominant coweights (again, this is explained in \cite{Zhu}).  We can then define
\[\qconn_G(C,D, \preceq \omega^\vee) = \qconn_G(C,D) \times_{\bun_G(\bb B)^k} \left(L^+G \!\!\bs\!\! \ol{\Gr_G^{\omega^\vee_1}} \times \cdots \times L^+G \!\!\bs\!\! \ol{\Gr_G^{\omega^\vee_k}}\right).\]
The full moduli stack $\qconn_G(C,D)$ is the filtered colimit of these moduli spaces.  One can additionally take the filtered colimit over all finite subsets $D$ in order to define a moduli stack $\qconn^\mr{sing}_G(C)$ of $q$-connections on $C$ with arbitrary singularities.
\end{remark}

\begin{examples}
The most important examples for our purposes are given by the following rational/trigonometric/elliptic trichotomy.
 \begin{itemize}
  \item \textbf{Rational:} We can enhance the definition of our moduli space by including a framing at a point $c \in C$ not contained in $D$.  We always assume that such framed points are fixed by the automorphism $q$.
    \begin{definition}
      \label{def:framing}
    The moduli space of $q$-difference connections on $C$ with a framing at $c$ is defined to be the relative mapping space 
    \[\qconn_G^\fr(C) = \map(C \times_q S^1_B, BG; f)\]
    where $f \colon \{c\} \times S^1_B \to BG$ (or equivalently $f \colon \{c\} \to G/G$) is a choice of adjoint orbit.  We can define the framed mapping space with singularities and fixed local data in exactly the same way as above.  
  \end{definition}
    
  In this paper we'll be most interested in the following example.  Let $C = \bb{CP}^1$ with framing point $c = \infty$ and framing given by a fixed element in $G/G$, and consider automorphisms of the form $z \mapsto z + \eps$ for $\eps \in \CC$.  Choose a finite subset $D \sub \bb A^1$ and label the points $z_i \in D$ by dominant coweights $\omega^\vee_{z_i}$.  We can then study the moduli space $\epsconn^\fr_G(\bb{CP}^1,D, \omega^\vee)$.  The main object of study in this paper will be the holomorphic symplectic structure on this moduli space. Note that the motivation for this definition comes in part from
    Spaide's formalism \cite{Spaide} of AKSZ-type derived symplectic structures (in the sense of \cite{AKSZ,PTVV}) on
    relative mapping spaces -- in this formalism $\bb{CP}^1$ with a single framing point is relatively 1-oriented, so mapping spaces
    out of it with 1-shifted symplectic targets have AKSZ 0-shifted symplectic structures.
  
  \item \textbf{Trigonometric:} Alternatively, we can enhance our definition by including a reduction of structure group at a point $c \in C$ not contained in $D$, again fixed by the automorphism $q$.
  \begin{definition}
   The moduli space of $q$-difference connections on $C$ with an $H$-reduction at $c$ for a subgroup $H \sub G$ is defined to be the fiber product
   \[\qconn_G^{H,c}(C) = \map(C \times_q S^1_B,BG) \times_{G/G} H/H\]
   associated to the evaluation at $c$ map $\map(C \times_q S^1_B,BG) \to G/G$.  We can define the moduli space with $H$-reduction with singularities and fixed local data in the same way as above.
  \end{definition}
  
  Again let $C = \bb{CP}^1$.  Fix a pair of opposite Borel subgroups $B_+$ and $B_- \sub G$ with unipotent radicals $N_\pm$ and consider the moduli space of $q$-connections with $B_+$-reduction at $0$ and $N_-$-reduction at $\infty$.  We'll now take $q$ to be an automorphism of the form $z \mapsto qz$ for $q \in \CC^\times$.  We'll defer in depth analysis of this example to future work.
  
  \item \textbf{Elliptic:} Finally, let $C = E$ be a smooth curve of genus one.  In this case we won't fix any additional boundary data, but just consider the moduli space $\qconn_G(E,D, \omega^\vee)$.  In the case $q = 1$ this space -- or rather its polystable locus -- was studied by Hurtubise and Markman \cite{HurtubiseMarkman}, who proved that it can be given the structure of an algebraic integrable system with symplectic structure related to the elliptic R-matrix of Etingof and Varchenko \cite{EtingofVarchenko}.
 \end{itemize}
\end{examples}

\begin{remark}
In the rational case, the moduli space of framed $q$-difference connections now depends on a new parameter: the value of the framing $g_\infty \in G/G$.  From the point of view of the ADE quiver gauge theory, as in Section \ref{intro_gauge_section}, this value -- or rather its image in $H/W$ -- corresponds to the value of the gauge coupling constants in the quiver gauge theory.

In a little more detail, for an ADE quiver gauge theory, a physical choice of finite gauge coupling constants $\mathfrak{q}_i = \exp (2 \pi i \tau_i)$ corresponds to a choice of regular semi-simple conjugacy class $[g_{\inf}]  = \prod_{i} \mathfrak{q}_{i}^{-\check \omega_i}$.  So the contribution of an instanton with second Chern class $k_i$ with respect to the $i^{\text{th}}$ factor of the gauge group is counted with weight $\mathfrak{q}_{i}^{k_i}$.
\end{remark}

\begin{remark} \label{Elliptic_AKSZ_remark}
In the elliptic case it's natural to ask to what extent Hurtubise and Markman's integrable system structure can be extended from the variety of polystable multiplicative Higgs bundles to the full moduli stack.  If $D$ is empty then it's easy to see that we have a symplectic structure given by the AKSZ construction of Pantev-To\"en-Vaqui\'e-Vezzosi \cite{PTVV}.  Indeed, $E$ is compact 1-oriented and the quotient stack $G/G$ is 1-shifted symplectic, so the mapping stack $\map(E, G/G) = \mhiggs_G(E)$ is equipped with a 0-shifted symplectic structure by \cite[Theorem 2.5]{PTVV}.  The role of the Hitchin fibration is played by the Chevalley map $\chi \colon G/G \to T/W$, and therefore
\[\map(E,G/G) \to \map(E,T/W).\]
The fibers of this map over regular points in $T/W$ are given by moduli stacks of the form $\bun_T(\wt E)^W$ where $\wt E$ is a $W$-fold cover of $E$ (the cameral cover).  Note that in this unramified case the curve $\wt E$ also has genus 1; counting dimensions we see that the base has dimension $r = \mr{rk}(G)$ and the generic fibers are $r$-dimensional (Lagrangian) tori.
\end{remark}

\begin{remark}
While the moduli space of multiplicative Higgs bundles makes sense on a general curve it's only after restricting attention to this trichotomy of examples that we'll expect the existence of a Poisson structure.  In the non-singular case, such a structure arises by the AKSZ construction, i.e. by transgressing the 1-shifted symplectic structure on $G/G$ to the mapping space using a fixed section of the canonical bundle on $C$ (possibly with a boundary condition).  
\end{remark}

\subsection{The Multiplicative Hitchin System} \label{Hitchin_system_section}
We can define the global Chevalley map as in Remark \ref{Elliptic_AKSZ_remark} in the case of non-empty $D$ as well.  We'll show that in the rational case this defines a completely integrable system structure.  Let $T \sub G$ be a maximal torus, and write $W$ for the Weyl group of $G$.

\begin{definition} \label{mult_Hitchin_system_def}
Fix a curve $C$, a divisor $D$ and a dominant coweight $\omega_{z_i}^\vee$ at each point $z_i$ in $D$.  The \emph{multiplicative Hitchin base} is the stack
\[\mc B(C,D,\omega^\vee) = \mr{Sect}(C, X(D,\omega^\vee)/W)\]
of sections of $X(D,\omega^\vee)/W$, the $T/W$-bundle on $C$ where $X(D,\omega^\vee)$ is the $T$-bundle characterized by the condition that the associated line bundle $X(D,\omega^\vee) \times_T {\lambda}$ corresponding to a weight $\lambda$ is given by $\OO(\sum \omega_{z_i}^\vee(\lambda) \cdot z_i)$ (c.f. \cite[Section 3.3]{HurtubiseMarkman}).

The \emph{multiplicative Hitchin fibration} is the map
\[\pi \colon \mhiggs_G(C,D,\omega^\vee) \to \mc B(C,D,\omega^\vee)\]
given by post-composing a map $C \bs D \to G/G$ with the Chevalley map $\chi \colon G/G \to T/W$.  
\end{definition}

\begin{prop}
The multiplicative Hitchin fibration described above is well-defined.
\end{prop}

\begin{proof}
We need to verify that the image of a point in $\mhiggs_G(C,D,\omega^\vee)$ under $\pi$, viewed as a section of the trivial $T/W$-bundle on $C \bs D$, extends to a section of $X(D,\omega^\vee)/W$.  We look locally near a singularity $z_i \in D$.  Let $\phi \in G((z_i))$ be a local representative for a multiplicative Higgs field in $\mhiggs_G(C,D,\omega^\vee)$, i.e. an element of the associated $G[[z_i]]$-adjoint orbit.  Since $\phi$ is equivalent to $z_i^{-\omega^\vee_{z_i}}$ under the action of $G[[z_i]]^2$ by left and right multiplication, without loss of generality we can say that $\phi = z_i^{-\omega^\vee_{z_i}} \phi_0$ for some $\phi_0 \in G[[z_i]]$.  Consider $\chi(\phi) \in T((z_i))$ -- the singular part of this element is the same as $\chi(z_i^{-\omega^\vee_{z_i}} g)$ for some $g \in G$, which implies the section extends to a meromorphic $T/W$-valued function of the required type.
\end{proof}

\begin{remark}
Note that the base stack that we're defining here is not exactly the same as the base of the integrable system defined in \cite{HurtubiseMarkman}.  The difference comes from the way in which we treat the Weyl group quotient; the Hurtubise-Markman base is an algebraic variety defined essentially by taking the part of our space of sections where $W$ acts freely as an open subset, then constructing a compactification using techniques from toric geometry.
\end{remark}

We'll argue in Section \ref{Reduced_section} that this fibration defines a completely integrable system in the rational case where $C = \bb{CP}^1$ with a framing at $\infty$. More specifically, we'll describe a non-degenerate pairing on the tangent space of the multiplicative Higgs moduli space and check that the multiplicative Hitchin fibers are isotropic for this pairing.  We'll then show in Theorem \ref{theorem:symplectic_leaf} that the pairing was in fact symplectic.

Here, we'll first observe that the generic fibers are half-dimensional tori, which will ultimately imply that the multiplicative Hitchin system is generically a Lagrangian fibration.  Computing the fibers of the Hitchin fibration works similarly to the non-singular case: a point in the base is, in particular, a map $C \bs D \to T/W$.  Suppose this map lands in the regular locus $T^{\mr{reg}}/W$, then an element of the fiber over this point defines, in particular, a map $C \bs D \to BT/W$.  We would like to argue that the fiber consists of $T$-bundles on the cameral cover $\wt C$: a $W$-fold cover of $C$ ramified at the divisor $D$.

In order to say this a bit more precisely this we'll compare our moduli space with the space of abstract Higgs bundles introduced by Donagi and Gaitsgory \cite{DonagiGaitsgory} (see also \cite{DonagiLectures}, where Donagi proposed the applicability of this abstract Higgs theory to the multiplicative situation and asked for a geometric interpretation).  Our argument will follow the same ideas as the arguments of \cite[Section 6]{HurtubiseMarkman}.

\begin{definition}
An \emph{abstract $G$-Higgs bundle} on a curve $C$ is a principal $G$-bundle $P$ along with a sub-bundle $\mf c$ of $\gg_P$ of \emph{regular centralizers}, meaning that the fibers are subalgebras of $\gg$ which arise as the centralizer of a regular element of $\gg$.  Write $\higgs_G^{\mr{abs}}(C)$ for the moduli stack of abstract $G$-Higgs bundles on $C$.
\end{definition}

There's an algebraic map from the regular part of our moduli space $\mhiggs_G(C,D,\omega^\vee)_{\mr{reg}}$ (where the Higgs field is required to take regular values) into $\higgs_G^{\mr{abs}}(C)$ that sends a multiplicative Higgs bundle $(P,g)$ to the abstract Higgs bundle $(P, \mf c_g)$, where $\mf c_g$ is the sub-bundle of $\gg_P$ fixed by the adjoint action of the multiplicative Higgs field $g$.  What's more, there is a commutative square relating the Hitchin fibration for the multiplicative moduli space with a related projection for the abstract moduli stack:
\[\xymatrix{
\mhiggs_G(C,D,\omega^\vee)_{\mr{reg}} \ar[r] \ar[d] &\higgs_G^{\mr{abs}}(C) \ar[d] \\
\mc B(C,D,\omega^\vee)_{\mr{reg}} \ar[r] &\mr{Cam}_G(C),
}\]
where $\mr{Cam}_G(C)$ is the stack of cameral covers of $C$, as defined in \cite[Section 2.8]{DonagiGaitsgory}.  The map $\mc B(C,D,\omega^\vee)_{\mr{reg}} \to \mr{Cam}_G(C)$ is defined by sending a meromorphic function $f \colon C \to T/W$ to the $D$-ramified cameral cover $\wt C = C \times_{T/W} T$.  In particular there's a map from the multiplicative Hitchin fiber to the corresponding Donagi-Gaitsgory fiber: the moduli space of abstract $G$-Higgs bundles with fixed cameral cover.  This map is surjective: having fixed the cameral cover, and therefore the ramification data, every sub-bundle $\mf c \sub \gg_P$ of regular centralizers arises as the centralizer of some regular multiplicative Higgs field.  Likewise once one restricts to a single generic multiplicative Hitchin fiber the map is an unramified $W$-fold cover.

To conclude this discussion we'll discuss dimensions and the geometry of the multiplicative Hitchin fibers.  Firstly, we can compute the dimension of a regular multiplicative Hitchin fiber by computing the dimension of the base and the dimension of the total space.  The dimension of the base is given by computing the number of linearly independent sections of the $T$-bundle $X(D,\omega^\vee)$ on $\bb{CP}^1$ vanishing at $\infty$.  This is given by 
\[\dim \mc B(C,D,\omega^\vee) = \sum_{z_i \in D} \langle \rho, \omega^\vee_{z_i} \rangle,\]
where $\rho$ is the Weyl vector.  On the other hand the dimension of the total space is calculated in Corollary \ref{dim_of_moduli_space_cor} to be $2 \sum_{z_i \in D} \langle \rho, \omega^\vee_{z_i} \rangle$.  The base is indeed half-dimensional, therefore so is the fiber.

The Donagi-Gaitsgory fiber is, according to the main theorem of \cite{DonagiGaitsgory}, equivalent to the moduli space of $W$-equivariant $T$-bundles on the cameral curve $\wt C$ up to a discrete correction involving the root datum of $G$.  In particular it is generically an abelian variety. The multiplicative Hitchin fiber is isogenous to this abelian variety, since the map from the multiplicative Hitchin fiber to the Donagi-Gaitsgory fiber is surjective and \'etale.

\begin{remark} \label{non_compact_fiber_remark}
This argument that the fibers are abelian varieties does \emph{not} apply to our main case of interest, because it does not account for the data of the framing at $\infty \in \bb{CP}^1$.  In fact, according to the physical arguments in \cite{NekrasovPestun}, we do not expect the fibers in this example to be compact.  In order to resolve this we will need to take the Hamiltonian reduction of $\mhiggs^\fr(\bb{CP}^1, D, \omega^\vee)$ by the adjoint action of a maximal torus of $G$.  We'll discuss this reduction in Section \ref{Reduced_section}.
\end{remark}

\begin{remark} \label{q_opers_remark}
Like in the case of the ordinary Hitchin system, in good examples the multiplicative Hitchin system admits a canonical Hitchin section.  One can construct this section using the Steinberg section (the multiplicative analogue of the Kostant section) \cite{Steinberg}.  This is a section of the map $G/G \to T/W$, canonical after choosing a Borel subgroup $B$ with maximal torus $T$ and a basis vector for each simple root space, and well defined as long as $G$ is simply connected \footnote{By a theorem of Popov \cite{Popov} this condition is necessary for semisimple $G$.  A section also exists for $G = \GL_n$, but we aren't aware of a necessary and sufficient condition for general reductive groups.}.  The \emph{multiplicative Hitchin section} is the map $\sigma \colon \mc B(C,D,\omega^\vee) \to \mhiggs_G(C,D,\omega^\vee)$ defined by post-composition with the Steinberg section.  One can use this section to define the moduli space of \emph{$q$-opers} for the group $G$ and the curve $\bb{CP}^1$ with its framing at infinity.  We'll discuss the hyperk\"ahler structure on the reduced moduli space of multiplicative Higgs bundles in Section \ref{hyperkahler_section}.  In particular we'll show that when one rotates to $q$ in the twistor sphere one obtains the moduli space of $q$-connections on $\bb{CP}^1$.  The moduli space of $q$-opers is defined to be the Hitchin section, but viewed as a subspace of $\qconn^{\fr}_G(\bb{CP}^1,D,\omega^\vee)$.  For a more detailed discussion of the multiplicative Hitchin section and $q$-opers see Section \ref{qchar_section}. 
\end{remark}

\subsection{Stability Conditions} \label{stability_section}
For comparison to results in the literature it is important that we briefly discuss the role of stability conditions for difference connections.  In our main example of interest -- the rational case -- these conditions will take a relatively simple form, but they do appear more generally in the comparison results between $q$-connections and monopoles in the literature for more general curves.  For definitions for general $G$ we refer to \cite{Smith}, although see also \cite{AnchoucheBiswas} on polystable $G$-bundles.  In what follows we fix a choice of vector $\overrightarrow{t}$ with $0 < t_1 < \cdots < t_k < 2\pi R$.

\begin{remark} \label{t_stability_remark}
There will be a constraint on the allowable values for the $t_i$, as explained in \cite[Section 3.2]{CharbonneauHurtubise}. 
\end{remark}

\begin{definition} \label{polystable_def}
Let $(P,g)$ be a $q$-connection on a curve $C$, and let $\chi$ be a character of $G$.  The \emph{$(\chi, \overrightarrow{t})$-degree} of $(P,g)$ is defined to be 
\[\deg_\chi(P,g) = \deg(P \times_\chi \CC) - \frac 1{2\pi R} \sum_{i=1}^k t_i\deg(\chi \circ \omega^\vee_{z_i}).\]

A $q$-connection $(P,g)$ on $C$ is \emph{stable} if for every maximal parabolic subgroup $H \sub G$ with Levi decomposition $H = LN$ and every reduction of structure group $(P_H, g)$ to $H$, we have
\[\deg_\chi(P_H, g) < 0\]
for the character $\chi = \det(\mr{Ad}_L^{\mf n})$ defined to be the determinant of the adjoint representation of $L$ on $\mf n$.

The $q$-connection $(P,g)$ is \emph{polystable} if there exists a (not necessarily maximal) parabolic subgroup $H$ with Levi factor $L$ and a reduction of structure group $(P_L, g)$ to $L$ so that $(P_L,g)$ is a stable $q$-connection and so that the associated $H$-bundle is admissible, meaning that for every character $\chi$ of $H$ which is trivial on $Z(G)$ the associated line bundle $P_H \times_\chi \CC$ has degree zero. 
\end{definition}

\begin{remark}
The $(\chi, \overrightarrow{t})$-degree here, at least in the case where $G = \GL_n$, can be thought of as the average, over $S^1$ of the degree of the holomorphic vector bundle obtained by restricting $P$ to the slice $C \times \{t\}$ for $t \in S^1$.  See the argument in \cite[Lemma 4.5]{CharbonneauHurtubise}. 
\end{remark}

Below we'll write $\qconn_G^{\text{ps}}(C, D, \omega^\vee) \sub \qconn_G(C,D,\omega^\vee)$ for the moduli space of polystable $q$-connections.  This moduli space is a smooth algebraic variety of finite type, at least when $q$ is the identity \cite{CharbonneauHurtubise,Smith}.  

\begin{example}
An example of a multiplicative Higgs bundle on $\bb{CP}^1$ which is not polystable, for the group $G = \SL_2$, is provided by taking the trivial $G$-bundle $P$ and the constant multiplicative Higgs field with value \[g = \begin{pmatrix}1&1\\0&1\end{pmatrix}.\]  The multiplicative Higgs field does not fix a maximal torus in $G$, so the pair $(P,g)$ cannot be reduced to any non-trivial Levi subgroup.  On the other hand, for the upper triangular Borel subgroup $B \sub G$ we have $\deg_\chi(P_B, g) = 0$, so $(P,g)$ is not stable.
\end{example}

In our main example of interest -- the rational setting where $C = \bb{CP}^1$ with a fixed framing at infinity -- we will impose an additional stability condition constraining the value of the framing.  We'll discuss why this condition is necessary in Section \ref{Reduced_section} (from the point of view of Hamiltonian reduction) and Section \ref{periodic_monopole_section} (when we compare to the moduli space of periodic monopoles).

\begin{definition} \label{rational_stability_def}
A $q$-connection $(P,g)$ on $\bb{CP}^1$ with a fixed framing $g_\infty$ at $\infty$ is polystable if it is polystable as in Definition \ref{polystable_def}, and the chosen framing $g_\infty$ is regular semisimple.
\end{definition}

\begin{remark}
We observe that the element $g_\infty$ determines an $S^1$-equivariant holomorphic $G$-bundle on an elliptic curve, where $S^1$ acts on $T^2$ by rotating one of the two circles.  This $G$-bundle is polystable if $g_\infty$ is semisimple, which ensures that the $G$-bundle admits a reduction of structure group to a $T$-bundle for a maximal torus $T \sub G$.
\end{remark}

\subsection{Poisson Structures from Derived Geometry}
As we mentioned above in Remark \ref{Elliptic_AKSZ_remark}, in the case where $C$ is an elliptic curve and there are no punctures there is a symplectic structure on $\mhiggs_G(C)$ given by the AKSZ formalism.  More generally, when we do have punctures, we expect the moduli space $\mhiggs_G(C,D)$ to have a Poisson structure with a clear origin story coming from the theory of derived Poisson geometry.  In this section we'll explain what this story looks like.  However, we emphasise that there are technical obstructions to making this story precise with current technology: this section should be viewed as motivation for the structures we'll discuss in the rest of the paper.  On the other hand, readers who aren't familiar with derived symplectic geometry can freely skip this section.  We refer the reader to \cite{CPTVV} for the theory of derived Poisson structures and to \cite{MelaniSafronov1, MelaniSafronov2, Spaide} for that of derived coisotropic structures.  We would like to thank Pavel Safronov for explaining many of the ideas discussed in this section to us.

Here's the idea.  Recall that we can identify the moduli space of singular $q$-connections on a curve $C$ as a fiber product: $\qconn_G(C, D) \iso \bun_G^\fr(C) \times_{\bun_G(C \! \bs \! D)^2} \bun_G(C \! \bs \! D)$ where the map $g \colon \bun_G^\fr(C) \to \bun_G(C \! \bs \! D)$ is given by $P \mapsto (P|_{C \! \bs \! D}, q^*P|_{C \! \bs \! D})$.  Consider the following commutative cube:

\[\xymatrix@C-30pt@R-8pt{
& \qconn_G(C, D) \ar[rr]^{f_1} \ar'[d][dd]^(.25){\mr{res}} & & \bun_G(C) \ar[dd]
\\
\bun_G(C \bs D) \ar@{<-}[ur]^{f_2} \ar[rr]^(.6){g_2} \ar[dd] & & \bun_G(C \bs D)^2 \ar@{<-}[ur]^{g_1} \ar[dd]^(.4)r
\\
& \bun_G(\bb B)^k \ar'[r][rr] & & BL^+G^{2k}
\\
BLG^k \ar[rr]\ar@{<-}[ur] & & BLG^{2k}. \ar@{<-}[ur]
}\]
Here the top and bottom faces are homotopy Cartesian squares.  What does this setup buy us?  We'll first answer informally.

\begin{claim}
First consider the bottom face of the cube.  The stack $BLG$ is 2-shifted symplectic because the Lie algebra $L\gg$ has a non-degenerate invariant pairing: the residue pairing.  The Lie subalgebra $L^+\gg$ forms part of a Manin triple $(L\gg, L^+\gg, L^-+0\gg)$ which means that $BL^+G \to BLG$ is 2-shifted Lagrangian.  Therefore the bottom face of the cube defines a 2-shifted Lagrangian intersection, which means that the pullback $\bun_G(\bb B)^k$ is 1-shifted symplectic.

Now consider the top face of the cube.  If either $C$ is an elliptic curve, or $C=\bb{CP}^1$ and we fix a framing at $\infty$, then the map $\bun_G(C \bs D) \to BLG^k$ is also 2-shifted Lagrangian.  In particular $\bun_G(C \bs D)$ is 1-shifted Poisson.  Finally, the map $\bun_G(C) \to \bun_G(C \bs D)$ is 1-shifted coisotropic, or equivalently the canonical map $\bun_G(C) \to \bun_G(C \bs D) \times_{BLG^k} BL^+G^k$ is 1-shifted Lagrangian.  That means that the top face of the cube defines a 1-shifted coisotropic intersection, which means that the pullback $\qconn_G(C,D)$ is 0-shifted Poisson.

The restriction map $\qconn_G(C,D) \to \bun_G(\bb B)^k$ is 1-shifted Lagrangian, which means that if we form the intersection with a $k$-tuple of Lagrangians in $\bun_G(\bb B)$ then we obtain a 0-shifted symplectic stack.  For example, if $\omega_i^\vee$ is a point in $\bun_G(\bb B)$ corresponding to a dominant coweight with stabilizer $\Lambda_i$ then $B \Lambda_i \to \bun_G(\bb B)$ is 1-shifted Lagrangian, so the moduli stack $\qconn_G(C,D, \omega^\vee)$ obtained by taking the derived intersection is ind 0-shifted symplectic.
\end{claim}

Now, let us make this claim more precise.  The main technical condition that makes this claim subtle comes from the fact that most of the derived stacks appearing in this cube, for instance the stack $BLG$, are not Artin.  As such we need to be careful when we try to, for instance, talk about the tangent complex to such stacks.  One can make careful statements using the formalism of ``Tate stacks'' developed by Hennion \cite{Hennion}.  We can therefore make our claim into a more formal conjecture.

\begin{conjecture}
Suppose $C$ is either an elliptic curve or $\bb{CP}^1$ with a fixed framing at $\infty$.
\begin{enumerate}
\item The stack $BLG$ is Tate 2-shifted symplectic, and both $BL^+G \to BLG$ and $\bun_G(C \bs D) \to BLG^k$ are Tate 2-shifted Lagrangian.  
\item The stack $\bun_G(C \bs D)$ is ind 1-shifted Poisson, and the map $\bun_G(C) \to \bun_G(C \bs D)$ is ind 1-shifted coisotropic witnessed by the 2-shifted Lagrangian map $BL^+G^k \to BLG^k$.
\item The Lagrangian intersection $\bun_G(\bb B)$ is Tate 1-shifted symplectic, and the map $B\Lambda_i \to \bun_G(\bb B)$ associated to the inclusion of the stabilizer of a closed point is 1-shifted Lagrangian.
\end{enumerate}
As a consequence, the moduli stack $\qconn_G(C,D)$ is ind 0-shifted Poisson and the moduli stack $\qconn_G(C,D, \omega^\vee)$ is 0-shifted symplectic.
\end{conjecture}

\begin{remark}
We should explain heuristically why the Calabi-Yau condition on $C$ is necessary.  This is a consequence of the AKSZ formalism in the case where $D$ is empty: for the mapping stack $\map(C, G/G)$ to be 0-shifted symplectic, or for the mapping stack $\map(C,BG)$ to be 1-shifted symplectic, we need $C$ to be compact and 1-oriented.  A $d$-orientation on a smooth complex variety of dimension $d$ is exactly the same as a Calabi-Yau structure.

More generally we can say the following.  Let us consider the rational case where $C = \bb{CP}^1$.  Consider the inclusion $\d r \colon \gg_- = \bb T_{\bun_G^\fr(\bb{CP}^1 \! \bs \! D)}[-1] \to r^*\bb T_{BLG^k}[-1] = \gg(\!(z)\!)^k$: a map of ind-pro Lie algebras concentrated in degree zero.  The residue pairing vanishes after pulling back along $r$ since elements of $\gg_-$ are  $\gg$-valued functions on $\bb{CP}^1$ with at least a simple pole at every puncture in $D$.  So the map $r$ is isotropic with zero isotropic structure; this structure is unique for degree reasons.  We must check that this structure is non-degenerate.  It suffices to check that the sequence
\[\bb T_{\bun_G^\fr(\bb{CP}^1 \! \bs \! D)}[-1] \to r^*\bb T_{BLG^k}[-1] \to (\bb T_{\bun_G^\fr(\bb{CP}^1 \! \bs \! D)}[-1])^\vee\]
is an exact sequence of ind-pro vector spaces, and therefore an exact sequence of quasi-coherent sheaves on the stack $\bun_G^\fr(\bb{CP}^1 \! \bs \! D)$.  To do this we identify the pair $(\gg_-, \gg(\!(z)\!)^k)$ as part of a Manin triple, where a complementary isotropic subalgebra to $\gg_-$ is given by $\gg_+ = \gg[[z]]^k$.  Using the residue pairing we can identify $\gg_+$ with $(\gg_-)^\vee$ and therefore identify our sequence with the split exact sequence
\[0 \to \gg_- \to \gg(\!(z)\!)^k \to \gg_+ \to 0.\]
\end{remark}

\begin{remark}
We will conclude this section with some comments on the multiplicative Hitchin system described above in Section \ref{Hitchin_system_section}.  In particular, the derived point of view suggests that both the multiplicative Hitchin fibers and the multiplicative Hitchin section of Remark \ref{q_opers_remark} will be -- at least generically -- canonically Lagrangian.  We can motivate this directly from the Chevalley map $\chi \colon G/G \to T/W$, by studying the non-singular example.  Generically, i.e. after restricting to the regular semisimple locus, we can identify $G^{\mr{rss}}/G$ with $(T^{\mr{reg}} \times BT)/W$, so that the fibers of the tangent space to $G/G$ are generically equivalent to $\mf t[1] \oplus \mf t$.  The directions tangent to the generic fibers of $\chi$ are concentrated in degree $-1$, meaning that the generic fibers of $\chi$ are canonically 1-shifted Lagrangian for degree reasons.  Likewise, the directions tangent to the Steinberg section $\sigma \colon T/W \to G/G$, at regular semisimple points, are concentrated in degree 0, meaning that after restriction to the regular semisimple locus is also canonically 1-shifted Lagrangian for degree reasons.  Now, if $L \to X$ is $n$-shifted Lagrangian and $M$ is $k$-oriented then there is an AKSZ $n-k$-shifted Lagrangian structure on the mapping space $\map(M,L) \to \map(M,X)$ \cite[Theorem 2.10]{Calaque}, which establishes our claim in the non-singular case. 

We should compare this discussion to Corollary \ref{isotropic_fiber_cor} in the rational case where $C = \bb{CP}^1$, where we prove that, in the rational case, the generic multiplicative Hitchin fibers are, indeed, Lagrangian.  
\end{remark}

\begin{remark} \label{Bouthier_remark}
There's yet another perspective that one might hope to pursue in order to define our symplectic structures in the language of derived symplectic geometry.  Bouthier \cite{Bouthier2} gave a description of the multiplicative Hitchin system along the following lines.  The Vinberg semigroup $V_G$ of $G$ is a family of affine schemes over $\CC^r$ whose generic fiber is isomorphic to $G$ but whose fibers lying on coordinate hyperplanes correspond to various degenerations of $G$.  Bouthier showed that the moduli space of multiplicative Higgs bundles is equivalent to the moduli space of $G$-bundles, along with a section of an associated bundle $V_G^{\omega^\vee}$ built from the Vinberg semigroup and the data of the coloured divisor $(D,\omega^\vee)$.  It's reasonable to ask whether one can construct an AKSZ shifted symplectic structure on the moduli space from this point of view, using the oriented structure on $(\bb{CP}^1,\infty)$ and the results of Ginzburg and Rozenblyum on shifted symplectic structures on spaces of sections \cite{GinzburgRozenblyum}.  
\end{remark}

\section{Twisted Gauge Theory} \label{twist_section}
Before we move on to the main mathematical content of the paper, we'll digress a little to talk about one situation where multiplicative Hitchin systems appear in gauge theory.  This story was our original motivation for studying the objects appearing in this paper, but we should emphasise that it is quite independent from the rest of the paper.  The reader who is only interested in algebraic and symplectic geometry, and not in gauge theory, can safely skip this section.

We'll describe our multiplicative Hitchin systems as the moduli spaces of solutions to the classical equations of motion in certain twisted five-dimensional supersymmetric gauge theories.  This story is distinct from the appearance of the moduli space in \cite{NekrasovPestun} as the Seiberg-Witten integral system of a 4d $\mc N=2$ theory.  Instead the moduli space will appear as the moduli space of solutions to the equations of motion in a twisted 5d $\mc N=2$ supersymmetric gauge theory, compactified on a circle.  This story will be directly analogous to the occurence of the ordinary moduli stack of Higgs bundles in a holomorphic twist of 4d $\mc N=4$ theory (see \cite{CostelloSH,ElliottYoo1} for a discussion of this story); we'll recover that example in the limit where the radius of the circle shrinks to zero.

\subsection{Background on Supersymmetry and Twisting}

We should begin by briefly recalling the idea behind twisting for supersymmetric field theories.  This idea goes back to Witten \cite{WittenTQFT}.  Supersymmetric field theories have odd symmetries coming from odd elements of the supersymmetry algebra.  Choose such an odd element $Q$ with the property that $[Q,Q]=0$.  Then $Q$ defines an odd endomorphism $\nu(Q)$ of the algebra of observables of the supersymmetric field theory with the property that $\nu(Q)^2 = 0$.  The \emph{twisted algebra of observables} associated to $Q$ is the cohomology of the operator $\nu(Q)$.  If $Q$ is chosen appropriately -- if the stress-energy tensor of the theory is $\nu(Q)$-exact -- then the $Q$-twisted field theory becomes topological.

\begin{remark}
From a modern perspective, using the language of factorization algebras, the first author and P. Safronov discussed the formalism behind topological twisting in \cite{ElliottSafronov}, and gave criteria for twisted quantum field theories to genuinely be topological.  The supersymmetry algebras and their loci of square zero elements are discussed in all cases in dimensions up to 10.  This classification is also performed in a paper of Eager, Saberi and Walcher \cite{EagerSaberiWalcher}.
\end{remark}

With the basic idea in hand, we'll focus in on the example we're interested in.  We'll be interested in twists by square-zero supercharges $Q$ that are not fully topological.  We'll begin by describing the supercharges we'll be interested in in dimensions 5 and 6.

Recall that there is an exceptional isomorphism identifying the groups $\Spin(5)$ and $\mr{USp}(4)$.  The Dirac spinor representation $S$ of $\Spin(5)$ is four dimensional: under the above exceptional isomorphism it is identified with the defining representation of $\mr{USp}(4)$.  Likewise there is an exceptional isomorphism identifying the groups $\Spin(6)$ and $\SU(4)$.  The two Weyl spinor representations $S_\pm$ are four dimensional: under the exceptional isomorphism they are identified with the defining representation of $\SU(4)$ and its dual.

\begin{definition}
The complexified $\mc N=k$ supersymmetry algebra in dimension 5 is the super Lie algebra
\[\mf A^5_k = (\sp(4;\CC) \oplus \sp(2k;\CC)_R \oplus V) \oplus \Pi(S \otimes W),\]
where $V$ is the five-dimensional defining representation of $\so(5;\CC) \iso \sp(4;\CC)$, $W$ is the $2k$-dimensional defining representation of $\sp(2k;\CC)_R$, and where there's a unique non-trivial way of defining an additional bracket $\Gamma \colon \sym^2(S \otimes W) \to V$.

Likewise, the complexified $\mc N=(k_+,k_-)$ supersymmetry algebra in dimension 6 is the super Lie algebra
\[\mf A^6_{(k_+,k_-)} = (\sl(4;\CC) \oplus \sp(2k_+;\CC)_{R} \oplus \sp(2k_-;\CC)_R \oplus V) \oplus \Pi(S_+ \otimes W_+ \oplus S_- \otimes W_-),\]
where $V$ is the six-dimensional defining representation of $\so(6;\CC) \iso \sl(4;\CC)$, $W_\pm$ is the $2k_\pm$-dimensional defining representation of $\sp(2k_\pm;\CC)_R$, and where there's a unique non-trivial way of defining an additional bracket $\Gamma_\pm \colon \sym^2(S_\pm \otimes W_\pm) \to V$.  Choosing a hyperplane in $V$ defines a super Lie algebra map $\mf A^5_{k_+ + k_-} \to \mf A^6_{(k_+,k_-)}$ which is an isomorphism on the odd summands.
\end{definition}

Let us fix some notation for the square-zero supercharges -- odd elements $Q$ where $\Gamma(Q,Q)=0$ -- that we will refer to in the discussion below.  Fix once and for all an embedding $\CC^5 \inj \CC^6$ and a symplectic embedding $W_+ \inj W$ (the twisted theories that we'll define don't depend on these choices).  Compare to the discussion in \cite[4.2.5--4.2.6]{ElliottSafronov}.
\begin{itemize}
 \item Let $Q_{\mr{min}} \in \mf A^5_1$ be any non-zero element such that $\Gamma(Q_{\mr{min}},Q_{\mr{min}})=0$ -- here ``min'' stands for ``minimal''.  All such elements have rank one, and lie in a single orbit for the action of $\so(5;\CC) \oplus \sp(2;\CC)$.  The image $\Gamma(Q_{\mr{min}}, -) \sub \CC^5$ is 3-dimensional.  Using the given embeddings we can view $Q_{\mr{min}}$ equally as an element of the larger supersymmetry algebras $\mf A^5_2, \mf A^6_{(1,0)}$ and $\mf A^6_{(1,1)}$.
 \item Let $Q_{\mr{HT}} = Q_{\mr{min}} + Q' \in \mf A^5_2$ be a non-zero element that squares to zero, and where the image $\Gamma(Q_{\mr{min}}, -) \sub \CC^5$ is 4-dimensional -- here ``HT'' stands for ``holomorphic-topological''.  All such elements have rank two and lie in a single orbit under the action of $\so(5;\CC) \oplus \sp(4;\CC)$.  Using the given embedding $\CC^5 \inj \CC^6$ we can view $Q_{\HT}$ equially as an element of $\mf A^6_{(1,1)}$.
\end{itemize}

\subsection{Description of the Twist}
Having set up our notation, we can state one the way in which the moduli space of multiplicative Higgs bundles arises from twisted supersymmetric gauge theory.  We will first state the relationship, then explain what exactly the statement is supposed to mean.

\begin{claimnum} \label{twist_claim}
If we take the twist of $\mc N=2$ super Yang-Mills theory with gauge group $G$ by $Q_\HT$, on the 5-manifold $C \times S^1 \times \RR^2$ where $C$ is a Calabi-Yau curve, with monopole surface operators placed at the points $(z_1,0), \ldots, (z_k,0)$ in $C \times S^1$ with charges $\omega^\vee_{z_1}, \ldots, \omega^\vee_{z_k}$ respectively, its moduli space of solutions to the equations of motion can be identified with the shifted cotangent space
\[\mr{EOM}_\HT(C \times S^1 \times \RR^2) \iso T^*[1]\mhiggs(C,D,\omega^\vee).\]
\end{claimnum}

\begin{remark}
The shifted cotangent space of a stack is a natural construction in the world of derived geometry.  In brief, the $k$-shifted cotangent space of a stack $\mc X$ has a canonical projection $\pi \colon T^*[k]\mc X \to \mc X$, and the fiber over a point $x \in \mc X$ is a derived affine space (i.e. a cochain complex) -- the tangent space to $\mc X$ at $x$ shifted down in cohomological degree by $k$.
\end{remark}

So, let us try to unpack the meaning of Claim \ref{twist_claim}.  Firstly, what does it mean to identify ``the moduli stack of solutions to the equations of motion'' of our twisted theory?  We have the following idea in mind (discussed in more detail, for instance, in \cite{CostelloSUSY, ElliottYoo1}).

\begin{construction} \label{twist_construction}
We'll use the Lagrangian description for the twisted field theory.  On the one hand, we can describe a twisted space of fields and twisted action functional.  This is a fairly standard construction: one factors the action functional into the sum $S_{\mr{tw}}(\phi) + Q\Lambda(\phi)$ of a ``twisted'' action functional and a $Q$-exact functional.  Note that the action of the group $\Spin(n) \times G_R$, where $G_R$ is the group of R-symmetries, is broken to a subgroup.  Only the subgroup of this product stabilizing $Q$ still acts on the twisted theory. 

Having written down the twisted action functional, we'll consider its critical locus: the moduli space of solutions to the equations of motion.  However, we can consider a finer invariant than only the ordinary critical locus: we can form the critical locus in the setting of derived geometry.  In classical field theory this idea is captured by the classical \emph{BV formalism}: for a modern formulation see e.g. \cite{CostelloBook}.  In particular, for each classical point in the critical locus, we can calculate the classical BV-BRST complex of the twisted theory -- concretely, we calculate the BV-BRST complex of the untwisted theory, and add the operator $Q$ to the differential.  The twisted BV-BRST complex models the tangent complex to the derived critical locus of the twisted theory.

So, to summarize, the BV formalism gives us a space: the critical locus of the twisted action functional, equipped with a sheaf of cochain complexes, whose fiber at a solution is the twisted BV-BRST complex around that solution.  This structure is weaker than a derived stack, but when we say a derived stack $\mc X$ ``can be identified with'' the moduli space $\EOM$ of solutions to the equations of motion, we mean that we can identify the space of $\CC$-points of $\mc X$ in the analytic topology with the underlying space of $\EOM$, and we can coherently identify the tangent complex of $\mc X$ at each $\CC$-point $x$ with the twisted BV-BRST complex at the corresponding classical solution.
\end{construction}

\begin{remark}
There's another subtlety that we glossed over in our discussion of the BV-BRST complex above: a priori when you add $Q$ to the differential of the untwisted BV-BRST complex the result is only $\ZZ/2\ZZ$-graded, since $Q$ has ghost number 0 and superdegree 1, whereas the BV-BRST differential has ghost number 1 and superdegree 0.  In order to promote it to a $\ZZ$-graded complex we need to modify the degrees of our fields using an action of $\mr U(1)$ inside the R-symmetry group, so that $Q$ has weight 1.  This is possible in all the examples we'll discuss below, and in our sketch argument we'll use this $\mr U(1)$ action to collapse the $\ZZ/2\ZZ \times \ZZ$-grading down to a single $\ZZ$-grading everywhere without further comment.
\end{remark}

Now that we know what Claim \ref{twist_claim} means, why is it true?  We'll only outline an argument here, since the whole twisted gauge theory story is somewhat orthogonal to the emphasis of the rest of this paper.  The outline we'll give goes via the minimal twist of 6d $\mc N=(1,1)$ super Yang-Mills theory.  This calculation was done independently by Dylan Butson -- a detailed version will appear in his forthcoming article \cite{Butson}.

\begin{enumerate}
 \item First, compute the twist of $\mc N=(1,0)$ super Yang-Mills theory in 6-dimensions by $Q_{\mr{min}}$.  This twisted theory is defined on a compact Calabi-Yau 3-fold $X^3$, and can be identified, in the sense of Construction \ref{twist_construction}, with the shifted cotangent space $T^*[-1]\bun_G(X^3)$.  Note: there are at least two ways of doing this calculation.  One can calculate the twist directly using a method very similar to Baulieu's calculation of the minimal twist of 10-dimensional $\mc N=1$ super Yang-Mills theory in \cite{Baulieu}.  Alternatively one can work in a version of the first order formalism where one introduces an auxiliary 2-form field: this is the approach followed by Butson. 
 
 \item One can extend this calculation to include a hypermultiplet valued in a representation $W$ of $G$.  One finds that the inclusion of the twisted hypermultiplet couples the moduli space of holomorphic $G$-bundles to sections in the associated bundle corresponding to $W$.  That is, the moduli space of solutions of the twisted theory can be identified with the shifted cotangent space $T^*[-1]\map(X^3, W/G)$, where $W/G$ is the quotient stack.  This tells us the twist of $\mc N=(1,1)$ super Yang-Mills: this is the case where $W = \gg$.
 
 \item The twist by the holomorphic-topological supercharge $Q_{\mr{HT}}$ corresponds to a deformation of this moduli space.  Specifically, let us set $X^3 = S \times C$ where $S$ is a Calabi-Yau surface and $\Sigma$ is a Calabi-Yau curve.  In this case we can identify the $Q_{\mr{min}}$-twisted $\mc N=(1,1)$ theory with
 \begin{align*}
 T^*[-1]\map(S \times \Sigma, \gg/G) &\iso T^*[-1]\map(S \times \Sigma, T[-1]BG) \\
 &\iso T^*[-1]\map(S \times T[1]\Sigma, BG).  
 \end{align*}
 The deformation to the twist by $Q_\HT$ corresponds to the Hodge deformation, deforming $T[1]\Sigma$ (the Dolbeault stack of $\Sigma$) to $\Sigma_{\mr{dR}}$ (the de Rham stack of $\Sigma$: this has the property that $G$-bundles on $\Sigma_{\mr{dR}}$ are the same as $G$-bundles on $\Sigma$ with a flat connection).  
 
 \item Now, all of these theories can be dimensionally reduced down to 5 dimensions.  Specifically, let us split $S$ as $C \times E_q$, and send the radius of one of the $S^1$ factors to zero, or equivalently degenerate the smooth elliptic curve $E_q$ to a nodal curve.  We're left with the identification, for the example we're interested in:
 \[\EOM_\HT(C \times S^1 \times \Sigma) \iso T^*[-1]\map(C \times \Sigma_{\mr{dR}}, G/G).\]
 Here we have identified $\map(E^{\mr{nod}}, BG)$ with the adjoint quotient $G/G$ \footnote{Strictly speaking this is only the semistable part of the moduli space.  We get $G/G$, for instance, by asking for only those bundles which lift to the trivial bundle on the normalization.}.  So far, tn this discussion $\Sigma$ was a compact curve.  If we want to set $\Sigma = \CC$ instead the only change is that the $-1$-shifted tangent space becomes the 1-shifted tangent complex.  There's a unique flat $G$-bundle on $\CC$, so we can identify
 \[\EOM_\HT(C \times S^1 \times \CC) \iso T^*[1]\map(C, G/G).\]
 
 \item Finally, we need to include surface operators.  In order to see how to do this, let us go back to the Lagrangian description of our classical field theory.  The inclusion of monopole operators is usually thought of as modifying the space of fields in our supersymmetric gauge theory, so that the gauge field is permitted to be singular along the prescribed surface, with specified residues.  If we track the above calculation where the fields are allowed those prescribed singularities, the result is that we should replace $\map(C \times E^{\mr{nod}}, BG)$ with the moduli space of singular multiplicative Higgs fields with local singularity data corresponding to the choice of charges of the monopole operators.  That is we can identify,
 \[\mr{EOM}_\HT(C \times S^1 \times \RR^2) = T^*[1]\mhiggs(C,D,\omega^\vee).\]
\end{enumerate}

\begin{remark}[Reduction to 4-dimensions]
If we reduce further, to four dimensions, by sending the radius of the remaining $S^1$ to zero, then we recover a more familiar twist first defined by Kapustin \cite{KapustinHolo}.  We can interpret this further degeneration as degenerating the factor $E_q$ not to a nodal, but to a cuspidal curve.  The semistable part of the stack $\map(E^{\mr{cusp}},BG)$ of $G$-bundles on a cuspidal curve is equivalent the Lie algebra adjoint quotient stack $\gg/G$, so we can identify the $Q_\HT$-twisted 4d $\mc N=4$ theory with
\begin{align*}
\EOM_\HT(C \times \Sigma) &\iso T^*[-1]\map(C \times \Sigma_{\mr{dR}}, \gg/G) \\
&\iso T^*[-1]\map(T[1]C \times \Sigma_{\mr{dR}}, BG).
\end{align*}
This is the shifted cotangent space whose base is the moduli space of $G$-bundles on $C \times \Sigma$ with a flat connection on $\Sigma$ and a Higgs field on $C$ (note that this twisted theory is defined where $C$ and $\Sigma$ are any curves, not necessarily Calabi-Yau).  This agrees with the calculation performed in \cite{ElliottYoo1}.  In particular, if we set $\Sigma = \CC$ the result is
\[\EOM_\HT(C \times \CC) \iso T^*[1]\higgs_G(C).\]
\end{remark}

\begin{remark} \label{Kapustin_twist_remark}
This supersymmetric gauge theory story should be compared to several recent twisting calculations in the literature.  Firstly, the minimal twist of 5d $\mc N=2$ gauge theory was discussed by Qiu and Zabzine \cite{QiuZabzine} on quite general contact five-manifolds.  They described the twisted equations of motion in terms of the Haydys-Witten equations.

A recent article of Costello and Yagi \cite{CostelloYagi} also described a relationship between multiplicative Higgs moduli spaces and twisted supersymmetric gauge theory in yet another context.  They calculated the twist of 6d $\mc N=(1,1)$ gauge theory, but with respect to yet another twisting supercharge $Q$ with the property that the image of $[Q,-]$ in $\CC^6$ is five-dimensional.  They then consider this theory on $\RR^2 \times C \times \Sigma$, and place it in the $\Omega$-background in the $\RR^2$ directions.  They argue that the resulting theory is a four-dimensional version Chern-Simons theory on $C \times \Sigma$ holomorphic on $C$ and topological in $\Sigma$.

The $C$-holomorphic $\Sigma$-topological 4d Chern-Simons theory on $C \times \Sigma$  was introduced for the first time that we're aware of by Nekrasov in his Ph.D. thesis (page 89 of \cite{NekrasovThesis}), where it was also speculated that such a Chern-Simons theory
would provide a new geometric origin for quantum affine algebras. The idea was recently independently discovered and investigated in detail by Costello \cite{CostelloYangian}): the moduli stack of derived classical solutions of 4d Chern-Simons theory on $C \times \Sigma$ can be identified with the mapping stack $\map(C \times \Sigma_{\mr{dR}}, BG)$. Costello and Yagi in \cite{CostelloYagi} go on to argue via a sequence of string dualities that this 4d CS theory is dual to the ADE quiver gauge theories discussed in \cite{NekrasovPestun, NekrasovPestunShatashvili}; this approach was discussed independently by both Costello \cite{CostelloStringMath} and Nekrasov \cite{NekrasovStringMath} at String Math 2017.  From another point of view, dualities of this form were also discussed by Ashwinkumar, Tan and Zhao \cite{AshwinkumarTanZhao}.

Heuristically speaking, the 4d Chern-Simons theory of \cite{NekrasovThesis} and \cite{CostelloYangian} on $C_{\mathrm{Dol}} \times \Sigma_{\mathrm{dR}}$ with $\Sigma_{\mathrm{dR}} = S^{1}_{dR} \times \mathbb{R}_{\mathrm{time}}$ quantizes, in the path formalism, the holomorphic symplectic phase space of solutions to the Bogomolny monopole equations on $C_{\mathrm{Dol}} \times S^{1}_{\mathrm{dR}}$, in the same way that ordinary 3d Cherns-Simons theory on $C_{\mathrm{dR}} \times \mathbb{R}_{\mathrm{time}}$  quantized the moduli space
of flat connections on the Riemann surface $C_{\mathrm{dR}}$ in the work of Witten \cite{witten1989}.
While Costello \cite{CostelloYangian} considers Wilson operators,  we consider Dirac singularties of the monopoles or, equivalently, t'Hooft operators.
\end{remark}

\subsection{Langlands Duality} \label{Langlands_section}
Our main motivation for describing the multiplicative Hitchin system in terms of a twist of 5d $\mc N=2$ super Yang-Mills theory is to make the first steps towards a ``multiplicative'' version of the geometric Langlands conjecture.  The description in terms of supersymmetric gauge theory allows us to use the description of the geometric Langlands conjecture in terms of S-duality for topological twists of 4d $\mc N=4$ super Yang-Mills theory given by Kapustin and Witten \cite{KapustinWitten}.  Better yet, Witten described S-duality in terms of the 6d $\mc N=(2,0)$ superconformal field theory compactified on a torus: one obtains dual theories by compactifying on the two circles in either of the two possible orders \cite{Witten6d}.  We can leverage this story in order to describe a conjectural duality for twisted 5d gauge theories. 

The A- and B-twists of 4d $\mc N=4$ super Yang-Mills can both be defined by deforming an intermediate twist sometimes referred to as the Kapustin twist \cite{KapustinHolo} (see \cite{CostelloSUSY, ElliottYoo1} for more details).  The moduli stack of solutions to the equations of motion on $C \times \RR^2$ in the Kapustin twist is the 1-shifted cotangent space of $\higgs_G(C)$: the moduli stack of $G$-Higgs bundles on $C$.  The A- and B-twists modify this in two different ways.  In the B-twist $\higgs_G(C)$ is deformed by a hyperk\"ahler rotation, and it becomes $\Flat_G(C)$: the moduli stack of flat $G$-bundles on $C$.  In the A-twist the shifted cotangent bundle $T^*[1]\higgs_G(C)$ is deformed to the de Rham stack $\higgs_G(C)_{\mr{dR}}$ -- the de Rham stack of $\mc X$ has the property that quasi-coherent sheaves on $\mc X_{\mr{dR}}$ are the same as $D$-modules on $\mc X$.  

To summarize Kapustin and Witten's argument therefore, S-duality interchanges the A- and B-twists of 4d $\mc N=4$ super Yang-Mills thory,  and provides an equivalence between their respective categories of boundary conditions on the curve $C$.  These categories of boundary conditions can be identified on the B-side as the category of coherent sheaves on $\Flat_G(C)$, and on the A-side as the category of D-modules on $\bun_G(C)$, or equivalently coherent sheaves on $\bun_G(C)_{\mr{dR}}$.

Now, we observe that this family of twists arises as the limit of a family of twisted of 5d $\mc N=2$ gauge theory compactified on a circle, where the radius of the circle shrinks to zero.  This motivates an analogous five-dimensional story, which we summarize with the following ``pseudo-conjecture'', by which we mean a physical statement whose correct mathematical formulation remains to be determined.

\begin{pseudoconj}[Multiplicative Geometric Langlands] \label{multLanglands}
Let $G$ be a Langlands self-dual group.  There is an equivalence of categories
\[\text{A-Branes}_{q^{-1}}(\mhiggs_G(C,D,\omega^\vee)) \iso \text{B-Branes}(\qconn_G(C, D, \omega^\vee))\]
where the category on the right-hand side depends on the value $q$.
\end{pseudoconj}

What does this mean, and are there situations in which we can make it precise?  We'll discuss a few examples where we can say something more concrete.  In each case, by ``B-branes'' we'll just mean the category $\coh(\qconn_G(C, D, \omega^\vee))$ of coherent sheaves.  By ``A-branes'' we'll mean some version of \emph{$q^{-1}$-difference modules} on the stack $\bun_G(C)$. 

\begin{remark}
This equivalence is supposed to interchange objects corresponding to branes of opers on the two sides, and introduce an analogue of the Feigin-Frenkel isomorphism between deformed W-algebras (see \cite{FrenkelReshetikhinSTS, STSSevostyanov}.  This isomorphism only holds for self-dual groups, which motivates the restriction to the self-dual case here.
\end{remark}

\begin{remark}
Even for the ordinary geometric Langlands conjecture there are additional complications that we aren't addressing here.  For example, the most natural categorical version of the geometric Langlands conjecture is false for all non-abelian groups.  Arinkin and Gaitsgory \cite{ArinkinGaitsgory} explained how to correct the statement to obtain a believable conjecture: one has to consider not all coherent sheaves on the B-side, but only those sheaves satisfying a ``singular support'' condition.  In \cite{ElliottYoo2} it was argued that, from the point of view of twisted $\mc N=4$ super Yang-Mills theory these singular support conditions arise when one restricts to only those boundary conditions compatible with a choice of vacuum.  The same sorts of subtleties should equally occur in the multiplicative setting.
\end{remark}

\subsubsection{The Abelian Case}
Suppose $G = \GL(1)$ (more generally we could consider a higher rank abelian gauge group).  In general for an abelian group the moduli spaces we have defined are trivial -- for instance the rational and trigonometric spaces are always discrete.  However there is one interesting non-trivial example: the elliptic case.  For simplicity let us consider the abelian situation with $D = \emptyset$: the case with no punctures.

\begin{definition}
A \emph{$q$-difference module} on a variety $X$ with automorphism $q$ is a module for the sheaf $\Delta_{q,X}$ of non-commutative rings generated by $\OO_X$ and an invertible generator $\Phi$ with the relation $\Phi \cdot f = q^*(f) \cdot \Phi$.  Write $\diff_q(X)$ for the category of $q$-difference modules on $X$.
\end{definition}

In the abelian case the space $\qconn_{\GL(1)}(E)$ is actually a stack, but one can split off the stacky part to define difference modules on it.  Indeed, for any $q$ one can write
\[\bun_{\GL(1)}(E) \iso B\GL(1) \times \ZZ \times E^\vee\]
and so
\[\qconn_{\GL(1)}(E) \iso B\GL(1) \times \ZZ \times (E^\vee \times_q \CC^\times)\]
which means one can define difference modules on these stacks associated to an automorphism of $E^\vee$ or $E^\vee \times_q \CC^\times$ respectively.

\begin{conjecture}
There is an equivalence of categories for any $q \in \bb{CP}^1$
\[\diff_q(\bun_{\GL(1)}(E)) \iso \coh(q^{-1}\conn_{\GL(1)}(E)).\]
\end{conjecture}

In this abelian case we can go even farther and make a more sensitive 2-parameter version of the conjecture.

\begin{conjecture}
There is an equivalence of categories for any $q_1, q_2 \in \bb{CP}^1$
\[\diff_{q_1}(q_2\conn_{\GL(1)}(E)) \iso \diff_{q_2^{-1}}(q_1^{-1}\conn_{GL(1)}(E)),\]
where $q_1$ is the automorphism of $E^\vee \times_{q_2} \CC^\times$ acting fiberwise over each point of $\CC^\times$.
\end{conjecture}

This conjecture should be provable using the same techniques as the ordinary geometric Langlands correspondence in the abelian case, i.e. by a (quantum) twisted Fourier-Mukai transform (as constructed by Polishchuk and Rothstein \cite{PolishchukRothstein}).

\subsubsection{The Classical Case}
Now, let us consider the limit $q \to 0$.  This will give a conjectural statement involving coherent sheaves on both sides analogous to the classical limit of the geometric Langlands conjecture as conjectured by Donagi and Pantev \cite{DonagiPantev}.  The existence of an equivalence isn't so interesting in the self-dual case (where both sides are the same), but the classical multiplicative Langlands functor should be an \emph{interesting} non-trivial equivalence.  For example we can make the following conjecture

\begin{conjecture}
Let $G$ be a Langlands self-dual group and let $E$ be an elliptic curve.  There is an automorphism of categories (for the rational, trigonometric and elliptic moduli spaces)
\[F \colon \coh(\mhiggs_G(E)) \iso \coh(\mhiggs_{G}(E))\]
so that the following square commutes:
\[\xymatrix{
\coh(\mhiggs_T(E)) \ar[r]^{\mathrm{FM}} \ar[d]_{p_*q^!} &\coh(\mhiggs_T(E)) \ar[d]^{p_*q^!} \\
\coh(\mhiggs_G(E)) \ar[r]^{F} &\coh(\mhiggs_G(E)).
}\]
Here we're using the natural morphisms $p \colon \mhiggs_B(E) \to \mhiggs_G(E)$ and $q \colon \mhiggs_B(E) \to \mhiggs_T(E)$, and $\mr{FM}$ is the Fourier-Mukai transform.
\end{conjecture}

Alternatively, we can say something about classical geometric Langlands duality in the non-simply-laced case.  Recall in the usual geometric Langlands story that non-simply-laced gauge theories in dimension 4 arise by ``folding'' the Dynkin diagram.  In other words, one identifies a non-simply laced simple Lie group as the invariants of a simply laced self-dual Lie group $\wt G$ with respect to its finite group of outer automorphisms (either $\ZZ/2\ZZ$, or $S_3$ in the case of the exceptional group $G_2$).  One obtains a 5d $\mc N=2$ gauge theory with gauge group $G$ by taking the invariants of the 6d $\mc N=(2,0)$ theory reduced on a circle, where $\mr{Out}(\wt G)$ acts simultaneously on $\wt G$ and on the circle we reduce along.

Putting this together yields the following conjecture.
\begin{conjecture}
\label{eq:classical-q-langlands}
Let $G$ and $G^\vee$ be Langlands dual simple Lie groups, and say that $G$ arose by folding the Dynkin diagram of a self-dual group $\wt G$.  Then there is an equivalence of categories
\[\coh(\mhiggs_{G^\vee}(E) \iso \coh(\mhiggs_{\wt G}(E)^{\mr{Out}(\wt G)}),\]
where $\mr{Out}(\wt G)$ acts simultaneously on $\wt G$ and on the circle $S^1_B$, under the identification $\mhiggs_{\wt G}(E) = \bun_{\wt G}(E \times S^1_B)$.  As above, this equivalence should be compatible with the Fourier-Mukai transform relating the categories for the maximal tori.
\end{conjecture}

\section{Construction of the Symplectic Structure} \label{symp_section}
\subsection{An Example with $G=\GL_2$} \label{GL2_example_section}

Let us discuss in detail the geometry of the moduli space of multiplicative Higgs bundles in some of the simplest examples, for the group $G = \GL_2$ and minimal singularity data.  The symplectic structures, and the Hitchin integrable system, which we'll analyze for the more general moduli spaces can be described very concretely in this simple situation. 

Fix $G=\GL_2$.  We'll work in the rational case, so we'll work over the curve $C = \bb{CP}^1 = \CC \cup \{\infty\}$ where $\CC$ has coordinate $z$, with fixed framing point $z_\infty = \infty$, and fixed value $g_\infty \in \GL_2$ for the framing. 

Consider the connected component of $\mhiggs^\fr_{G}(\bb{CP}^1, D, \omega^{\vee})$ consisting of those multiplicative Higgs bundles whose underlying $G$-bundle is trivializable.  Fixing the framing fixes a trivialization of the underlying $G$-bundle, which means this connected component of the moduli space can be identified with the space of $G$-valued
rational functions $g(z)$ on $\CC$ with certain singularity conditions that we can write down explicitly.

As our first explicit example we will consider the moduli space with two singularities at distinct points $z_1$ and $z_2$ in $\CC$.  Fix the corresponding coweights to be the generators $\omega^\vee_{z_1} = (1,0)$ and $\omega^\vee_{z_2} = (0,-1)$ in the defining basis of the coweight lattice of $\GL_2$.  We'll denote the zero connected component in the moduli space of multiplicative Higgs bundles by $\mc M(z_1,z_2)$.
 
The zero connected component of our moduli space  $\mhiggs_{G}^\fr(C, D, \omega^{\vee})$ can then be identified with the space of functions $g(z)$ valued in $2 \times 2$ matrices 
   \begin{equation*}
     g(z) =
     \begin{pmatrix}
       a(z) & b(z) \\
       c(z) & d(z)
     \end{pmatrix}
   \end{equation*}
   where $a(z), b(z), c(z), d(z)$ are rational functions on $\bb{CP}^1$ satisfying the following conditions. 
   \begin{enumerate}
   \item The functions $a(z), b(z), c(z), d(z)$ are regular everywhere on $\bb{CP}^1 \setminus \{z_2\}$,
     in particular they are regular at $\infty$ and $z_1$.
   \item When evaluated at the point at infinity, $g(\infty) = g_\infty$ where $g_\infty \in \GL_2$ is a representative of the conjugacy class we fixed by our choice of framing
     \begin{equation*}
       g_\infty =
       \begin{pmatrix}
         a_\infty & b_\infty \\
         c_\infty & d_\infty 
       \end{pmatrix}, \qquad a_\infty d_\infty - c_\infty b_\infty \neq 0 
     \end{equation*}
where $a_\infty, b_\infty, c_\infty, d_\infty \in \CC$  satisfy the condition
   \item
     \begin{equation*}
       \det g(z) =  \frac{ z- z_1}{ z - z_2}  \det g_\infty.
     \end{equation*}
   \end{enumerate}
 
   The conditions (1), (2) and (3) together imply that the functions $a(z), b(z), c(z), d(z)$ have the form
   \begin{equation*}
     a(z) =  \frac{a_\infty z  - a_0}{z - z_2},\quad
     b(z) = \frac{ b_\infty z  - b_0}{z - z_2},\quad
     c(z) = \frac{ c_\infty  z  - c_0}{z - z_2},\quad
     d(z) = \frac{d_\infty  z  - d_0}{z - z_2},
   \end{equation*}
for an element $(a_0, b_0, c_0, d_0) \in \CC^4$ such that
   \begin{equation*}
 (a_\infty z - a_0)( d_\infty z - d_0) -   (b_\infty z - b_0)(c_\infty z -c_0) =
     (z - z_1)(z- z_2) (a_\infty d_\infty - b_\infty c_\infty).
   \end{equation*}

   The above equation translates into the system of a linear
   equation and a quadric equation on $(a_0, b_0, c_0, d_0) \in \CC^4$, so our moduli space is encoded by the following complex affine variety:
   \begin{multline*}
\mc M(z_1,z_2)=
\big \{  (a_0, b_0, c_0, d_0) \in \CC^4 | \\
       -a_0 d_\infty - a_\infty d_0 + b_0 c_\infty + b_\infty c_0 
       = ( - z_1 - z_2) (a_\infty d_\infty - b_\infty c_\infty), \\
      a_0 d_0 - b_0 c_0 = 
       z_1 z_2 (a_\infty d_\infty - b_\infty c_\infty) \big \}
     \end{multline*}
     We conclude that, in this example, our moduli space $\mc M(z_1,z_2)$ is described by the complete intersection of a hyperplane and a  quadric in $\CC^4$, therefore by a quadric on $\CC^3$. For example,
     say $(a_\infty, b_\infty, c_\infty, d_\infty) = (1,0,0,1)$,
     and $z_1 = -m, z_2 = m$ for some $m \in \CC^\times$, then the linear equation implies that $d_0 = -a_0$,
     and the quadratic equation gives a canonical form of the smooth affine quadric
     surface 
     \begin{equation*}
       \label{eq:quadric}
         a_0^2 + b_0 c_0  = m^2 
     \end{equation*}
on  $\CC^3 = (a_0, b_0, c_0)$.

\begin{remark}
  In the limit when the singularities $z_1$ and $z_2$ collide, that is where $m \to 0$,
  the quadric becomes singular: $a_0^2 + b_0 c_0 = 0$. The resolved
  singularity on a quadric obtained by blowing up the singularity is identified with $T^{*} \bb{CP}^1$.
  The $m$-deformed quadric $a_0^2 + b_0 c_0 = m^2$ can be identified
  with the total space of an affine line bundle over $\bb{CP}^1$. This base $\bb{CP}^1$ is the orbit
  of the fundamental miniscule weight in the affine Grassmanian of the group $\GL_2$.
  We see that the moduli space $\mc M(z_1,z_2)$ of multiplicative Higgs bundles in the case of two miniscule co-weight singularities for $\GL_2$ is an affine line bundle over the flag variety $\bb{CP}^1$, 
where, locally, the 1-dimensional base arises 
from the insertion of one singularity,
and 1-dimensional fiber comes from the insertion of the other.
\end{remark}

\begin{remark}
  The canonical coordinates $a \in \CC, b \in \CC^{\times}$ on the quadric (\ref{eq:quadric}) are given by
  \begin{equation*}
    a_0 = a, \qquad b_0 = b(m - a), \qquad c_0 = b^{-1}(m+a)
  \end{equation*}
  with Poisson brackets
  \begin{equation*}
    {a,b} = b
  \end{equation*}
  and symplectic form $\d a \wedge \frac{\d b} b$.  
\end{remark}

\begin{remark}
We can likewise calculate the multiplicative Hitchin section as in Remark \ref{q_opers_remark} in this example.  The Steinberg section for the group $G = \GL_2$ sends a pair of eigenvalues $(s,t) \in T/W$ to the element
\[\begin{pmatrix}-s-t & st \\ 1 & 0\end{pmatrix} \in G.\]
In order to describe the multiplicative Hitchin section, let us use a framing that lands in the Steinberg section, say $(a_\infty, b_\infty,c_\infty, d_\infty) = (0,-1,1,0)$.  Then the moduli space $\mc M(z_1,z_2)$ as above with this framing can be identified with the smooth affine quadric surface 
    \begin{equation*}
        a_0d_0 + b_0^2  = m^2 .
    \end{equation*}
The multiplicative Hitchin section lies within the locus where $d_0=0$, i.e. to the pair of lines $b_0 = \pm m$ inside our quadric surface.  The fixed framing at infinity picks out the line $b_0 = - m$, or the locus of matrices of the form 
\[g(z) = \begin{pmatrix} \frac{a_0}{z-m} & - \frac{z+m}{z-m} \\ 1 & 0\end{pmatrix}.\]
In Poisson coordinates $a,b$ as above, the multiplicative Hitchin section is equivalent to the Lagrangian line $a = m$: a section of the projection map onto the space $\CC^\times$ on which $b$ is a coordinate.    

For a more general singularity datum, but for the group $G=\GL_2$ the multiplicative Hitchin section admits a similar description, where now the datum $(D,\omega^\vee)$ is encoded by a rational function $p_1(z)/p_2(z)$ where $p_1$ and $p_2$ are monic polynomials of the same degree, say $d$.  The multiplicative Hitchin section then can be described as the locus of all matrices of the form
\[g(z) = \begin{pmatrix} \frac{q(z)}{p_2(z)} & - \frac{p_1(z)}{p_2(z)} \\ 1 & 0\end{pmatrix},\]
where $q(z)$ is a polynomial of degree less than $d$ (so the section is $d$-dimensional).  The sections which are in fact $\SL_2$-valued are those with $p_1 = p_2$.
\end{remark}

More general explicit cases of $\mhiggs^\fr_G(C)$ with parametrization by Darboux coordinates for $C = \bb{CP}^1$ and $G=\GL_n$ case are discussed in \cite{FrassekPestun}.

\subsection{The Rational Poisson Lie Group} \label{Poisson_Lie_section}
The symplectic structures we'll build on the rational moduli spaces of multiplicative Higgs bundles ultimately arise by viewing the moduli spaces, ranging over colored divisors $(D,\omega^\vee)$, as symplectic leaves in an infinite-dimensional Poisson Lie group. This connection should be compared to the connection between the elliptic moduli space and the elliptic quantum group in the work of Hurtubise and Markman \cite[Theorem 9.1]{HurtubiseMarkman}.  

Consider the infinite-type moduli space $\mhiggs^{\text{fr,sing}}(\bb{CP}^1)$ of multiplicative Higgs bundles with arbitrary singularities.  There is a map
\[r_\infty \colon \mhiggs^{\text{fr,sing}}(\bb{CP}^1) \to G_1[[z^{-1}]],\]
defined by restricting a multiplicative Higgs bundle to a formal neighborhood of infinity.  In particular we can restrict $r_\infty$ to any of the finite-dimensional subspaces $\mhiggs^\fr_G(\bb{CP}^1,D,\omega^\vee)$ (the symplectic leaves).  Here the notation $G_1[[z^{-1}]]$ denotes Taylor series in $G$ with constant term 1.  Recall that $G_1[[z^{-1}]]$ is a Poisson Lie group, with Poisson structure defined by the Manin triple $(G(\!(z)\!), G[z], G_1[[z^{-1}]])$ (\cite{DrinfeldICM}, see also \cite{Shapiro, Williams}).

The Poisson Lie group $G_1[[z^{-1}]]$ has the structure of an ind-scheme using the filtration where $G_1[[z^{-1}]]_n$ consists of $G$-valued polynomials in $z^{-1}$ with identity constant term and where the matrix elements in a fixed faithful representation have degree at most $n$.  When we talk about the Poisson algebra $\OO(G_1[[z^{-1}]])$ of regular functions on $G_1[[z^{-1}]]$, we are therefore referring to the colimit over $n \in \ZZ_{\ge 0}$ of the finitely generated algebras $\OO(G_1[[z^{-1}]]_n)$.  Due to the condition on the constant term the group $G_1[[z^{-1}]]$ is nilpotent, which means we can identify its algebra of functions with the algebra of functions on the Lie algebra.  That is:
\begin{align*}
\OO(G_1[[z^{-1}]]) &= \underset n \colim \OO(G_1[[z^{-1}]]_n) \\
&\iso \underset n \colim \sym(\gg_1[[z^{-1}]]_n^*).
\end{align*}
The Poisson bracket is compatible with this filtration, in the sense that each filtered piece $\OO(G_1[[z^{-1}]]_n)$ is Poisson, the structure maps $\OO(G_1[[z^{-1}]]_m) \to \OO(G_1[[z^{-1}]]_n)$ are Poisson maps, and so is the multiplication map
\[m \colon \OO(G_1[[z^{-1}]]_m) \otimes \OO(G_1[[z^{-1}]]_n) \to \OO(G_1[[z^{-1}]]_{m+n})\] 
for each $m$ and $n$.

We can describe the Poisson structure on $G_1[[z^{-1}]]$ concretely, using the rational classical $r$-matrix (see for instance \cite{GelfandCherednik,DrinfeldICM}, or \cite{Shapiro}).  Define
\[r = \frac \Omega{z-w} = \Omega \sum_{i=0}^\infty z^{-i-1}w^i \in \gg^{\otimes 2}[w][[z^{-1}]],\] 
where $\Omega$ is the quadratic Casimir element in $\gg^{\otimes 2}$ associated to a non-degenerate invariant pairing $\kappa$.  If $f$ is a function in $\OO(G_1[[z^{-1}]])$, let $\nabla_L(f)$ and $\nabla_R(f)$ be the $\gg[z]$-valued functions on $G_1[[z^{-1}]]$ obtained from the left- and right-invariant vector fields respectively on the group.  In these terms, the Poisson bracket can be calculated as 
\[\{f_1, f_2\}(g) = \langle r, \nabla_L(f_1)(g) \otimes \nabla_L(f_2)(g) - \nabla_R(f_1)(g) \otimes \nabla_R(f_2)(g)\rangle,\]
where here $\langle - , - \rangle$ denotes the residue pairing on $\gg(\!(z^{-1})\!)^{\otimes 2}$.  We'll sometimes call this the \emph{Sklyanin Poisson structure} on $G_1[[z^{-1}]]$ \cite{Sklyanin}.

\begin{example} \label{ev_function_example}
Let's investigate the Poisson bracket of some specific functions.  Let $u$ and $v$ be points in $\CC$, and let $\phi$ and $\psi \colon G \to \CC$ be algebraic functions.  Associated to these pairs we can define \emph{evaluation functions} $\phi_u$ and $\psi_v$ on a subgroup of $G_1[[z^{-1}]]$ by setting
\begin{align*}
\phi_v \colon \{g \in G_1[[z^{-1}]] \colon g(v) \text{ converges}\} &\to \CC \\
g &\mapsto \phi(g(v)).
\end{align*}
The left and right derivatives $\nabla_L(\phi_v)(g)$ and $\nabla_R(\phi_v)(g)$ of the evaluation function at $g$ can be paired with a vector field $X$ by 
\begin{equation}
\label{eq:left-right-gradient}
  (X, \nabla_{L}(\phi_u)) := \left. \frac{\d}{\d t} \phi(e^{X(u)t} g(u))\right|_{t=0},
  \qquad (X, \nabla_{R}(\phi_u)) \rangle := \left. \frac{\d}{\d t} \phi(g(u) e^{X(u)t})\right|_{t=0}.
\end{equation}

If $g$ is a point in $G_1[[z^{-1}]]$ where $g(v)$ and $g(u)$ both converge, we can calculate the value of the Poisson bracket $\{\phi_u, \psi_v\}$ at $g$, using a choice of basis $\{X^i\}$ for the Lie algebra $\gg$.  
\begin{align*}
\{\phi_u, \psi_v\}(g) &= \left\langle \frac \Omega {z-w}, \nabla_L(\phi_u)(g)(z) \otimes \nabla_L(\psi_v)(g)(w) - \nabla_R(\phi_u)(g)(z) \otimes \nabla_R(\psi_v)(g)(w)\right\rangle \\
 &= \sum_{i,j=1}^{\dim \gg} \sum_{k=0}^\infty \Big\langle \Omega_{ij}X^i \otimes X^j z^{-k-1}w^k, \nabla_L(\phi_u)(g)(z) \otimes \nabla_L(\psi_v)(g)(w) \\
 &\qquad\qquad\qquad\qquad\qquad\qquad\qquad\qquad\qquad\qquad - \nabla_R(\phi_u)(g)(z) \otimes \nabla_R(\psi_v)(g)(w)\Big\rangle \\
 &= \sum_{i,j,k} \Big\langle X^j,\frac{\d}{\d t} \left.\phi \left(e^{\Omega_{ij}u^{-k-1}w^kX^it}g(u)\right)\right|_{t=0} \otimes \nabla_L(\psi_v)(g)(w) \\
 &\qquad\qquad\qquad\qquad\qquad\qquad\qquad -\left.\phi \left(g(u) e^{\Omega_{ij}u^{-k-1}w^kX^it}\right)\right|_{t=0} \otimes \nabla_R(\psi_v)(g)(w)\Big\rangle \\
&= \sum_{i,j} \frac {\Omega_{ij}}{u-v} \left( \left.\frac{\d}{\d t} \phi(e^{X^i t}g(u))\right|_{t=0} \left.\frac{\d}{\d s} \psi(e^{X^i s}g(v))\right|_{s=0} - \left.\frac{\d}{\d t} \phi(g(u)e^{X^i t})\right|_{t=0} \left.\frac{\d}{\d s} \psi(g(v)e^{X^i s})\right|_{s=0} \right)\\
&= \frac 1 {u-v} (\kappa(\nabla_L^0(\phi)(g(u)), \nabla_L^0(\psi)(g(v))) - \kappa(\nabla_R^0(\phi)(g(u)), \nabla_R^0(\psi)(g(v)))),
\end{align*}
where now $\nabla_L^0$ and $\nabla_R^0$ denote the left and right derivatives on the finite-dimensional group $G$, rather than the infinite-dimensional group $G_1[[z^{-1}]]$, and where we've used the canonical identification between $\gg$ and $\gg^*$.
\end{example}

\begin{lemma} \label{evaluation_vector_fields_generate_lemma}
Let $g$ be an element of the subgroup $G_1(\bb{CP}^1) \sub G_1[[z^{-1}]]$ consisting of rational functions $\bb{CP}^1 \to G$ framed at $\infty$.  The cotangent space $T^*_g G_1(\bb{CP}^1)$ is generated by the set of derivatives of evaluation vector fields $\{\d \phi_u(g)\}$, where $u$ ranges over points in $\CC$ where $g$ is regular, and $\phi$ ranges over algebraic functions $G \to \CC$.
\end{lemma}

\begin{proof}
Let $\gg_0(\bb{CP}^1)$ be the Lie algebra of $G_1(\bb{CP}^1)$, consisting of $\gg$-valued rational functions $f$ on $\bb{CP}^1$ with $f(\infty) = 0$. We can first identify the cotangent space $T^*_g G_1(\bb{CP}^1)$ with the dual Lie algebra $\gg_0(\bb{CP}^1)^* \sub \gg_0[[z^{-1}]]^*$ using the left-invariant vector field; the image of $\d \phi_u(g)$ under this identification is exactly $\nabla_L(\phi_u)$.  We realize the infinite-dimensional Lie algebra as the colimit over $n$ of finite-dimensional Lie algebras $\gg_n^*$ as in Section \ref{Poisson_Lie_section}. 

We can pair our element $\nabla_L(\phi_u)$ with rational functions $f(z) \in \gg_0(\bb{CP}^1)$, the result is
\[(\nabla_{L} \phi_u, f(z)) = \frac{\d}{\d t} \phi(e^{f(u) t} g(u))|_{t=0}.\]
In particular, for any choice of rational function $f$ we can find some $u$ and $\phi$ such that this pairing is non-vanishing.  Since the pairing between $\gg_0(\bb{CP}^1)$ and its dual is non-degenerate, the set $\{\nabla_L(\phi_u)\}$ of covectors must span the entire cotangent space of the group $G_1(\bb{CP}^1)$.
\end{proof}

We'll conclude this subsection with some observations about vector fields generated by adjoint invariant evaluation functions $\phi_u$ that will be useful for understanding the integrable system structure on multiplicative Higgs moduli spaces. 
\begin{lemma}
 If $\phi$ is an adjoint invariant function on the group $G$, then
$\nabla_{L} \phi_u  = \nabla_{R} \phi_u$. 
\end{lemma}

\begin{proof}
  If $\phi$ is adjoint invariant then
  \[
    \frac{\d}{\d t} \phi(e^{Xt} g e^{-Xt}) = 0
  \]
  and consequently $\nabla_{L} \phi_u - \nabla_{R} \phi_u  = 0$. 
\end{proof}

\begin{corollary}\label{cor:poisson-commuting}
  If $\phi$ and $\psi$ are adjoint invariant functions then the evaluation functions $\phi_u$ and $\psi_v$ Poisson commute.
\end{corollary}

In particular this tells us the following.

\begin{corollary}\label{isotropic_fiber_cor}
The intersection of the image of $r_\infty \colon \mhiggs^\fr_G(\bb{CP}^1, D, \omega^\vee) \to G_1[[z^{-1}]]$ with a fiber of the Chevalley map $\chi \colon G_1[[z^{-1}]] \to T[[z^{-1}]]/W$ is isotropic for the Sklyanin Poisson structure.
\end{corollary}

\subsection{Deformations of Multiplicative Higgs Bundles} \label{def_section}
We'll begin this section by considering the tangent complex to the moduli space of $q$-connections.  For the arguments in this article we'll only need to carefully consider the case $q=\id$ of multiplicative Higgs bundles, but we'll include some remarks regarding the more general case.  In this case the calculation was performed by Bottacin \cite{Bottacin}, see also \cite[Section 4]{HurtubiseMarkman}. We begin with an illustrative description of the deformation complex, then use a more careful description gollowing Hurtubise and Markman.  Fix a multiplicative Higgs bundle $(P,g)$ on $C$.  

Let $G = \GL_n$, so we can identify the adjoint bundle $\gg_P$ as the bundle of endomorphisms of the rank $n$ vector bundle $V$, obtained from $P$ as the associated bundle for the defining representation.  We consider the sheaf of cochain complexes on $C$
\[\mc F_{(P,g)} = (\gg_P[1] \overset {A_g} {\to} \gg_P(D))\]
in degrees $-1$ and 0 with differential given by the difference $A_g = L_g - R_g$ of the left- and right-multiplication maps by the singular section $g$ of $P$.  We can alternatively phrase this, in a way that avoids analyzing the extension across the singularities, as in \cite[Section 4]{HurtubiseMarkman}.  

\begin{definition} \label{ad_g_definition}
Define $\ad(g)$ to be the vector bundle
\[\ad(g) = (\gg_P \oplus \gg_P)/\{(X, -g X g^{-1}): X \in \gg_P\}.\]
Then we can write $\mc F$ as the sheaf of complexes
\[\mc F_{(P,g)} = (\gg_P[1] \overset {A_g} {\to} \ad(g))\]
where now $A_g$ is just the map $X \mapsto [(X,-X)]$.  We will often denote a section of $\ad(g)$ by an equivalence class of pairs $[(X^L,X^R)]$.
\end{definition}

To connect this to the paragraph above, note that there is a surjective bundle map $\gg_P \oplus \gg_P \to \gg_P(D)$ sending $(X,Y)$ to $L_g(X) + R_g(Y)$, whose kernel is exactly $\{(X, -g X g^{-1}): X \in \gg_P\}$, and so it induces an isomorphism $\ad(g) \to \gg_P(D)$.  This isomorphism intertwines the two descriptions of the bundle map $A_g$.

\begin{remark}
In \cite{HurtubiseMarkman}, the map $A_g$ is denoted by $\Ad_g$.  We are choosing to use slightly difference notation in order to avoid confusion with the adjoint action by $g$.  The map $A_g$ heuristically sends a section $X$ of $\gg_P$ to a section $gX - Xg$ of $\gg_P(D)$, or equivalently, after an overall right multiplication by $g^{-1}$, to $gXg^{-1} - X$.  Using the notation $\Ad_g$ risks confusion with the map $X \mapsto gXg^{-1}$.
\end{remark}

\begin{remark}
If one introduces a framing at a point $c \in C$ then we must correspondingly twist the complex $\mc F$ above by the line bundle $\OO(-c)$ on $C$, i.e. we restrict to deformations that preserve the framing and therefore are zero at the point $c$.  So in that case we define
\[\mc F^\fr_{(P,g)} = (\gg_P[1] \overset {A_g} {\to} \ad(g)) \otimes \OO(-c).\]
\end{remark}

\begin{remark}
For more general $q$ we should modify this description by replacing $g$ by a $q$-connection.  Note that one can still define the ($q$-twisted) adjoint action $X \mapsto g X g^{-1}$ using a $q$-connection, and so we can still define the complex
\[\mc F_{(P,g)} = (\gg_P[1] \overset {A_g} {\to} \ad(g))\]
just as in the untwisted case.
\end{remark}

This complex defines the deformation theory of the moduli space of multiplicative Higgs bundles.

\begin{prop}[{\cite[Proposition 3.1.3]{Bottacin}}]
The tangent space to $\mhiggs_G(C, D, \omega^\vee)$ at the point $(P,g)$ is quasi-isomorphic to the hypercohomology $\bb H^0(C; \mc F_{(P,g)})$ of the sheaf $\mc F$.
\end{prop}

\begin{remark}
The remaining hypercohomology of the sheaf $\mc F_{(P,g)}$ generically has dimension $\dim \mf z_{\gg}$ (or 0 if we fix a framing at $c \in C$) in degree $-1$, and dimension $\mr{genus}(C) \cdot \dim \mf z_{\gg}$ in degree $1$.  However the moduli space $\mhiggs_G(C, D,\omega^\vee)$ is in fact a smooth algebraic variety.  This follows from a result of Hurtubise and Markman \cite[Theorem 4.13]{HurtubiseMarkman}, noting that their argument does not rely on the curve $C$ being of genus 1.
\end{remark}

\begin{corollary} \label{dim_of_moduli_space_cor}
In the rational case, the moduli space $\mhiggs^\fr_G(\bb{CP}^1, D, \omega^\vee)$ has dimension 
\[2 \sum_{z_i \in D} \langle \rho, \omega^\vee_{z_i} \rangle.\]
\end{corollary}

\begin{proof}
One can use the same argument as \cite{HurtubiseMarkman} (see also \cite[Proposition 5.6]{CharbonneauHurtubise}), with the additional observation that tensoring by the line bundle $\OO(-c)$ (where here the framing point $c$ is the point at infinity) kills the outer cohomology groups ($\bb H^{-1}$ and $\bb H^1$ with our degree conventions, which differ from the conventions of loc. cit. by one).  Indeed, $\bb H^{-1}$ consists of sections of $\gg_P$ that are annihilated by $A_g$ (given for generic $g$ by constant sections valued in $\mf z_{\gg}$) while vanish at $\infty$, which are necessarily 0.  Likewise we can use the equivalence between the sheaf $\mc F_{(P,g)}$ and its Serre dual to see that $\bb H^1$ also vanishes.  Finally the Euler characteristic of the two step complex is unchanged by tensoring by $\OO(-c)$. 
\end{proof}

In the rational case with our framing at infinity, when additionally $P$ is the trivial $G$-bundle, we can actually replace hypercohomology with ordinary cohomology.  Indeed, consider the spectral sequence on hypercohomology for our two step complex $(\gg_P \to \ad(g))(-c)$.  The $E_2$ page of the spectral sequence has the form
\[H^\bullet(\bb{CP}^1; \gg_P(-c))[1] \to H^\bullet(\bb{CP}^1; \ad(g)(-c)),\]
with differential induced from the map $A_g$.  Since the cohomology of the two steps are concentrated in degrees 0 and 1, there are no further differentials in the spectral sequence.  If $P$ is trivial, the first term, $H^\bullet(\bb{CP}^1; \gg_P(-c))$, vanishes identically, so in particular we have the following.

\begin{corollary}
The tangent space to $\mhiggs^\fr_G(\bb{CP}^1, D, \omega^\vee)$ at the point $(P,g)$ with $P$ trivial is quasi-isomorphic to the space of global sections $H^0(\bb{CP}^1; \ad(g)(-c))$ of the sheaf $\ad(g)(-c)$.  In other words, to pairs $(X^L, X^R)$ of $\gg$-valued functions on $\bb{CP}^1$ vanishing at infinity, modulo the equivalence relation (as in Definition \ref{ad_g_definition})
\begin{equation}\label{eq:equivalence} 
(X^L, X^R) \sim (X^L + X, X^R - gXg^{-1}).
\end{equation}
\end{corollary}

\begin{remark} \label{atlas_remark}
In what follows we will use a standard atlas for $\bb{CP}^1$ associated to our choice $D$ of divisor.  We cover $\bb{CP}^1$ by open subsets $\{U_0, U_1, \ldots, U_k, U_\infty\}$, where for $i=1, \ldots, k$ $U_i$ is a small open disk around $z_i \in D$, $U_\infty$ is a small open disk around $\infty$, and $U_0$ is 
\[U_0 = \bb{CP}^1 \bs (V_1 \cup V_2 \cup \cdots \cup V_k \cup V_\infty)\]
for $V_i \sub U_i$ closed sub-disks.
\end{remark}

\subsection{Multiplicative Higgs bundles as Symplectic Leaves} \label{symplectic_leaf_section}
We'll now investigate the restriction of the Sklyanin Poisson structure on $G_1[[z^{-1}]]$ to our finite-dimensional moduli spaces of multiplicative Higgs bundles, $\mhiggs^\fr_G(\bb{CP}^1, D, \omega^\vee)$.  Let $g$ be a point in $\mhiggs^\fr_G(\bb{CP}^1, D, \omega^\vee)$; we'll abuse notation and also write $g$ for its image under the restriction map $r_\infty$ to a formal neighbourhood of infinity.

Choose $u \in \CC \bs D$, and let $\phi_u$ be an evaluation function as in Example \ref{ev_function_example}.  We can identify the Hamiltonian vector field associated to $\phi_u$ using the Sklyanin Poisson structure.  In terms of left and right components, it is represented as
\begin{equation}
\label{eq:Xphiz}
X_{\phi_u}(w) =  \frac{1}{w - z} (\nabla_{L} \phi_u,  - \nabla_{R} \phi_u),
\end{equation}
where we've identified $\gg$ and $\gg^*$ using the Killing form $\kappa$.  We must first check that this gives a well-defined vector field on the moduli space $\mhiggs^\fr_G(\bb{CP}^1, D, \omega^\vee)$, not just on the Poisson Lie group.

\begin{lemma}\label{phitoX}
  If $\phi$ is an algebraic function on $G$, and $u$ is a point in $\CC \bs D$, then the Hamiltonian vector field $X_{\phi_u}$ on $\mhiggs^\fr_G(\bb{CP}^1, D, \omega^\vee)$ belongs to the tangent space to $\mhiggs^\fr_G(\bb{CP}^1, D, \omega^\vee)$.  
\end{lemma}

\begin{proof}
  Let us pass to a representative for the Hamiltonian vector field (\ref{eq:Xphiz}) of the form $(X^{L}, 0)$ using the equivalence relation (\ref{eq:equivalence}). We get
  \begin{equation}
\label{eq:xlw0}
(X^L_{\phi_u} (z) , 0) =  \frac{1}{z - u} (\nabla_{L} \phi_u  - \Ad_{g(z)} (\nabla_{R} \phi_{u}), 0 )
  \end{equation}
  We need to check two points to ensure that $(X^L_{\phi_z} (w) , 0)$ is a deformation
  of $g$ in tangent direction to $\mhiggs^{\fr}_{D}$
  \begin{enumerate}
  \item in each chart in $\bb{CP}^1$ there exists an equivalence frame in which $X(z) \sim (\tilde X_{\phi_{u}}^{L}(z), \tilde X_{\phi_{u}}^{R}(z))$  are regular sections (as functions of $z$).
  \item $X(z) \to 0$ as $z \to \infty$.
  \end{enumerate}

  To check (1) we need to look on the potential singularities as $z \to u$ or as $z \to z_i$.  There is no singularity as $z \to u$ since the representative (\ref{eq:xlw0}) can be rewritten as
  \begin{equation}
\label{eq:sklyanin-left}
(X^L_{\phi_u}(z) , 0)  =     \frac{1}{z - u} (\Ad_{g(u)}(\nabla_{R} \phi_u)  - \Ad_{g(z)} (\nabla_{R} \phi_{u}),0)
  \end{equation}
  and since $g(z)$ is regular in the limit $z \to u$ the ratio is also regular as $z \to u$.
  
  There is also no singularity as $z \to z_i$ in the original equivalence frame (\ref{eq:Xphiz}) (or equivalently, the singularity of (\ref{eq:xlw0}) near $z \to z_i$  is in the image under $\Ad_{g(z)}$ of the space of regular functions by (\ref{eq:Xphiz})).
 
 Point (2) is clear since $g(z)$ is regular at $z = \infty$. 
\end{proof}

Write $\pi \colon T^*_g G_1[[z^{-1}]] \to T_g G_1[[z^{-1}]]$ for the map induced by the Sklyanin Poisson structure.  We've just shown that the image under $\pi$ of the space of evaluation covectors -- derivatives of evaluation functions -- is contained in the tangent space $T_{g}\mhiggs^\fr_G(\bb{CP}^1, D, \omega^\vee)$.  In particular, by Lemma \ref{evaluation_vector_fields_generate_lemma}, the image of the cotangent space $T^*_g G_1(\bb{CP}^1)$ to the subgroup of $G$-valued framed rational functions on $\bb{CP}^1$ is contained in $T_{g}\mhiggs^\fr_G(\bb{CP}^1, D, \omega^\vee)$.  We'll now compare the Sklyanin Poisson structure with the restriction to $\mr{im}(\pi)$ of a natural formula for a symplectic structure on $T_{g}\mhiggs^\fr_G(\bb{CP}^1, D, \omega^\vee)$ given by a sum over residues.

 Recall that a symplectic structure $\Omega$ on a subspace $S \sub M$ of a Poisson manifold is compatible with the Poisson structure if, for any two Hamiltonian functions $\phi, \psi$ on $M$, we have an equality
 \begin{equation}
\label{eq:compatibility}
   \Omega(X_{\psi}, X_{\phi}) = \pi (d\phi, d \psi),
 \end{equation}
 where $X_{\psi}, X_{\phi}$ are Hamiltonian vector fields in $T_{S}$ generated by $\phi$ and  $\psi$, i.e.
 \begin{equation}
   X_{\psi} = \pi d \psi, \qquad X_{\phi} = \pi d \phi.
 \end{equation}
So, in other words, we want to find an explicit symplectic form $\Omega: T_{S} \to T_{S}^{*} $ such that for any Hamiltonian function $\phi$ on $M$ and any vector field $X$ in $T_{S}$ we have
 \begin{equation}
   \Omega(X, \pi d \phi) = d_{X} \phi 
 \end{equation}

Let $X$ and $X'$ be vectors in the tangent space to our moduli space $\mhiggs^\fr_G(\bb{CP}^1, D, \omega^\vee)$ at a point $g$. Consider the bilinear form $\Omega(X, X')$ defined by the sum of residues over the set $\tilde D := D \cup \{\infty\}$,
 \begin{equation}
\label{eq:Omega}
   \Omega(X, X')  = \frac{1}{2 \pi \imath} \sum_{z_i \in \tilde D} \oint_{\partial U_i}  \d w (
\kappa(X^{L}_{i}, X^{L'}_{0})  - \kappa( X^{R}_{i},  X^{R'}_{0})),
\end{equation}
where we have choosen equivalence frames in the open subsets $U_i$ and $U_0$ defined in Remark \ref{atlas_remark} such that $X_i^{L}, X_i^{R}$ are regular in the open subset $U_i$.

\begin{prop}
The bilinear form $\Omega(X, X')$ is well-defined, i.e. independent of the choice of representative $(X^L_i, X^R_i)$ in each chart $U_i$, and anti-symmetric.
\end{prop}

\begin{proof}
First we consider the change of equivalence frame in the second argument.  So let $X'$ and $\tilde X'$ be two representatives of the same equivalence class differing by $X_0'$ in the patch $U_0$. The section $X_{0}'$ is regular in $U_{0}$. 
\begin{align*}
   \Omega(X, \tilde X') - \Omega( X,X')  &= \frac{1}{2 \pi \imath} \sum_{z_i \in \tilde D} \oint_{\partial U_i}  \d z (
   \kappa(  X^{L}_{i} ,  X^{L'}_{0} + \Ad_{g(z)} X_{0}' )  - \kappa( X^{R}_{i} ,  X^{R'}_{0} - X_{0}') ) \\
   &\qquad - \frac{1}{2 \pi \imath} \sum_{z_i \in \tilde D} \oint_{\partial U_i}  \d z (
   \kappa(  X^{L}_{i} ,  X^{L'}_{0} )  - \kappa( X^{R}_{i} ,  X^{R'}_{0} ) ) \\  
   &= \frac{1}{2 \pi \imath} \sum_{z_i \in \tilde D} \oint_{\partial U_i}  \d z (
\kappa(   X_{i}^{L},   \Ad_{g(z)}X^{'}_{0} )  + \kappa(  X_{i}^{R},  X^{'}_{0} ) ) \\ 
&= \frac{1}{2 \pi \imath} \sum_{z_i \in \tilde D} \oint_{\partial U_i}  \d z (\kappa(   \Ad_{g(z)^{-1}}  X_{i}^{L} + X_{i}^{R}, X_{0}' ) ) \\ 
&= \frac{1}{2 \pi \imath} \sum_{z_i \in \tilde D} \oint_{\partial U_i}  \d z (\kappa(   \Ad_{g(z)^{-1}}  X_{0}^{L} + X_{0}^{R}, X_{0}' ) ) \\ 
&= - \frac{1}{2 \pi \imath} \oint_{\partial U_0} \d z (\kappa(   \Ad_{g(z)^{-1}}  X_{0}^{L} + X_{0}^{R}, X_{0}' ) ) \\ 
&= 0 
\end{align*}

We used that $\Ad_{g(z)^{-1}}  X_{i}^{L} + X_{i}^{R}$ (the right variation of $g$) is invariant across the patches, and the fact that the integrand is regular on $U_0$ in the final equality.

Now, since we've proven invariance under the change of equivalence frame in the second argument, let's set $X_{0}^{R'} = 0$ by a suitable change of equivalence frame.  In this equivalence frame we have $X_{0}^{L'} = X_i^{L'} + \Ad_{g} X_i^{R'}$
on the overlap between $U_i$ and $U_0$, hence the original formula for the bilinear pairing becomes 
\[
\Omega(X, X') =   \frac{1}{2 \pi \imath} \sum_{z_i \in \tilde D} \oint_{\partial U_i}  \d z (
\langle  X^{L}_{i}, X_i^{L'} + \Ad_{g} X_i^{R'} \rangle 
\]
Since $\langle X_i^{L}, X_{i}^{L'} \rangle $ is regular on $U_i$ this term vanishes and we get
\begin{equation}
\label{eq:drop-right}
  \Omega(X, X') =   \frac{1}{2 \pi \imath} \sum_{z_i \in \tilde D} \oint_{\partial U_i}  \d z (
\langle  X^{L}_{i},  \Ad_{g} X_i^{R'} \rangle  =  \frac{1}{2 \pi \imath} \sum_{z_i \in \tilde D} \oint_{\partial U_i}  \d z (
\langle  \Ad_{g^{-1}} X^{L}_{i},   X_i^{R'} \rangle 
\end{equation}

On the other hand, we can set $X_0^{L'} = 0$ in the original formula by a suitable change
of framing in the second argument, and in this equivalence frame we have $X_0^{R'} =  X_{i}^{R'}  + \Ad_{g^{-1}} X_{i}^{L'}$ on the
overlaps between $U_i$ and $U_0$, and consequently, the original formula is transformed into the form
\begin{equation}
\label{eq:drop-left}
  \Omega(X, X') =   -\frac{1}{2 \pi \imath} \sum_{z_i \in \tilde D} \oint_{\partial U_i}  \d w (
\langle  X^{R}_{i},  \Ad_{g^{-1}} X_i^{L'} \rangle 
\end{equation}
Comparing the expression (\ref{eq:drop-left}) with (\ref{eq:drop-right}) we have demonstrated that $\Omega(X,X')$ is anti-symmetric.

To conclude, we've proven that $\Omega(X,X')$ is invariant under a change of equivalence frame in the second argument, and so by the anti-symmetry we obtain that $\Omega(X, X')$ is also invariant under a change of equivalence frame in the first argument.
\end{proof}

\begin{remark}
Note that we don't yet assume that $\Omega$ is non-degenerate: we will demonstrate this as part of Theorem \ref{theorem:symplectic_leaf} below.
\end{remark}

\begin{lemma}\label{lemma:OmegaPi}
The restriction of the bilinear form (\ref{eq:Omega}) to the image $\mr{im}(\pi) \sub T_g \mhiggs^\fr_G(\bb{CP}^1, D, \omega^\vee)$  is compatible with Sklyanin's Poisson structure $\pi$ on $G_1[[z^{-1}]]$.
\end{lemma}

\begin{proof}
  Given two evaluation functions $\phi_{u}$ and $\psi_{v}$, where $u$ and $v$ are points in $\CC \bs D$, we'd like to check
  that Sklyanin's Poisson bracket is compatible (in the sense of (\ref{eq:compatibility})) with the bilinear form $\Omega$, so
\begin{equation}
  \Omega(X_{\psi_{u}}, X_{\phi_{v}})  = - \{\psi_{u}, \phi_{v}\}.
\end{equation}
First we'll choose equivalence frames such that the Hamiltonian vector field $X_{\phi_{u}}$ is regular at $\tilde D$, and $X_{\psi_{v}}$ is regular in $\bb{CP}^1 \setminus \tilde D$.  The equivalence frame regular in $\tilde D$ can be taken as in  (\ref{eq:Xphiz}),
and the equivalence frame regular in $C \setminus \tilde D$ can be taken
as in (\ref{eq:xlw0}). That is, we'll take
\begin{align*}
    (X_{\phi_{u}})_{i} &= ((X_{\phi_{u}}^{L}, X_{\phi_{u}}^{R}))_{i} \\
    &=  \frac{1}{z - u} \left(\Ad_{g(u)} \nabla_R \phi_{u},   \nabla_{R} \phi_{u}\right) \\
    (X_{\psi_{v}})_{0} &= ((X_{\psi_{v}}^{L}, X_{\psi_{v}}^{R}))_{0} \\
    &= \frac{1}{z - v} \left(\left( \Ad_{g(v)} \nabla_R \psi_{v} - \Ad_{g(z)} \nabla_{R} \psi_{v}\right), 0\right).
\end{align*}
Then the definition of the bilinear form $\Omega$ in this equivalence frame becomes
\begin{equation}
  \label{eq:pairing}
  \Omega(X_{\phi_{u}}, X_{\psi_{v}}) = \frac{1}{2 \pi \imath }
  \sum_{z_i \in \tilde D} \oint_{\partial U_i} \frac{\d z}{(z - u)(z - v)} \langle \Ad_{g(u)} \nabla_{R} \phi_{u},
  \Ad_{g(v)} \nabla_{R} \psi_{v}  - \Ad_{g(z)} \nabla_{R} \psi_{v} \rangle.
\end{equation}
The integrand is regular everywhere on $U_0 \setminus \{u, v\}$, and since
\[\sum_{i \in \tilde D} \oint_{\partial U_i } f(z) \d z  = - \oint_{\partial U_0} f(z) \d z\]
the pairing (\ref{eq:pairing}) is given as a sum of residues at $w = u$ and $w = v$.  Explicitly:
  \begin{align*}
   \Omega(X_{\phi_{u}}, X_{\psi_{v}}) &= - (\res_{z = u}  + \res_{z = v}) \frac{\d z}{(z - u)(z - v)} \langle \Ad_{g(u)} \nabla_{R} \phi_{u},   \Ad_{g(v)} \nabla_{R} \psi_{v}  - \Ad_{g(z)} \nabla_{R} \psi_{v} \rangle \\
   &=   \frac{-1}{u - v}  \langle \Ad_{g(u)} \nabla_{R} \phi_{u} , \Ad_{g(v)} \nabla_{R} \psi_{v}  - \Ad_{g(u)} \nabla_{R} \psi_{v} \rangle  \\
   &\quad - \frac{-1}{u - v}  \langle \Ad_{g(u)} \nabla_{R} \phi_{u} , \Ad_{g(v)} \nabla_{R} \psi_{v}  - \Ad_{g(v)} \nabla_{R} \psi_{v}   \rangle \\
      &=   \frac{-1}{u - v}  \langle \Ad_{g(u)} \nabla_{R} \phi_{u} ,  - \Ad_{g(u)} \nabla_{R} \psi_{v} \rangle   - \frac{-1}{u - v}  \langle \Ad_{g(u)} \nabla_{R} \phi_{u} ,  - \Ad_{g(v)} \nabla_{R} \psi_{v}   \rangle \\
   &= \frac{1}{u - v} (\langle \nabla_{R} \phi_{u} , \nabla_{R} \psi_{v}\rangle - 
   \langle  \nabla_{L} \phi_{u} , \nabla_{L} \psi_{v} \rangle ) \\
   &=  - \{\phi_{u}, \psi_{v} \}.
  \end{align*}
\end{proof}

To complete the argument that $\mhiggs^\fr_G(\bb{CP}^1, D, \omega^\vee)$ is a symplectic leaf for the Sklyanin Poisson structure on $G_1[[z^{-1}]]$ we need to show that in fact, we have an identification 
\[\mr{im}(\pi) = T_g \mhiggs^\fr_G(\bb{CP}^1, D, \omega^\vee).\]
That means for any tangent vector $X \in T_g \mhiggs^\fr_G(\bb{CP}^1, D, \omega^\vee)$ we'd like to find a Hamiltonian
function $\phi$ on $G_1[[z^{-1}]]$ such that $X = X_{\phi} = \pi d \phi$.

\begin{theorem}\label{theorem:symplectic_leaf}
The restriction map $r_\infty \colon \mhiggs^\fr_G(\bb{CP}^1, D, \omega^\vee) \to G_1[[z^{-1}]]$ to a formal neighbourhood of infinity is the inclusion of a symplectic leaf for the group $G_1(\bb{CP}^1)$ of $G$-valued framed rational functions on $\bb{CP}^1$ under the Sklyanin Poisson structure.
\end{theorem}

We'll show this at the level of tangent spaces at a point $g$ in $\mhiggs^\fr_G(\bb{CP}^1, D, \omega^\vee)$.  For any deformation $\delta g $ of $g$ that preserves the singularity divisor $D$ of $g$ we want to find a Hamiltonian function $\phi$ for Sklyanin's bracket such that the associated Hamiltonian vector field $X_\phi = \delta g$; we've already shown the opposite inclusion in Lemma \ref{phitoX}.

First let's give an explicit basis for the tangent space of $\mhiggs^\fr_G(\bb{CP}^1, D, \omega^\vee)$ at a point $g$.  Assume that the restriction of $g$ to each punctured open neighbourhood $U^\times_i$ of a point in $D$ is regular semi-simple valued.  With this assumption, the operator $\Ad_{g}(z): \gg \to \gg$ is diagonalizable in $U_{i}^{\times}$.  We have a family of Cartan sublagebras $\hh_{z} \subset \gg$ parameterized by $z \in U_{i}^{\times}$, where $\hh_z$ centralizes the regular semisimple elements $g(z)$.  We then have a ($z$-dependent) decomposition of $\gg$ into the Cartan $\hh_z$ and the root spaces:
\[\gg = \hh_{z} \oplus \sum_{\alpha} \gg_{\alpha,z}.\]
Since the family $g(z)$ has a singularity of co-weight $\omega_{z_i}^{\vee}$ in $z_i$ with $[g(z)] \sim z^{-\omega_{z_i}^{\vee}}$, the operator $\Ad_{g}$ has eigenvalue $ (z - z_i)^{- (\alpha, \omega^{\vee}_{z_i})}$ on $\gg_{\alpha}$
in leading order. Let $e_{\alpha, z}$ be a generator of $\gg_{\alpha, z}$.

Below we'll need a technical assumption on $g(z)$ in $U_{i}^{\times}$: that the splitting
$\hh_{z} \oplus \sum_{\alpha} \gg_{\alpha,z}$ extends analytically from $U_{i}^{\times}$ to $U_{i}$,
i.e. that there is a limit 
 \[ e_{\alpha, z_i} = \lim_{z \to z_i} e_{\alpha, z}.\] 
We'll use these conditions to calculate the dimension of the moduli space $\mhiggs^\fr_G(\bb{CP}^1, D, \omega^\vee)$; they are satisfied for $g$ living in a dense open subset of the moduli space. 

\begin{lemma}
  Assuming $g(z)$ has regular semisimple values near each singularity, the tangent space of $\mhiggs^\fr_G(\bb{CP}^1, D, \omega^\vee)$ at $g$ is isomorphic to the space of meromorphic $\gg$-valued functions $X^{L}$ of the form
  \begin{equation}
\label{eq:XL}
    X^{L} = \sum_{i} \sum_{\alpha: \langle \alpha, \omega_{z_i}^{\vee} \rangle > 0 }
    \sum_{k_{i, \alpha} = 1}^{ \langle \alpha , \omega_{z_i}^{\vee} \rangle} e_{\alpha,{z_i}} x_{i, \alpha, k_{i, \alpha}} (z - z_i)^{-\langle \alpha, \omega_{z_i}^{\vee}\rangle }, \qquad x_{i, \alpha, k_{i, \alpha}} \in \mathbb{C}
  \end{equation}
  and consequently the tangent space $T_g \mhiggs^\fr_G(\bb{CP}^1, D, \omega^\vee)$ has dimension 
  \begin{equation}
\label{eq:dim-unreduced} 
\dim T_g \mhiggs^\fr_G(\bb{CP}^1, D, \omega^\vee) = \sum_{i} \sum_{\alpha: \langle \alpha, \omega_{z_i}^{\vee} \rangle > 0 } \langle \alpha, \omega_{z_i}^{\vee} \rangle = 2 \sum_{i} \langle \rho, \omega^{\vee}_{z_i}\rangle 
  \end{equation}
  where $\rho = \frac 1 2 \sum_{\alpha > 0} \alpha$ is the Weyl vector.
\end{lemma}

\begin{proof}
  In the framing $(X^{L}, X^{R}) \sim (X^{L} + \Ad_{g} X^{R}, 0 )$, the singular
  part of $X^{L}$ in $U_{i}^{\times}$ in the subspace generated by $e_{\alpha, z_i}$
  is in the image of the operator $\Ad_{g}$ applied to a regular section generated by $e_{\alpha, z_i}$.  Consequently, there is an equivalence frame $(X^{L}, X^{R}) \sim (\tilde X^{R}, \tilde X^{R})$  in which $\tilde X^{L}$ and $\tilde X^{R}$ are both regular in $U_{i}$. 
\end{proof}

Having described the tangent space to $\mhiggs^\fr_G(\bb{CP}^1, D, \omega^\vee)$ at $g$, we'd like to show that for every tangent vector $X$ there is a Hamiltonian function $\phi$ on $G_1[[z^{-1}]]$ such that $X = X_\phi$ with respect to Sklyanin's Poisson structure on $G_1[[z^{-1}]]$.  Suppose that is $X$ is represented in the equivalence class by $((X^{L}, X^{R}))_{i}$, where $(X^{L}, X^{R})_{i}$ are regular in each chart $U_i$ around $z_i \in \tilde D$.  We'll describe a Hamiltonian for $X$ in the following way.

\begin{lemma}\label{xtophi}
  For any $X \in T_{g} \mhiggs^\fr_G(\bb{CP}^1, D, \omega^\vee)$ there is a Hamiltonian potential $\phi$ on $G_1[[z^{-1}]]$ such that $X = \pi \d \phi|_{g}$.  For $\tilde g$ in a local neighborhood of $g$, a potential $\phi$ is given by the formula 
\[
  \phi(\tilde g) : = -\frac 1 2 \sum_{z_i \in \tilde D} \frac{1}{2 \pi \imath c_\rho} \oint_{\partial U_{i}} 
dz  \tr \rho( \tilde g_z  g_z^{-1}) \rho(X^{L}_i - \Ad_{ g_z} X_{i}^{R}),
\]
where $\rho$ is a faithful representation of $G$ and we use a non-degenerate pairing of the form $\tr\rho(X) \rho(X') = c_\rho \kappa(X, X')$.
\end{lemma}
  
\begin{proof}
The left gradient of $\phi(\tilde g)$ under $ \delta \tilde g = Y^{L} \tilde g$ at $\tilde g = g$ is given by
\begin{equation}
  Y^{L} \nabla_{L} \phi|_{\tilde g = g}  = - \frac 1 2  \sum_{z_i \in \tilde D} \frac{1}{2 \pi \imath } \oint_{\partial U_{i}}  \d z
 \langle  Y^{L} , (X^{L}_i - \Ad_{g_z} X_i^R) ) 
\end{equation}
so 
\begin{equation}
  \nabla_{L} \phi|_{\tilde g = g}  = -\frac 1 2  \sum_{z_i \in \tilde D} \frac{1}{2 \pi \imath } \oint_{\partial U_{i}}  \d z
 (X^{L}_i - \Ad_{g_z} X_i^R) 
\end{equation}
where we've identified $\gg$ and $\gg^{*}$ using the pairing $\kappa$. 

Now we'll compute Sklyanin's Hamiltonian vector field in the left frame $(X_{\phi}^{L}, 0)$ using the equation (\ref{eq:xlw0}) (and using the fact that $\nabla_{R} \phi  = \Ad_{g^{-1}} \nabla_{L} \phi$).  We find
\begin{align*}
  X^{L}_\phi &= -\frac 1 2 \sum_{i \in \tilde D} \frac{1}{2 \pi \imath} \oint_{\partial U_i}
  \frac{\d z}{w - z}  (X_i^{L} - \Ad_{g_z} X_i^{R})  - \Ad_{g_w} \Ad_{g_z}^{-1} (X_i^{L} - \Ad_{g_z} X_i^{R}) \\
&=  -\frac 1 2 \sum_{i \in \tilde D} \frac{1}{2 \pi \imath} \oint_{\partial U_i}
\frac{\d z}{w - z}  (X_i^{L}  + \Ad_{g_w} X_i^{R}) - (\Ad_{g_z} X_i^{R} +  \Ad_{g(z)} \Ad_{g_{z}}^{-1} X_i^{L}) \\
 &= -\frac 1 2  \sum_{i \in \tilde D} \frac{1}{2 \pi \imath} \oint_{\partial U_i}
 \frac{\d z}{w - z}  (X_i^{L}  + \Ad_{g_w} X_i^{R}) - (X_0^{L} + \Ad_{g_z} X_0^{R}- X_{i}^{L} +  \Ad_{g(z)} ( \Ad_{g_{z}}^{-1} X_0^{L} + X_0^{R} - X_{i}^{R})) \\
  &= -\frac 1 2  \sum_{i \in \tilde D} \frac{1}{2 \pi \imath} \oint_{\partial U_i}
    \frac{\d z}{w - z} 2 (X_i^{L}  + \Ad_{g_w} X_i^{R}) - (X_0^{L} + \Ad_{g_z} X_0^{R} +  \Ad_{g(z)} ( \Ad_{g_{z}}^{-1} X_0^{L} + X_0^{R} )).
  \end{align*}
Now, as a function of $z$, the first term of the numerator of the integrand is regular in each chart $U_i$ for $I \in \tilde D$, and the second term of the numerator is regular in the remaining chart $U_0$. Therefore the first term evaluates to the residue at $z = w$ while the second is zero if $w \in U_i$, and the second term evaluates to a residue at $z = w$ while the first is zero if $w \in U_0$. (Assume without loss of generality that $w$ is not on the integration contour).  Either way, if $w \in U_i$ we have, from the residue in the first term,
  \[
    (X^{L}_\phi, 0)(w) = \mathrm{res}_{z = w} \frac{\d z}{ z - w} (X_{i}^{L} + \Ad_{g_w} X_{i}^{R}) =
    (X_{i}^{L} + \Ad_{g_w} X_{i}^{R})(w),
  \]
  and if $ w \in U_0$ we have, from the residue in the second term,
  \begin{align*}
    (X^{L}_\phi, 0)(w) &= \frac 1 2 \mathrm{res}_{z = w} \frac{\d z}{ z - w} (X_0^{L} + \Ad_{g_z} X_0^{R} +  \Ad_{g(z)} ( \Ad_{g_{z}}^{-1} X_0^{L} + X_0^{R} )) \\
    &= ( X_{0}^{L} + \Ad_{g_w} X_{0}^{R})(w).
  \end{align*}
  Thus $X_\phi$, in the left equivalence frame $(X^{L}_{\phi}, 0)$ in each chart, coincides with the vector field $X$ that we've started
  in the same left equivalence frame $(X^{L}, X^{R}) \sim (X^{L} + \Ad_{g} X^{R}, 0)$.
\end{proof}

\begin{proof}[Proof of Theorem \ref{theorem:symplectic_leaf}]
We've shown in Lemma \ref{phitoX} that $\mathrm{im}(\pi) \subset T_{g} \mhiggs^\fr_G(\bb{CP}^1, D, \omega^\vee)$ and in Lemma \ref{xtophi} that $T_{g} \mhiggs^\fr_G(\bb{CP}^1, D, \omega^\vee) \subset \mathrm{im}(\pi)$, hence $\mr{im}(\pi) = T_{g} \mhiggs^\fr_G(\bb{CP}^1, D, \omega^\vee)$. In combination with Lemma \ref{lemma:OmegaPi} proving the compatibility of the Poisson structure $\pi$ and the bilinear form $\Omega$, this implies that $\mhiggs^\fr_G(\bb{CP}^1, D, \omega^\vee)$ is
a symplectic leaf for the Poisson-Lie group $G_1(\bb{CP}^1)$. In particular, since $\pi$ is Poisson, the bilinear form $\Omega$ is non-degenerate and closed, and so the form $\Omega$ actually defines a symplectic structure on the moduli space $\mhiggs^\fr_G(\bb{CP}^1, D, \omega^\vee)$.
\end{proof}

\begin{remark} \label{shapiro_leaves_remark}
It's instructive to compare this calculation with the work of Shapiro \cite{Shapiro} on symplectic leaves for the rational Poisson Lie group.  According to Shapiro, and for the group $G = \SL_n$, there are symplectic leaves in $G_1[[z^{-1}]]$ indexed by Smith normal forms, i.e. by a sequence $d_1, \ldots, d_n$ of polynomials where $d_i | d_{i+1}$ for each $i$.  This data is equivalent to a dominant coweight coloured divisor $(D,\omega^\vee)$ in the following way.  The data of the sequence of polynomials is equivalent to the data of an increasing sequence $(D_1, \ldots, D_n)$ of $n$ effective divisors in $\bb{CP}^1$ disjoint from $\infty$: the tuples of roots of the polynomials $d_i$.  Let $D = \{z_1, \ldots, z_k\}$ be the support of the largest divisor $D_n$, and for each $z_j$ let $\omega^\vee_{z_j} = (m_1, \ldots, m_n)$ be the dominant coweight where $m_i$ is the order of the root $z_j$ in the polynomial $d_i$.  The dominant coweight Shapiro refers to as the ``type'' of a leaf is, therefore, the sum of all these coweights.

To be a little more precise, Shapiro proves that for $G = \SL_n$ these symplectic leaves span the Poisson subgroup $\mc G \sub G_1[[z^{-1}]]$ of elements that can be factorized as the product of a polynomial element in $G_1[z^{-1}]$ and a monic $\CC$-valued power series.  In fact, the map $r_\infty$ (for any group, not necessarily $G = \SL_n$) factors through the subgroup $\mc G$: the Taylor expansion of all \emph{rational} $G$ valued functions can be factorized in this way.  As a consequence, Theorem \ref{theorem:symplectic_leaf} implies that our symplectic leaves for the group $\SL_n$ agree with Shapiro's symplectic leaves.
\end{remark}

\begin{remark}
We can also compare this description of the symplectic structure $\Omega$ on our symplectic leaves with the symplectic structure on the moduli space of multiplicative Higgs bundles on an elliptic curve studied by Hurtubise and Markman \cite{HurtubiseMarkman}.  In order to construct a symplectic structure by a procedure analogous to their construction, we could describe a pairing induced from the natural equivalence between the tangent and cotangent spaces of the moduli space of multiplicative Higgs bundles as described in Section \ref{def_section}.  That is, there's a map of complexes of sheaves
\[\xymatrix{
(\mc F^\fr_{(P,g)})^*[1] \ar@{=}[r] \ar[d] &\Big(\gg^*_P(-D)[1] \otimes \OO(-c) \ar[d]^{\kappa \circ A_g^*} \ar[r]^{A_g^*} &\gg^*_P \otimes \OO(-c)\Big)\ar[d]^{A_g \circ \kappa^{-1}} \\
\mc F^\fr_{(P,g)} \ar@{=}[r] &\Big(\gg_P[1] \otimes \OO(-c) \ar[r]^{A_g} &\gg_P(D) \otimes \OO(-c)\Big)
}\]
where here $\kappa$ denotes the isomorphism from $\gg_P \to \gg^*_P$ induced by the Killing form.  The top line is the Serre dual complex to the bottom line; note that the incorporation of the framing was necessary for this to be the case (that is, we're using the relative Calabi-Yau structure on the pair $(\bb{CP}^1, c)$).  Taking 0th hypercohomology we obtain a map from the cotangent space to the tangent space of our moduli space of multiplicative Higgs bundles.
\end{remark}

\subsection{Hamiltonian Reduction} \label{Reduced_section}

Now, let us assume that we consider the moduli space of framed multiplicative Higgs bundles on $\bb{CP}^1$ where the framing at infinity is a regular semi-simple framing element $g_{\infty}$ of the group $G$.  Multiplication by the constant function with value $g_\infty$ allows us to identify the space $G_{g_\infty}[[z^{-1}]]$ of $G$-valued power series with constant term $g_\infty$ as a torsor for the Poisson Lie group $G_1[[z^{-1}]]$, so the moduli space of multiplicative Higgs bundles with framing value $g_\infty$ can still be viewed as a symplectic leaf in the Poisson Lie group.

Let $T \sub G$ be the centralizer of the element $g_\infty \in G$; since $g_\infty$ is a regular semisimple element, $T$ is a maximal torus in $G$. Notice that adjoint action $\Ad_{h}$ by a constant element $h \in T$ on the Higgs field $g(z)$
\[
\Ad_{h} \colon  g(z) \mapsto h g(z) h^{-1} 
\]
preserves the degree of singularities at the divisor $D$ and also preserves the framing $g_\infty$. Therefore, the differential of $\Ad_{h}$ defines a cotangent vector in $T^*_{g} \mhiggs^\fr_G(\bb{CP}^1, D, \omega^\vee)$. 

Recall from Corollary \ref{cor:poisson-commuting} that adjoint invariant functions of $g(z)$ descend to Poisson commuting functions on the moduli space $\mhiggs^\fr_G(\bb{CP}^1, D, \omega^\vee)$.
 
\begin{example}
Let $\rho$ be a faithful representation of $G$, and let $u$ be an element of $\CC$. The \emph{evaluation character}  
\begin{align*}
\varphi_{\rho,v} \colon \{g \in G_1[[z^{-1}]] \colon g(v) \text{ converges}\} &\to \mathbb{C}\\
g &\mapsto  \tr_{\rho} g(v)
\end{align*}
is an adjoint invariant function. 
\end{example}

In this subsection, we'll show that the $T$ action on $\mhiggs^\fr_G(\bb{CP}^1, D, \omega^\vee)$ by global conjugation is Hamiltonian, generated by the residues at infinity of adjoint invariant Hamiltonian functions.

\begin{remark}
More precisely, in the reductive case we won't have a Hamiltonian $T$ action but a Hamiltonian action of a subtorus $T'$ of rank $r$, the rank of the semisimple part of the gauge group.
\end{remark}

First we'll proceed in the reverse direction.  That is, we'll show that the Hamiltonian vector field generated by the residue at infinity of an adjoint invariant function corresponds to the adjoint action by an element $h \in T$. 
 \begin{lemma} 
 Let $\varphi_{\rho, z}$ be the evaluation character of a faithful $G$-representation evaluated at the point $z$. Define
   \begin{equation}
    \res  \varphi_{\rho, \infty} : = \frac{1}{2 \pi \imath} \oint_{\partial U_{\infty}}  \d z \varphi_{\rho, z} 
   \end{equation}
   to be the residue of $\varphi_{\rho, z}$ at $z=\infty$.   Then the Hamiltonian vector field  $X_{\rho, \infty} = \pi d \res \varphi_{\rho, \infty}$ generates a constant adjoint action by $h \in T$ on $\mhiggs^\fr_G(\bb{CP}^1, D, \omega^\vee)$
   where $T \subset G$ is the centralizer of the framing value $g_\infty$. 
 \end{lemma}
 
 \begin{proof}
   From the left frame version of Sklyanin's formula (\ref{eq:sklyanin-left})  we have
   \begin{equation}
     (X^{L}_{\rho, \infty}(w),0)  =\frac{1}{2 \pi \imath} \oint_{\partial U_{\infty}} \frac{\d z}{w - z}
     ((\Ad_{g_z} \nabla_{R} \varphi_{\rho, z} - \Ad_{g_{w}} \nabla_{R} \varphi_{\rho, z} ), 0)
   \end{equation}
Here, as usual, we've identified $\gg$ and $\gg^{*}$ using the Killing form.  The numerator of the integrand is a regular function in the chart $U_{\infty}$ away from the point $z = \infty$.   Therefore the contour integral can be evaluated as a sum of the residues at $z=\infty$ and at $z = w$ if $w \in U_{\infty}$.  The residue at $z = w$ vanishes because the numerator vanishes at $z = w$. Therefore
   \begin{align*}
     (X^{L}_{\rho, \infty}(w),0)  &= \mathrm{res}_{z = \infty} (\frac{dz}{z - w}
     (\Ad_{g_z} \nabla_{R} \varphi_{\rho,z}  - \Ad_{g_{w}} \nabla_{R} \varphi_{\rho, z} ),0) \\
    &= ( (\Ad_{g_\infty} - \Ad_{g(w)}) \nabla_{R} \varphi_{\rho, \infty} ,0) \\
    &\sim (\nabla_{L} \varphi_{\rho, \infty}, - \nabla_{R} \varphi_{\rho, \infty})
\end{align*}
where the last operation $\sim$ indicates equivalence under the relation (\ref{eq:equivalence}). Consequently we've obtained that $X_{\rho, \infty} \sim (\nabla_{L} \varphi_{\rho, \infty}, - \nabla_{R} \varphi_{\rho, \infty})$ is a constant vector field, i.e. is independent of $w$. Moreover, since $\varphi$ is adjoint invariant it follows that $\nabla_{L} \varphi_{\rho, \infty} = \nabla_{R} \varphi_{\rho, \infty}$, and therefore $X_{\rho, \infty}$ generates the constant adjoint action on $\mhiggs^\fr_G(\bb{CP}^1, D, \omega^\vee)$ generated by the element $\nabla_{L} \varphi_{\rho, \infty} $ of $\gg$. It remains to show that the element $\nabla_{L} \varphi_{\rho, \infty} $ belongs to the Lie
algebra of the centralizer $T \subset G$ of $g_{\infty}$.  Recall that $\varphi_{\rho, \infty} = \tr \rho(g_\infty)$. By definition, using the Killing form,
$X = \nabla_{L} \varphi_{\rho, \infty}$ is an element of $\gg$ such that
\[\tr_{\rho}(Y g_\infty)  = \langle Y, X \rangle \quad \text{ for all } Y \in \gg,\]
where we are using the short-hand notation $\tr_\rho(a_1 a_2 \cdots a_n)$ for $\tr(\rho(a_1)\rho(a_2) \cdots \rho(a_n))$.

We want to show that $X$ is in the centralizer of $g_\infty$, so $\tilde X = X$
where $\tilde X: = \Ad_{g_\infty} X$. We can characterize $\tilde X$ by the condition
\begin{equation}
  \tr_{\rho}(Y g_\infty)  = \langle \Ad_{g_\infty} Y, \tilde X \rangle \quad \text{ for all } Y \in \gg
\end{equation}
Let $\tilde Y =  \Ad_{g_\infty} Y$.  Since $\Ad_{g_\infty}: \gg \to \gg$ is an isomorphism, we can equivalently write this as
\begin{align*}
  \tr_{\rho}(\Ad_{g_\infty^{-1}} \tilde Y g_\infty)  &= \langle \tilde Y, \tilde X \rangle \quad \text{ for all } \tilde Y \in \gg\\
  \text{or } \tr_\rho(\tilde Y g_\infty) &= \langle \tilde Y, \tilde X \rangle \quad \text{ for all } \tilde Y \in \gg
\end{align*}
by cyclic invariance of the trace. Therefore, $\tilde X$ is uniquely determined by the same relation as $X$, and hence $\tilde X = X$, i.e. $\Ad_{g_\infty} X = X$, and hence $X$ is in the centralizer $T \subset G$ of $g_\infty$ as required. 
\end{proof}

Now we'd like to find a collection of Hamiltonian functions of the form $\res \varphi_{\rho_i, \infty}$, for a set of faithful $G$-representations $\{ \rho_i \}$, which generate the action of the whole group $T$.  The vector field $X_\rho$ generated by $\res \varphi_{\rho, \infty}$ as in the previous lemma is
defined by the relation 
\[\tr_{\rho}(Y g_\infty)  = \langle Y, X_\rho \rangle \quad \text{ for all } Y \in \mf t,\]
where now we can assume that $X_\rho, Y$ lie in the Lie algebra $\mf t$ of $T$, and that $g_\infty \in H$.  Write $g_\infty$
  in the form
\[g_\infty =   \prod_{k} \qq_k ^{\omega_k^{\vee}}\]
where $\omega_k^{\vee}$ are fundamental co-weights $\omega_{k}^{\vee} \colon \mathbb{C}^{\times} \to T$
and their evaluation is denoted exponentially, i.e. by $\omega \colon \qq \mapsto \qq^{\omega^{\vee}}$. Assume that $|\qq_k | < 1$. 

\begin{remark}
This form for the framing value $g_\infty$ is motivated by the gauge theory construction in \cite{NekrasovPestun}, where $\qq_i$ correspond to the exponentiated coupling constants for an ADE quiver theory.
\end{remark}

In this form, $X_\rho$ is determined by
\[\sum_{w \in \mathrm{weights}_{\rho}}(w(Y) w(g_\infty))  = \langle Y, X_\rho \rangle .\]
 Identified with an element of the dual Lie algebra $\mf t^{*}$, the $X_\rho$ can then be written as
  \begin{equation}
\label{eq:Xrho}
    X_{\rho} = \sum_{w \in \mathrm{weights}_{\rho}} w   \prod_{k} \qq^{w(\omega_k^{\vee})},
\end{equation}
 where the sum is over weights $w \in \mf t^{*}$ of the representation $\rho$.  We'll use this form to show the following.

\begin{lemma}\label{lemma:smallq}
  Suppose that $g_\infty = \prod_{k} \qq_k ^{\omega_k^{\vee}}$ is a regular semi-simple element.  There exists $\eps > 0$ such that if $|\qq_k| < \eps$ for all $k$, then for a collection of highest weight representations $\rho_w$ with linearly independent highest weights $w$ in the Lie algebra $\mf t^*$, the Hamiltonian functions $\res \varphi_{\rho_w, \infty}$ generate the adjoint action of a subtorus $T'$ of the centralizer $T \subset G$ of the element $g_\infty$ of dimension $r$, the rank of the semisimple part of $G$. 
\end{lemma}  

\begin{proof}
Assume that $\rho$ is a finite-dimensional irreducible representation with highest weight $w_h$.
  We'll pick the highest weight term and recall that the other weights $w'$ of $\rho$ can be related to the highest weight $w_{h}$ by 
  \[w' = w_{h} - \sum {n_i \alpha_i},\]
  where $n_i \geq 0$,  $\sum n_i \geq 1$, and $\alpha_i$ are simple roots dual to coweights $\omega_i^{\vee}$.
  Therefore, the terms in (\ref{eq:Xrho}) corresponding to the lower weights $w'$ in the sum will be suppressed in the $\qq_k \to 0$ limit by the coefficient $\prod_{k} \qq_k^{n_k}$.
  
  More explicitly, take a collection of highest weight irreducible representations $\{\rho_{w}\}$ with linearly independent highest weights $w_{\rho}$ (for example, a set of fundamental weights). Then in the leading order in $\qq_k$, and in the limit $\qq_k \to 0$, the value $X_{\rho}$ is determined by the first term, with highest weight $w_{\rho}$, and the correction terms are given by polynomials in $\qq_k$ with zero constant term, therefore by expressions vanishing in the limit $\qq_k \to 0$. In particular, since the representations $\{\rho_w\}$ are linearly independent, the vector fields $X_{\rho_{w}}$ are also linearly independent in the limit where $\qq \to 0$.  Since linear independence is an open condition, it follows that the Hamiltonian vector fields $X_{\rho_w}$ are linearly independent for sufficiently small $|\qq_k|$.
\end{proof}

\begin{remark}
  The limit where the parameters $\qq_k$ become very small is identified the perturbative limit of a quiver gauge theory in \cite{NekrasovPestun}.  This limit also appears in Foscolo's work \cite{FoscoloGluing} on the construction of periodic monopoles, corresponding to the large value of the Higgs field at infinity (we'll discuss the connection to periodic monopoles in Section \ref{periodic_monopole_section}. 
\end{remark}

\begin{remark}
The rank $r$ subtorus $T'$ is not generally canonical, but any choice is canonically isogenous to the quotient $T/Z(G)$ of $T$ by the center of $G$ (by the composition of the inclusion $T' \inj T$ with the projection onto the quotient, which is a homomorphism between rank $r$ tori with finite stabilizers, therefore surjective).  Different choices will give canonically isomorphic Hamiltonian reductions.
\end{remark}

We define the \emph{reduced moduli space} of multiplicative Higgs bundles to be the Hamiltonian reduction by this adjoint $T'$-action.

\begin{corollary}
  For a regular semi-simple element $g_\infty \in G$, with $|\qq_k|$ sufficiently small, there is a symplectic space $\mhiggs^\red_G(\bb{CP}^1, D, \omega^\vee)$ defined to be the symplectic reduction of $\mhiggs^\fr_G(\bb{CP}^1, D, \omega^\vee)$ by the torus $T'$, i.e.
  \[\mhiggs^\red_G(\bb{CP}^1, D, \omega^\vee) := (\res \varphi_{\rho_{w}, \infty} )^{-1} (c_{w}) // T'\]
  where $\rho_{w}$ runs over a collection of irreducible highest weight representations with linearly independent highest weights $w$,  and fixed $c_{w} \in \mathbb{C}$. This reduced moduli space has dimension
  \begin{equation}
    \label{eq:dim-reduced}
\dim_{\mathbb{C}} \mhiggs^\red_G(\bb{CP}^1, D, \omega^\vee)  =    \dim_{\mathbb{C}}  \mhiggs^\fr_G(\bb{CP}^1, D, \omega^\vee) - 2 \rank(T').
    \end{equation}
\end{corollary}

Using the multiplicative Hitchin fibration from Section \ref{Hitchin_system_section}, this reduction is expected to resolve the problem mentioned in Remark \ref{non_compact_fiber_remark}.  We encapsulate this expectation in the following conjecture.

\begin{conjecture}
The reduced moduli space $\mhiggs^\red_G(\bb{CP}^1, D, \omega^\vee)$ is the total space of an algebraic integrable system with generically compact smooth fibers.
\end{conjecture}

\begin{remark}
We've already argued that the total space is algebraic symplectic, and that the multiplicative Hitchin fibers are generically Lagrangian.  It only remains to verify that the fibers are generically complex tori.  When $G = \GL(r)$ one can verify the conjecture directly by computing the genus of the spectral curve with the Newton polygon. 
\end{remark}

\begin{example}
Consider the moduli spacee $\mhiggs^\red_{\GL(r)}(\bb{CP}^1, D, \omega^\vee)$ associated to a colored divisor $(D, \omega^\vee)$
given by a collection of $n$ points valued in $\omega_1^{\vee}$, and a collection of $n$ points valued in $\omega_{r-1}^{\vee}$.  Here $\omega_i^{\vee}$ denotes the fundamental co-weight associated to the $i^\mr{th}$ node in the Dynkin diagram of $A_{r-1}$, numbered in increasing order from one end of the diagram. In this case, the corresponding $A_{r-1}$ quiver gauge theory in \cite{NekrasovPestun} has gauge group $\prod_{i=1}^{r-1} \SU(n_i)$ with $n_i = n$, and there are $n$ fundamental matter multiplets attached to the node $1$, and $n$ anti-fundamental matter multiplets attached to the node $r-1$. The spectral curve is a planar curve in $C \times \mathbb{C}$ of the form (see e.g. \cite{NekrasovPestun})
\[P(x,t) = 0\]
where $P(x,t) = \sum_{i=0}^{r} t^{r-i} T_{i}(x)$, and the $T_i(x)$ all have the same degree $n$.

The Newton polygon of $P(x,t)$ is a $n \times r$ rectangle. By the genus formula for a generic plane curve \cite{Baker, Khovanskii}, the genus $g(P)$ of the spectral curve is equal to the number of integral points in the interior of the Newton polygon: $g(P) = (n-1)(r-1)$. The fibers of the algebraic integrable system are identified with the Jacobian of the spectral curve, and hence we obtain that
the dimension of the fiber is $(n-1)(r-1)$.

On the other hand, recall that the dimension formula (\ref{eq:dim-reduced}) together with (\ref{eq:dim-unreduced}) gives
\begin{align*}
  \dim \mhiggs^\red_G(\bb{CP}^1, D, \omega^\vee) &=  \dim \mhiggs_{G}(\bb{CP}^1, D, \omega^\vee) - 2 \rank (T') \\
  &= \sum_{i} 2 \langle \rho, \omega_i^{\vee} \rangle  - 2 (r-1)  \\
  &= 2 n (r - 1)  - 2(r - 1) \\ 
  &= 2(n-1)(r-1),
\end{align*}
 We see that, indeed, the dimension of the fiber of the Hitchin fibration is half of the dimension of the total reduced moduli space $\mhiggs^{\red}_{G}(\bb{CP}^1, D, \omega^\vee)$. 
\end{example}

In Part \ref{part2} of the paper, we'll show that the reduced moduli space carries not only a symplectic structure, but a hyperk\"ahler structure, and that it can be identified with a hyperk\"ahler moduli space of periodic monopoles on $\RR^2$ with a framing at $\infty$.

\section{Quantization and the Yangian} \label{quantization_section}

The description of the moduli spaces of multiplicative Higgs bundles as symplectic leaves in the rational Poisson Lie group has interesting consequences upon quantization.  The Poisson algebra $\OO(G_1[[z^{-1}]])$ has a well-studied quantization to the \emph{Yangian} $Y(\gg)$.  This is the unique Hopf algebra quantizing the algebra $U(\gg[z])$ with first order correction determined by the Lie bialgebra structure on $\gg[z]$.  Recall that $(\gg(\!(z)\!), \gg[z], \gg_1[[z^{-1}]])$ is a Manin triple, so the residue pairing on $\gg(\!(z)\!)$ induces an isomorphism between $\gg_1[[z^{-1}]]$ and the dual to $\gg[z]$, and therefore a Lie cobracket on $\gg[z]$.  The uniqueness of this quantization is a theorem due to Drinfeld {\cite[Theorem 2]{DrinfeldQuantum1}}.

\begin{definition}
The \emph{Yangian} of the Lie algebra $\gg$ is the unique graded topological Hopf algebra $Y(\gg)$ which is a topologically free module over the graded ring $\CC[[\hbar]]$ (where $\hbar$ has degree 1), so that when we set $\hbar=0$ we recover $Y(\gg) \otimes_{\CC[[\hbar]]} \CC \iso U(\gg[z])$ as graded rings (where $z$ has degree 1), and where the first order term in the comultiplication $\Delta$ is determined by the cobracket $\delta$ in the sense that
\[\hbar^{-1}(\Delta(f) - \sigma(\Delta(f))) \text{ mod } \hbar = \delta(f \text{ mod } \hbar),\]
for all elements $f \in Y(\gg)$.  Here $\sigma$ is the braiding automorphism of $Y(\gg)^{\otimes 2}$ (i.e. $\sigma(f \otimes g) = g \otimes f$).
\end{definition}

\begin{remark}
We can equivalently identify the classical limit $Y(\gg) \otimes_{\CC[[\hbar]]} \CC$ of the Yangian with the algebra of functions $\OO(G_1[[z^{-1}]])$ on the Poisson Lie group.  As a pro-vector space, or equivalently as a graded vector space where $z$ has degree 1, we can identify
\begin{align*}
Y(\gg) \otimes_{\CC[[\hbar]]} \CC &\iso U(\gg[z]) \\
&\iso \sym(\gg[z]) \\
&\iso \sym(\gg_0[[z^{-1}]]^\vee) \\
&\iso \OO(\gg_0[[z^{-1}]]) \iso \OO(G_1[[z^{-1}]]).
\end{align*}
\end{remark}

There is a very extensive literature on the Yangian and related quantum groups.  The general theory for quantization of Poisson Lie groups, including the Yangian, was developed by Etingof and Kazhdan \cite{EtingofKazhdanIII}.  For more information we refer the reader to Chari and Pressley \cite{ChariPressley}, or for the Yangian specifically to the concise introduction in \cite[Section 9]{CostelloYangian}.

Likewise, we can study deformation quantization for the algebra $\OO(\mhiggs^\fr_G(\bb{CP}^1,D,\omega^\vee))$ of functions on our symplectic moduli space.  This moduli space is, in particular, a smooth finite-dimensional Poisson manifold, so its algebra of functions can be quantized, for instance using Kontsevich's results on formality \cite{KontsevichQuantization}.

\begin{definition}
The quantum algebra of functions $\OO_\hbar(\mhiggs^\fr_G(\bb{CP}^1,D,\omega^\vee))$ is a choice of deformation quantization of the Poisson algebra $\OO(\mhiggs^\fr_G(\bb{CP}^1,D,\omega^\vee))$.  That is, an associative $\CC[[\hbar]]$-algebra where the antisymmetrization of the first order term in $\hbar$ of the product recovers the Poisson bracket. 
\end{definition}

While we have a Poisson morphism $\mhiggs^\fr_G(\bb{CP}^1,D,\omega^\vee) \to G_1[[z^{-1}]]$ by Theorem \ref{theorem:symplectic_leaf}, and therefore a map of Poisson algebras $\OO(G_1[[z^{-1}]]) \to \OO(\mhiggs^\fr_G(\bb{CP}^1,D,\omega^\vee))$ there's no automatic guarantee that we can choose a quantization of the target admitting an algebra map from the Yangian quantizing this Poisson map -- one would need to verify the absence of an anomaly obstructing quantization.  There is however a natural model for this quantization, to a $Y(\gg)$-module,  constructed by Gerasimov, Kharchev, Lebedev and Oblezin \cite{GKLO} (extending an earlier calculation \cite{GKL} for $\gg = \gl_n$).

\begin{theorem}
For any reductive group $G$, and any choice of local singularity data, the Poisson map $\OO(G_1[[z^{-1}]]) \to \OO(\mhiggs^\fr_G(\bb{CP}^1,D,\omega^\vee))$ quantizes to a $Y(\gg)$-module structure on an algebra $\OO_\hbar(\mhiggs^\fr_G(\bb{CP}^1,D,\omega^\vee))$ quantizing the algebra of functions on the multiplicative Higgs moduli space.  
\end{theorem}

This theorem is a direct consequence of the Gerasimov, Kharchev, Lebedev and Oblezin (GKLO) construction of $Y(\gg)$-modules whose classical limits are symplectic leaves in the rational Poisson Lie group.  They constructed these representations for any semisimple Lie algebra $\gg$; they are indexed by the data of a dominant coweight-coloured divisor $(D, \omega^\vee)$.  As in Shapiro, this data is encoded in the form of a set $\nu_{i,k}$ of complex numbers, where $i$ varies over simple roots of $\gg$: the coweight at a point $z \in D$ is encoded by the vector $(m_1, \ldots, m_r)$ where $m_i = \lvert\{\nu_{i,k} = z\}\rvert$.  Let us denote this module by $M_{D,\omega^\vee}$.

GKLO discussed the classical limits of these modules $M_{D,\omega^\vee}$ as symplectic subvarieties of the Poisson Lie group.  This is discussed with further detail in \cite[Section 4]{Shapiro}.  In particular, one can compare the Poisson Lie group with the Zastava space for the group $G$.  Recall that  moduli spaces of (non-periodic) monopoles on $\CC \times \RR$ can be modelled by spaces of framed maps from $\bb{CP}^1$ to the flag variety $G/B$, sending $\infty$ to the point $[B]$ (this is Jarvis's description of monopoles via scattering data \cite{Jarvis}).  The space $\mr{Map}_{\alpha^\vee}^\fr(\bb{CP}^1, G/B)$ of monopoles of degree $\alpha^\vee$ (where $\alpha^\vee$ is a positive coroot, identified as a class in $\mr H_2(G/B;\ZZ)$) admits a compactification known as \emph{Zastava space}.  Concretely, this is a stratified algebraic variety, with stratification given by
\[Z_{\alpha^\vee}(G) = \coprod_{\beta^\vee \preceq \alpha^\vee} \mr{Map}_{\alpha^\vee}^\fr(\bb{CP}^1, G/B) \times \sym^{\alpha^\vee - \beta^\vee}(\CC),\]
where the factor $\sym^{\alpha^\vee - \beta^\vee}(\CC)$ represents the space of coloured divisors $(D, \omega^\vee)$ with $\sum_{z_i \in D} \omega^\vee_{z_i} = \alpha^\vee - \beta^\vee$.  Zastava spaces were introduced by Finkelberg and Mirkovi{\'c} in \cite{FinkelbergMirkovicI}, based on ideas of Drinfeld and Lusztig.  See also \cite{BravermanICM}.

Using a birational map to this Zastava space, GKLO showed that the union of symplectic leaves corresponding to representations of a given type (meaning a given sum $\sum \omega^\vee_{z_i}$) is equivalent to the union of the classical leaves of that given type.  It is, therefore, reasonable to conjecture that the decompositions of these two varieties into their individual symplectic leaves (i.e. into decompositions of $\sum \omega^\vee_{z_i}$ labelled by points in $\CC$) should coincide.  We should note that this doesn't seem to follow directly from the GKLO calculations.  They analyzed the symplectic leaves corresponding to the classical limits of $Y(\gg)$-modules via a rational map, denoted by $e_i(z) = b_i(z)/a_i(z)$, defined on an open set of this symplectic leaf.  This rational map is, by construction, not sensitive to the positions of the singularities, only of the total type of a representation.

Nevertheless, the symplectic leaves obtained as classical limits of the $Y(\gg)$-modules $M_{D,\omega^\vee}$ \emph{are} symplectic subvarieties of the Poisson Lie group sweeping out the union of symplectic leaves of a given type.  In particular every symplectic leaf $\mhiggs^\fr_G(\bb{CP}^1,D,\omega^\vee)$ quantizes to some module within the GKLO classification, as claimed.

\begin{example}
One example is given by the case where the only pole lies at $0 \in \CC$, so the map to the Poisson Lie group $G_1[[z^{-1}]]$ factors through the polynomial group $G_1[z^{-1}]$.  This example is included in the work of Kamnitzer, Webster, Weekes and Yacobi \cite{KWWY}.  They calculate the quantization of slices in the thick affine Grassmannian $G(\!(t^{-1})\!)/G[t]$.  The thick affine Grassmannian has an open cell isomorphic to $G_1[[z^{-1}]]$.  For each dominant coweight $\omega^\vee$ there is a slice defining a symplectic leaf in this open cell: it's exactly the leaf corresponding to multiplicative Higgs fields with a single pole at 0 with degree $\omega^\vee$.  Kamnitzer et al quantize this slice (in particular; they also quantize slices through to the other cells in the thick affine Grassmannian) to a $Y(\gg)$-module of GKLO type. 
\end{example}

\begin{remark}
This result should be compared to the conjecture made in \cite[Chapter 8.1]{NekrasovPestun}.  The Poisson Lie group $G_1[[z^{-1}]]$ receives a Poisson map from the full moduli space $\mhiggs^{\text{fr,sing}}(\bb{CP}^1)$ of multiplicative Higgs bundles with arbitrary singularities.  Upon deformation quantization therefore, the quantized algebra of functions on this moduli space is closely related to the Yangian.
\end{remark}

\begin{remark}[$q$-Opers and Quantization]
Finally, let us refer back to Remark \ref{q_opers_remark} and discuss the brane of $q$-opers, and the associated structures that should arise after quantization.  Recall that the multiplicative Hitchin system has a natural section defined by post-composition with the Steinberg section $T/W \to G/G$.  The moduli space of $q$-opers is the subspace of $\qconn_G(C,D,\omega^\vee)$ defined to be the multiplicative Hitchin section after rotating to the point $q$ in the twistor sphere of complex structures.  A difference version of the moduli space of opers was studied in the work of Mukhin and Varchenko \cite{MukhinVarchenko}; it would be valuable to compare this definition, directly in the language of difference operators, with our approach to opers via the multiplicative Hitchin section, paralleling the work of Beilinson and Drinfeld in the additive case \cite[Section 3.1]{BDHitchin}.

Now, we expect that the space of $q$-opers will, in the rational case, be Lagrangian, and therefore will admit a canonical $\bb P_0$-structure -- a $-1$-shifted Poisson structure -- on its derived algebra of functions.  It would be interesting to study the quantization of this $\bb P_0$-algebra, and to investigate its relationship with the $q$-W algebras of Sevostyanov \cite{Knight,STSSevostyanov, Sevostyanov}, of Aganagic-Frenkel-Okounkov \cite{AFO}, and of Avan-Frappat-Ragoucy \cite{AvanFrappatRagoucy}.  We'll now proceed to discuss some first steps in this direction.  Preliminary results were announced by the second author at String-Math 2017 \cite{PestunStringMath}.  
\end{remark}

\part{Hyperk\"ahler Structures} \label{part2}

\section{Periodic Monopoles} \label{periodic_monopole_section}
\subsection{Monopole Definitions}
Moduli spaces of $q$-connections on a Riemann surface $C$ are closely related to moduli spaces of periodic monopoles, i.e. monopoles on 3-manifolds that fiber over the circle (more specifically, with fiber $C$ and monodromy determined by $q$).  Let $G_\RR$ be a compact Lie group whose complexification is $G$.  The discussion in this section will mostly follow that of \cite{CharbonneauHurtubise, Smith}.

Write $M = C\times_q S^1_R$ for the $C$-bundle over $S^1$ with monodromy given by the automorphism $q$.  More precisely, $M$ is the Riemannian 3-manifold obtained by gluing the ends of the product $C \times [0,2\pi R]$ of Riemannian manifolds by the isometry $(x,2\pi R) \sim (q(x), 0)$.

\begin{definition}
A \emph{monopole} on the Riemannian 3-manifold $M = C \times_q S^1_R$ is a smooth principal $G_\RR$-bundle $\bo P$ equipped with a connection $A$ and a section $\Phi$ of the associated bundle $\gg_{\bo P}$ satisfying the Bogomolny equation 
\[\ast F_A = \d_A \Phi.\]
\end{definition}

\begin{remark}
We should emphasise the difference between the Riemannian 3-manifold $M = C \times_q S^1_R$ appearing in this section and the derived stack $C \times_q S^1_B$ (the mapping torus) appearing in the previous section.  These should be thought of as smooth and algebraic realizations of the same object (justified by the comparison Theorem \ref{monopole_qconn_comparison_thm}) but they are a priori defined in different mathematical contexts.
\end{remark}

We can rephrase the data of a monopole on $M$ as follows.  Let $C_0 = C \times \{0\}$ be the fiber over $0$ in $S^1$, viewed as a Riemann surface.  Let $P$ be the restriction of the complexified bundle $\bo P_\CC$ to $C_0$.  Consider first the restriction of the complexification of $A$ to a connection $A_{\mr{hol}}$ on $P$ over $C_0$.  The $(0,1)$ part of $A_{\mr{hol}}$ automatically defines a holomorphic structure on $P$.  We can introduce an additional piece of structure on this holomorphic $G$-bundle.  In order to do so we can decompose the Bogomolny equation into one real and one complex equation as follows.
\begin{align}
F_{A_{0,1}} - \nabla_t \Phi \dvol_{C_0} &= 0 \label{Bogomolny_equation_real} \\
[\ol{\del}_{A_{0,1}}, \nabla_t - i\Phi \d t] &= 0 \label{Bogomolny_equation_complex}
\end{align}
where $\nabla_t$ is the component of the covariant derivative $\d_A$ normal to $C_0$.  

\begin{remark}
Note that the complex equation only depends on a complex structure on the curve $C$, and only the real equation depends on the full Riemannian metric.
\end{remark}

\begin{definition} 
From now on we'll use the notation $\mc A$ for the combination $\nabla_t - i\Phi \d t$: an element of the space $\Omega^0(M, \gg_P)\d t$ of sections of the complex vector bundle $\gg_P$ which is holomorphic after restriction to the curve $C_0$. 
\end{definition}

Let us now introduce singularities into the story.  We'll keep the description brief, referring the reader to \cite{CharbonneauHurtubise, Smith} for details.
\begin{definition}
Let $D \sub M$ be a finite subset.  Let $\omega^\vee$ be a choice of coweight for $G$.  A monopole on $M \bs D$ has \emph{Dirac singularity} at $z \in D$ with charge $\omega^\vee$ if locally on a neighborhood of $z$ in $M$ it is obtained by pulling back under $\omega^\vee$ the standard Dirac monopole solution to the Bogomolny equation, where $\Phi$ is spherically symmetric with a simple pole at $z$, and the restriction of a connection $A$ to a two-sphere $S^2$ enclosing the singularity defines a $U(1)$ bundle on this $S^2$ of degree $1$ so that
    \[\frac{1}{2\pi} \int_{S^2} F = 1 .\]
  See e.g. \cite[Section 2.2]{CharbonneauHurtubise} for a more detailed description.
\end{definition}

We can also introduce a framing (or a reduction of structure group as in the trigonometric example, though we won't consider the latter in this paper).  As usual let $c \in C$ be a point fixed by the automorphism $q$.

\begin{definition}
A monopole on $M$ with \emph{framing} at the point $c \in C$ is a monopole $(\bo P,A,\Phi)$ on $M$ (possibly with Dirac singularities at $D$) along with a trivialization of the restriction of $\bo P$ to the circle $\{c\} \times S^1_R$, with the condition that the holonomy of $\mc A$ around this circle lies in a fixed conjugacy class $f \in G/G$.
\end{definition}

The moduli theory of monopoles on general compact 3-manifolds was described by Pauly \cite{Pauly}.  In this paper we'll be interested in moduli spaces $\mon_G(M, D, \omega^\vee)$ of monopoles on 3-manifolds of the form $M = C \times S^1$, with prescribed Dirac singularities and possibly with a fixed framing at a point in $C$.  

\subsection{Hyperk\"ahler Structures on Periodic Monopoles} \label{HK_monopole_section}
Let us now describe the holomorphic symplectic and hyperk\"ahler structure on the moduli space of periodic monopoles.  We will first recall the idea that the moduli space of periodic monopoles can be described as a hyperk\"ahler quotient as in the work of Atiyah and Hitchin \cite{AtiyahHitchin}.  In a context with more general boundary data at $\infty$ this was demonstrated by Cherkis and Kapustin \cite{CherkisKapustin3} for the group $\SU(2)$, see also Foscolo \cite{FoscoloDef}[Theorem 7.12].  In the case of $\bb{CP}^1$ with a fixed framing the analysis is much easier.

\begin{remark}
In this subsection we'll define a holomorphic symplectic pairing on the moduli space of periodic monopoles, but we won't provide a proof that the pairing is closed and non-degenerate.  This will follow from Theorem \ref{symplectic_comparison_thm} below: where we prove that the holomorphic symplectic structure on periodic monopoles is equivalent to the symplectic structure we constructed in Section \ref{symp_section}.
\end{remark}

Consider the 3-manifold $M = \bb{CP}^1 \times S^1_R$.  Let us fix Dirac singularities at a divisor $D'$, and a regular semisimple framing $g_\infty$ at $\infty$. 

\begin{definition}
Let $\mc V$ be the infinite-dimensional vector space of pairs $(A,\Phi)$ where $A$ is a connection on a $G_\RR$-bundle $\bo P$ on $M \bs D'$, $\Phi$ is section of the adjoint bundle $(\gg_\RR)_{\bo P}$, and where 
\begin{enumerate}
\item The pairs $(A,\Phi)$ have a Dirac singularity with charge $\omega^\vee_{z_i}$ at each $(z_i,t_i)$ in $D'$.
 \item The holonomy of $A - i \Phi \d t$ around the circle $\{\infty\} \times S^1$ at infinity is equal to $g_\infty \in G$. 
\end{enumerate}
  Let $\mc G$ be the group of sections of $\bo P$ with value 1 on $\{\infty\} \times S^1$, acting on $\mc V$.  Let $T$ be the group of constant functions on $M$ valued in the centralizer of the element $g_\infty \in G$, also acting on $\mc V$.
\end{definition}

The vector space $\mc V$ admits a holomorphic symplectic structure using the complex structure on $\bb{CP}^1$.  We write $A$ in coordinates as $A_{0,1} + A_{1,0}  + A_t \d t$, and set $\mc A = (A_t + i \Phi) \d t$. One defines the holomorphic symplectic pairing on tangent vectors as
\begin{equation}
 \label{holo_symp_structure_fields}
 \omega_{\mr{hol}}((\delta A,\delta \Phi), (\delta A',\delta \Phi')) = \int_{\CC \times S^1_R} \kappa(\delta A_{0,1}, \delta \mc A') - \kappa(\delta \mc A, \delta A_{0,1}') \d z.
\end{equation}

The holomorphic moment map is defined to be the complex part of the Bogomolny equations \ref{Bogomolny_equation_complex}.  That is, we define
\begin{align*}
\mu_\CC \colon \mc V &\to \mr{Lie}(\mc G) \otimes_\RR \CC \\
(A,\Phi) &\mapsto [\ol{\del}_{A_{0,1}}, \dd_t + \mc A].
\end{align*}

\begin{lemma}
  The moduli space $\mon^\fr_G(\bb{CP}^1 \times S^1_R, D', \omega^\vee)$ of framed monopoles on $\bb{CP}^1 \times S^1_R$ is equivalent to the holomorphic symplectic reduction $\mu_\CC^{-1}(0)/\mc G_\CC$, and therefore inherits a holomorphic symplectic pairing from the formula (\ref{holo_symp_structure_fields}). 
\end{lemma}

This holomorphic symplectic structure has been studied in the literature in an alternative guise: that of the moduli space of $G_\RR$-instantons on a Calabi-Yau surface, from which our moduli space can be recovered under dimensional reduction \cite{Mukai1,Mukai2,Bottacin1,Bottacin2, HurtubiseMarkman}.   

In order to obtain a hyperk\"ahler structure on our moduli space, we will have to consider a further $T$-reduction identical to the reduction from Section \ref{Reduced_section}.

\begin{definition} \label{reduced_fields_def}
  Fix a regular element $g_1 \in G$.  Let $\mc V_{\mr{red}}$ be the infinite-dimensional vector space of pairs $(A,\Phi)$ as above, satisfying the additional condition that there exists a neighbourhood $\{|z| > R\}$ of $\infty$ in $\bb{CP}^1$ on which the fields $A$ and $\Phi$ can be written in the form
  \begin{align*}
  A &=  \mr{Re}(g_1)z^{-1}\d t + a\\
  \Phi &= - \mr{Im}(g_1)z^{-1} + \phi 
  \end{align*}
  where the matrix coefficients of $a$ and $\phi$ in any faithful representation $\rho$ of $G$ have decay rate $O(|z|^{-1-\tau})$ for some $\tau > 0$, i.e. the first order coefficient of the Taylor expansion around $\infty$ should be fixed. 
\end{definition}

We can promote the holomorphic symplectic pairing on the reduced space $\mc V_{\mr{red}}$ to a hyperk\"ahler structure.  We choose coordinates $x,y$ on $\RR^2 \sub \bb{CP}^1$, and write $A = A_1 \d t + A_2 \d x + A_3 \d y$.  Then for $i=1,2,3$ we define a pairing on the space of fields by
\begin{equation}
\label{hyperkahler_structure_fields}
\omega_i((\delta A,\delta \Phi),(\delta A',\delta \Phi')) = \int_{\RR^2 \times_\eps S^1_R} \left(\kappa(\delta A_i,\delta \Phi') - \kappa(\delta \Phi, \delta A_i') + \sum_{j,k=1,2,3} \eps_{ijk} \kappa(\delta A_j,\delta A'_k)\right) \d x \d y \d t.
\end{equation}
  
When we form the complex-valued pairing $\omega_2 + i \omega_3$, we recover the holomorphic symplectic pairing of (\ref{holo_symp_structure_fields}).  We need to pass to the reduced space of fields from Definition \ref{reduced_fields_def} for this hyperk\"ahler pairing to be well-defined -- on the full space $\mc V$ the pairing $\omega_1$ will generally involve a divergent integral. 

On the moduli space $\mc V_{\mr{red}}$ we can extend the holomorphic symplectic moment map to a hyperk\"ahler moment map, given by the Bogomolny functional for the flat metric on $\RR^2 \sub \bb{CP}^1$.  
\begin{align}
\mu_{\mr{HK}} \colon \mc V_{\mr{red}} &\to \mr{Lie}(\mc G) \otimes \RR^3 \\
(A,\Phi) &\mapsto \ast_{\mr{flat}} (F_A)|_{\RR^2 \times S^1}- \d_A \Phi|_{\RR^2 \times S^1}, \nonumber
\end{align}
where the group $\mc G \otimes \RR^3$ is identified with the space of sections in $\Omega^1(M \bs (D'); \gg_\RR)$ taking value 1 on the circle $\{\infty\} \times S^1$ at infinity.
  
Let $T' \sub T$ be a subtorus of the centralizer of the framing value $g_\infty$ with the same rank as the semisimple part of $G$, and define $\mu_T = \res \varphi_{\rho_w,\infty}$ to be the $\mf t'^*$-valued moment map as in Lemma \ref{lemma:smallq}.  We also fix $c = \{\rho_w(g_1)\}$ to be the element of $\mf t' \iso \mf t'^*$ associated to $g_1$ in Definition \ref{reduced_fields_def}. 
  
\begin{lemma}
Suppose that the intersection of $T'$ and the centralizer of $g_1$ is trivial.  Then the hyperk\"ahler quotient $\mu_{\mr{HK}}^{-1}(0)/\mc G$ of $\mc V_{\mr{red}}$ is equivalent, as a holomorphic symplectic manifold, to the following Hamiltonian reduction of the moduli space of framed monopoles:
\begin{align*}
\mon^\mr{red}_G(\bb{CP}^1 \times S^1_R, D', \omega^\vee) &:= \mon^\fr_G(\bb{CP}^1 \times S^1_R, D', \omega^\vee) /\!/ T'\\
&= \mu_{T'}^{-1}(c)/T'.
\end{align*}
\end{lemma}

\begin{proof}
There is a map of holomorphic symplectic manifolds $\iota \colon \mu_{\mr{HK}}^{-1}(0)/\mc G \to \mon^\fr_G(\bb{CP}^1 \times S^1_R, D', \omega^\vee)$ induced by the inclusion $\mc V_{\mr{red}} \inj \mc V$.  This map lands in the locus $\mu^{-1}_{T'}(c)$ for the moment map of the Hamiltonian $T'$ action, i.e. the map sending $(A_{0,1},\mc A)$ to $\res \varphi_{\rho_w,\infty}(\mc A)$ by definition of the reduced space $\mc V_{\red}$ of fields. Therefore, there is a holomorphic symplectic composite map 
\[f \colon \mu_{\mr{HK}}^{-1}(0)/\mc G \to \mon^\red_G(\bb{CP}^1 \times S^1_R, D', \omega^\vee)\] 
between holomorphic symplectic manifolds of the same dimension.  By looking at the $z^{-1}$ terms in the Taylor expansion at $\infty$, we see that the map $\iota$ meets each $T'$ orbit in a single point, so $f$ is an isomorphism. 
\end{proof}

\subsection{Deformations of Periodic Monopoles} \label{monopole_def_section}
We will now proceed with an informal description of the deformation complex of the moduli space of periodic monopoles.  The tangent complex to this hyperk\"ahler quotient at a point $(\bo P, \mc A)$ can be written as $\Omega^0(M \!\bs\! D; (\gg_\RR)_{\bo P})[1] \to \bb T_{\mu^{-1}(0)}$ where $\bb T_{\mu^{-1}(0)}$ is the tangent complex to the zero locus of the moment map, concentrated in non-negative degrees. Roughly speaking $\bb T_{\mu^{-1}(0)} = \bb T_{\mc V} \overset {\d\mu} \to \Omega^1(M; (\gg_R)_{\bo P})[-1]$.  

More explicitly, and on $C \times S^1$ for a general Riemannian 2-manifold $C$, following \cite{FoscoloDef} define $\mc F^{\mr{mon}}_{\bo P, \mc A}$ to be the sheaf of cochain complexes
\[\left(\xymatrix{
\Omega^0(M \!\bs\! D; (\gg_\RR)_{\bo P}) \ar[r]^(.36){\d_1} &\Omega^1(M \!\bs\! D; (\gg_\RR)_{\bo P}) \oplus \Omega^0(M \!\bs\! D; (\gg_\RR)_{\bo P}) \ar[r]^(.64){\d_2} &\Omega^1(M \!\bs\! D; (\gg_\RR)_{\bo P})
}\right) \otimes_\RR \CC\]
placed in degrees $-1$, 0 and 1 where $\d_1(g) = -(\d_A(g),[\Phi, g])$ and $\d_2(a,\psi) = \ast \d_A(a) - \d_A(\psi) + [\Phi,a]$.  Write $\d_{\mr{mon}}$ for the total differential.
  
\begin{remark}
Here we have chosen a point in the twistor sphere, forgetting the hyperk\"ahler structure and retaining a holomorphic symplectic structure.  In other words we have identified the target of the hyperk\"ahler moment map -- the space of imaginary quaternions -- with $\RR \oplus \CC$, which is equivalent to choosing a point in the unit sphere of the imaginary quaternions: the twistor sphere.  
\end{remark}

\begin{remark} \label{monopole_holo_restriction_rmk}
If we restrict the complexified complex $\mc F^{\mr{mon}}_{\bo P, \mc A} \otimes_\RR \CC$ to a slice $C_t = C \times \{t\}$ in the $t$-direction we can identify it with a complex of the form
\[\Omega^\bullet(C_t; \gg_P)[1] \overset {[\Phi,-]} \to \Omega^\bullet(C_t; \gg_P)\]
with total differential given by $\d_A$ on each of the two factors along with the differential $[\Phi,-]$ mixing the two factors. 

These two summands each split up into the sum of a Dolbeault complex on $C$ and its dual.  That is, there's a natural subcomplex of the form
\[\Omega^{0,\bullet}(C_t; \gg_P)[1] \overset {[\Phi,-]} \to i \Omega^{0,\bullet}(C_t; \gg_P) \d t\]
where the internal differentials on the two factors are now given by $\ol \dd_{A_{0,1}}$.  This complex is in turn quasi-isomorphic to the complex
\[\Omega^\bullet(S^1; \Omega^{0,\bullet}(C_t;\gg_P))[1]\]
with total differential $\ol \dd_{A_{0,1}} + \d_{\mc A}$.
\end{remark}

\begin{remark}
If one introduces a framing at a point $c \in C$ then we must correspondingly twist the complex $\mc F^{\mr{mon}}$ above by the line bundle $\OO(c)$ on $C$ -- i.e. we restrict to sections that vanish at the framing point.  So in that case we define
\[\mc F^{\text{mon,fr}}_{(\bo P,\mc A)} = \mc F^{\mr{mon}}_{\bo P, \mc A} \otimes (\CC_{S^1} \boxtimes \OO(c)).\]
\end{remark}

The following is proved in \cite{FoscoloDef}.
\begin{prop}
The tangent space of $\mon_G(S^1 \times C, D', \omega^\vee)$ at the point $(\bo P,\mc A)$ is quasi-isomorphic to the hypercohomology $\bb H^0(C \times S^1; \mc F'_{(\bo P,\mc A)})$ of a subsheaf $\mc F' \sub \mc F^{\mr{mon}}$ where growth conditions are imposed on the degree 0 part of $\mc F^{\mr{mon}}$ near the singularities.
\end{prop}
  
\subsection{Periodic Monopoles and $q$-Connections}
In this section we will begin to address the relationship between periodic monopoles and $q$-connections.  We will first recall the comparison theorem between multiplicative Higgs bundles and periodic monopoles proved by Charbonneau--Hurtubise \cite{CharbonneauHurtubise} for $\GL_n$, and Smith \cite{Smith} for general $G$. This approach was first suggested by Kapustin-Cherkis \cite{CherkisKapustin2} under the name `spectral data'. 

\begin{definition}
Write $D'$ for the finite subset $\{(z_1,t_1), \ldots, (z_k, t_k)\}$ of $D \times S^1$, where $D = \{z_1, \ldots, z_k\}$ is finite subset of $C$ and $0 < t_1 < t_2 < \cdots < t_k < 2\pi$ is a sequence of points in $S^1$ satisfying the condition of Remark \ref{t_stability_remark}.
\end{definition}

\begin{theorem}[Charbonneau--Hurtubise, Smith] \label{CHS_thm}
There is an analytic isomorphism between the moduli space of polystable monopoles on $C \times S^1$ with Dirac singularities at $D'$ (and a possible framing on $\{c\} \times S^1$) and the moduli space of multiplicative Higgs bundles on $C$ with singularities at $D$ and framing at $\{c\}$.  More precisely there is an analytic isomorphism
\[H \colon \mon^{(\fr)}_G(C \times S^1, D', \omega^\vee) \to \mhiggs_G^{\text{ps,(fr)}}(C, D, \omega^\vee)\]
given by the holonomy map around $S^1$, i.e. sending a monopole $(\bo P, \mc A)$ to the holomorphic bundle $P = (\bo P_\CC)|_{C_0}$ with multiplicative Higgs field $g = \Hol_{S^1}(\mc A) \colon P \to P$.
\end{theorem}

We would like to generalize this theorem to cover the twisted product $C \times_q S^1$ given by an automorphism $q$ of the curve $C$.  The most important example -- for the purposes of the present paper -- is the translation automorphism of $\bb{CP}^1$ sending $z$ to $z + \eps$, fixing the framing point $\infty$.  This automorphism is an isometry for the flat metric on $\RR^2$ but \emph{not} for any smooth metric on the compactification  $\bb{CP}^1$, which means that the moduli space of monopoles on the $\eps$-family of twisted products is only defined for the flat metric, and therefore that we'll have to generalize slightly the theorem of Charbonneau--Hurtubise and Smith to include the flat metric, with its singularity at infinity.

\begin{remark}
The moduli space of periodic monopoles on $\RR^2 \times S^1$ specifically has been studied in the mathematics literature by Foscolo \cite{FoscoloDef} , applying the analytic techniques of deformation theory to earlier work on periodic monopoles by Cherkis and Kapustin \cite{CherkisKapustin1, CherkisKapustin2}. This analysis considers a less restrictive boundary condition at infinity in $\RR^2$ than a framing -- the Higgs field may admit a regular singularity at infinity -- and therefore requires more sophisticated analysis than we'll need to consider in the present paper.
\end{remark}

\begin{definition} \label{monopole_moduli_def}
Let $D$ be a finite subset $\{(z_1,t_1), \ldots, (z_k, t_k)\}$ of points in $\RR^2 \times_\eps S^1_R$, and let $\omega^\vee_{i}$ be a choice of coweight for each point in $D$. The moduli space $\mon^\fr_G(\RR^2 \times_\eps S^1_R, D, \omega^\vee)$ is the moduli space of smooth principal $G_\RR$-bundles $\bo P$ on $\bb{CP}^1 \times_\eps S^1_R$ equipped with a connection $A$ and a section $\Phi$ of the associated bundle $\gg_{\bo P}$, so that the restriction of $(A,\Phi)$ to the open subset $\CC \times_\eps S^1_R$ satisfyies the Bogomolny equation for the flat metric.
\end{definition}

Now, in the case of the flat metric we can generalize Theorem \ref{CHS_thm} to include the $\eps$-deformation.  To do so we'll crucially use results of Mochizuki \cite{MochizukiKH} following earlier work of Biquard and Jardim \cite{BiquardJardim}.

\begin{theorem} \label{monopole_qconn_comparison_thm}
There is an isomorphism between the moduli space of polystable monopoles on $\RR^2 \times_\eps S^1$ with Dirac singularities at $D'$ and a regular semisimple framing on $\{\infty\} \times S^1$, and the moduli space of polystable $\eps$-connections on $\bb{CP}^1$ with singularities at $D$ and framing at $\{\infty\}$.  More precisely there is an analytic isomorphism
\[H \colon \mon^{\fr}_G(\bb{CP}^1 \times_\eps S^1, D', \omega^\vee) \to \epsconn_G^{\text{ps,fr}}(\bb{CP}^1, D, \omega^\vee)\]
given by the holonomy map around $S^1$, i.e. sending a monopole $(\bo P, \mc A)$ to the holomorphic bundle $P = (\bo P_\CC)|_{\bb{CP}^1 \times \{0\}}$ with $\eps$-connection $g = \Hol_{S^1}(\mc A) \colon P \to \eps^*(P)$.
\end{theorem}

\begin{proof}
First we'll establish injectivity.  Let $(\bo P, \mc A)$ and $(\bo P', \mc A')$ be a pair of periodic monopoles on $\bb{CP}^1 \times_\eps S^1$ with images $(P,g)$ and $(P', g')$ respectively, and choose a bundle isomorphism $\tau \colon P \to P'$ intertwining the $\eps$-connections $g$ and $g'$.  One observes first that $\bo P$ and $\bo P'$ are also isomorphic $G$-bundles since, firstly, we have an isomorphism $\bo P|_{\bb{CP}^1 \times \{t_0\}} \to \eps^*\bo P'|_{\bb{CP}^1 \times \{t_0\}}$ at a value $t_0$ not equal to any of $t_1, \ldots, t_k$, and because the holomorphic structure on $\bo P|_{\bb{CP}^1 \times \{t\}}$ varies smoothly in $t$ and the moduli space of holomorphic $G$-bundles on $\bb{CP}^1$ is discrete, we correspondingly have an isomorphism $\bo P|_{\bb{CP}^1 \times \{t\}} \to \bo P'|_{\bb{CP}^1 \times \{t\}}$ for every $t \in S^1$.  To match the monopole structures, we use the same argument as in \cite[Proposition 5.6]{Smith}.

For surjectivity, we again use the same argument as Charbonneau--Hurtubise and Smith, but with an important modification.  Charbonneau and Hurtubise use a crucial result of Simpson \cite[Theorem 1]{Simpson} on the existence of Hermitian Yang-Mills metrics on surfaces.  Simpson's argument only applies to K\"ahler metrics with finite volume, which isn't the case here.  We can, however, replace Simpson's theorem with a theorem of Mochizuki \cite[Corollary 3.13]{MochizukiKH} in the case of the curve $\CC$.

Now, let us consider surjectivity.  We extend the argument of Charbonneau--Hurtubise and Smith in two steps, in order to account for the two new subtleties described above.  We begin by extending a holomorphic $G$-bundle $P$ on $\bb{CP}^1 \times \{0\}$ with $\eps$-connection $g$ to a $G$-bundle on $M \bs (D') = (\bb{CP}^1 \times_\eps S^1_R) \bs (D')$.  Let $\wt M$ be the 3-manifold
\[\wt M = ((-2\pi R, 2\pi R) \times \bb{CP}^1) \bs \bigcup_{j=1}^k (A^+_j \cup A^-_j)\]
where $A^+_j$ is the line $\{(t+ t_0,z_j + t \eps/{2\pi R}) \colon t \in (0, 2\pi R - t_0]\}$ and $A^-_j$ is the line $\{(t + t_0 - t \pi R,z_j + 4 \eps/{2\pi R}) \colon t \in [2\pi R-t_0, 2 \pi R)\}$.
Let $\pi \colon \wt M \to \bb{CP}^1$ be the projection sending $(t,z)$ to $z + t\eps/2 \pi R$. The bundle $P$ pulls back to a bundle $\pi^*(P)$ on $\wt M$. We obtain a bundle on $M \bs (D \times t_0)$ by applying the identification $(t,z) \sim (t - 2 \pi R, q(z))$. This bundle extends to an $S^1$-invariant holomorphic $G$-bundle on $M \times S^1$. The remainder of the proof -- verifying the existence of the monopole structure associated to an appropriate choice of hermitian structure -- consists of local analysis which is independent of the value of the parameter $\eps$. 

In order to construct the monopole, we will apply Mochizuki's theorem.  It is necessary to verify that the assumptions of \cite[Corollary 3.13]{MochizukiKH} hold.  We must choose a hermitian metric on the vector bundle associated to the holomorphic $G$-bundle constructed above under a faithful representation $\rho$, so that the holonomy around $S^1$ describes the $\eps$-connection $g$, so that near each singularity the corresponding Chern connection described a Dirac singularity of the specified charge, and so that the Taylor expansion at infinity takes the form $g(z) = g_\infty +  g_1/z + g_2/\bar{z} + g_3/z \bar{z} + O(z^{-3})$ -- i.e. so that it has the appropriate framing, and so that its Laplacian has leading term of order $|z|^{-4}$. 

The result of this choice will then be that the quantity $|\Lambda F(h_0)|$ -- where the Lefschetz operator
\[\Lambda \colon \Omega^{2}(C;\gg_P) \to \Omega^{0}(C; \gg_P)\]
is the adjoint to wedging with the K\"ahler form -- is bounded in a neighbourhood of infinity by $(1 +\|z\|^2)^{-2}$, as required by Mochizuki.  We build our metric using a partition of unity, so away from a neighbourhood of $\infty$, the metric coincides with the metric described by Charbonneau and Hurtubise.  On a neighbourhood of infinity, first let us eliminate the $\eps$ parameter by choosing an analytic isomorphism between $D \times_\eps S^1_R \times S^1$ and $D \times S^1_R \times S^1$, where $D$ is a disk around infinity.  We can describe a hermitian metric on $D^\times \times S^1_R \times S^1$, constant in the second $S^1$ factor, by the formula
\[h_0(z,t) = (\rho(g(z))^\dagger)^{t/2\pi R}\rho(g(z))^{t/2\pi R}.\]

Now, to verify the bound on the growth of $|\Lambda F(h_0)|$, we'll use the fact that our framing $g_\infty$ at infinity was regular semisimple.  Because the regular semisimple locus is open in $G$, the group element $g(z)$ is regular semisimple everywhere in a sufficiently small disk around infinity.  We can, therefore, choose a local gauge transformation on the disk making $g(z)$ valued in the maximal torus $T$ centralizing $g_\infty$ in this neighbourhood.  If $g(z)$ is $T$-valued then $\Lambda F(h_0)$ vanishes, since in the abelian case $\log h_0$ takes the form
\[\log(h_0)(z,t) = \frac t{2\pi R}(\log f(\ol z) + \log f(z)),\]
for a holomorphic function $f$; if we write $s$ for the coordinate on the other copy of $S^1$, and $w = t+is$, then the operator $\frac{\dd^2}{\dd z \dd \ol z} + \frac{\dd^2}{\dd w \ol \dd w}$ annihilates $\log(h_0)$.  In particular, $\Lambda F(h_0)$ has the required decay condition.  If $A$ and $A'$ are related to one another by an isomorphism of holomorphic vector bundles on $D \times T^2$ then the functions $|\Lambda F(A)|$ and $|\Lambda F(A')|$ are equal, so the bound on $T$-valued functions implies the same bound on regular semisimple $G$-valued functions.

We must also verify analytic stability of the associated vector bundle to our $G$-bundle upon a choice of faithful representation.  To do this, we'll use Charbonneau and Hurtubise's result \cite[Lemma 4.5]{CharbonneauHurtubise}, which says that Mochizuki's definition of degree coincides with the definition from Section \ref{stability_section}.  That is, since $|\Lambda F(h_0)|$ is integrable in a neighbourhood of infinity, we can identify its integral with the algebraic definition of the degree.  Therefore, since our bundle is polystable in the sense of Section \ref{stability_section}, it is analytically polystable, and therefore we can apply Mochizuki's result.

With this established, we can apply the proof of \cite[Proposition 5.2]{Smith} using Mochizuki's theorem in place of Simpson's.  The local analysis at the singularities is independent of the value of $\eps$.
\end{proof}

\begin{remark}
In \cite{Mochizuki}, Mochizuki establishes a more general result, where rather than a framing at infinity one allows a regular singularity -- such monopoles are referred to as ``Generalized Cherkis-Kapustin (GCK) monopoles''.
\end{remark}

\section{Comparison of Symplectic Structures} \label{symp_comparison_section}

In this section, we'll show that the isomorphism of Theorem \ref{monopole_qconn_comparison_thm} provides an equivalence of holomorphic symplectic moduli spaces, where on the multiplicative Higgs side we have a holomorphic (indeed, algebraic) symplectic structure given in Section \ref{symplectic_leaf_section} (specifically in equation (\ref{eq:Omega})), and on the monopole side we have a holomorphic symplectic form of AKSZ type discussed in Section \ref{HK_monopole_section} (specifically in equation (\ref{holo_symp_structure_fields})).

We'll begin by describing the derivative of the holonomy map $H \colon \mon_G^\fr(\bb{CP}^1 \times S^1, D, \omega^\vee) \to \mhiggs^\fr_G(\bb{CP}^1, D, \omega^\vee)$. Let $\{U_0, U_1, \ldots, U_k, U_\infty\}$ be the open cover of $\bb{CP}^1$ defined in Remark \ref{atlas_remark}.  Given a tangent vector $(\delta A_{0,1}, \delta \mc A)$ to the moduli space of periodic monopoles at a point $(A_{0,1},\mc A)$, let $(\delta A_{0,1}, \delta \mc A)_i$ denote its restriction to the open subset $U_i \times (0, 2\pi) \sub \bb{CP}^1 \times S^1$. 

Choose a local potential $b_i \in \Omega^0(U_i \times (0,2\pi); \gg_\RR)$ for the tangent vector $(\delta A_{0,1}, \delta \mc A)$ on this patch, so that
\begin{align*} 
(\delta A_{0,1}, \delta \mc A)_i &= \d_{\mr{mon}} b_i \\
&= (\ol \dd_{A_{0,1}} b_i, \mr d_{\mc A} b_i).
\end{align*}
Such a local potential exists because the curvature of the connection $\ol \dd_{A_{0,1}} + d_{\mc A}$ vanishes, because $(A_{0,1}, \mc A)$ solves the Bogomolny equations.  Below we'll use the notation
\[(X^L_i, X^R_i) = (\lim_{t \to 2\pi} b_i(t), - \lim_{t \to 0} b_i(t)).\]

\begin{lemma} \label{local_derivative_description_lemma}
The derivative $\d H$ of the holonomy map $H$ is given on an open patch $U_i \times (0,2\pi)$ by the formula
\begin{equation}
\label{local_dH_eq}
\d H(\delta A_{0,1}, \delta \mc A)_i = \d H(\d_{\mr{mon}} b_i) = (X^L_i, X^R_i).
\end{equation}
\end{lemma}

\begin{proof}
The value of the derivative $\d H$ at a $\d_{\mr{mon}}$-exact local vector field $\d_{\mr{mon}} b_i$ is given by the derivative of the action of the group of gauge transformations on the holonomy $H(\mc A)$ of $\mc A$ from $t=0$ to $2\pi$, that is, of the map 
\[B_i \mapsto B_i(2\pi)H(\mc A)B_i(0)^{-1},\]
where $B_i \in \Omega^0((U_i \times (0,2\pi)) \bs \{(z_i, t_i)\}; \gg_\RR)$.  This derivative evaluates to $b_i(2\pi) H(\mc A) - H(\mc A)b_i(0)$, which is identified as section of the bundle $\ad(g)$ with the right-hand side of expression \ref{local_dH_eq}. 
\end{proof}

\begin{remark}
For the moment, we will only compare holomorphic symplectic structures on multiplicative Higgs and periodic monopole moduli spaces.  As holomorphic symplectic manifolds, the moduli spaces of periodic monopoles only depend on a conformal class of metrics on the curve $C$.  In particular, we can either consider $\bb{CP}^1$ with the round metric -- where Charbonneau--Hurtubise and Smith's theorem applies -- or with the flat metric -- and use our Theorem \ref{monopole_qconn_comparison_thm} using Mochizuki's growth estimate theorem.  In Section \ref{hyperkahler_section} we will consider hyperk\"ahler structures: in that case we'll need to consider the flat moduli space, and the reduced space $\mc V_{\red}$ of fields.
\end{remark}

\begin{theorem} \label{symplectic_comparison_thm}
The symplectic structure on $\mon_G^\fr(\bb{CP}^1 \times S^1,D \times\{0\},\omega^\vee)$ and the pullback of the non-degenerate pairing on $\mhiggs_G^{\text{ps,fr}}(\bb{CP}^1,D,\omega^\vee)$ under the holonomy map $H$ coincide.
\end{theorem}

\begin{proof}
Let us begin with the symplectic structure $\omega_{\mr{hol}}$ from Equation \ref{holo_symp_structure_fields}.  We can write the integral as the sum over open charts $U_0 \times (0,2\pi), U_1 \times (0,2\pi), \ldots, U_k \times (0,2\pi), U_\infty \times (0,2\pi)$ in our atlas \footnote{More precisely, we should replace the definition of $U_0$ from Remark \ref{atlas_remark} with the complement of closed disks $\bb{CP}^1 \bs (\ol U_1 \cup \cdots \cup \ol U_k \cup \ol U_\infty)$, so that the $U_i$ are disjoint in $\bb{CP}^1$ with dense union.}.  We assume that the holomorphic structure defined by $\bar \partial_{A_{0,1}}$ on the restriction of the $G$-bundle to $\mathbb{CP}^{1} \times \{0 \}$ is trivial, and therefore, by a smooth gauge transformation we can gauge away deformations of $\bar \partial_{A_{0,1}}$ along the slice $\mathbb{CP}^1 \times \{0\}$. Consequently, when $t  = 0$ or $t = 2\pi$ we have $\bar \partial_{A_{0,1}} b = 0$, so that in each chart $U_i$ the pair $(X^{L}, X^{R})_i = (b_i(2 \pi), b_i(0))$ is holomorphic.

On each open set $U_i$, we can use Stokes' theorem to expand the summands of the symplectic form corresponding to charts.  That is, we have 
\begin{align*}
\omega_{\mr{hol}}(\{\d_{\mr{mon}}b_i\}, \{\d_{\mr{mon}}b_i'\}) &= \sum_{i=0,1,\ldots, k,\infty} \int_{U_i \times (0,2\pi)} \d z \, \kappa(\ol \dd_{A_{0,1}} b_i \wedge \d_{\mc A}b_i') - (b \leftrightarrow b') \\
&= \sum_{i=0,1,\ldots, k,\infty} \int_{U_i \times (0,2\pi)} \d z \, \kappa((\ol \dd_{A_{0,1}} + \d_{\mc A})b_i, (\ol \dd_{A_{0,1}} + \d_{\mc A})b_i') \\
&= \left(\sum_{i=0,1,\ldots, k,\infty} \int_{\dd U_i \times (0,2\pi)} \kappa (b_i, \d_{\mc A}b_i') \right) + \left(\sum_{i=0,1,\ldots, k,\infty} \int_{U_i}\kappa(\ol \dd_{A_{0,1}} b_i, b_i'(2\pi) - b'_i(0)) \right) \\
&= \sum_{i=1, \ldots, k, \infty} \int_{\dd U_i \times (0,2\pi)} \kappa (b_i - b_0, \d_{\mc A}b_i') 
\end{align*}
using the fact that $\d_{\mc A}b_i'$ and $\d_{\mc A}b_0'$ coincide on $\dd U_i \times (0,2\pi)$ to simplify the first summand, and the fact that $\ol \dd_{A_{0,1}} b_i(t) = 0$ when $t = 0$ or $t = 2\pi$ as remarked above to eliminate the second summand.

Now, taking the summands one-by-one, we can write
\begin{align}
 \int_{\dd U_i \times (0,2\pi)} \d z \, \kappa (b_i - b_0, \d_{\mc A}b_i')  &= \oint_{\dd U_i} \d z \, \left(\kappa(X_i^L - X_0^L,X_i'^L) - \kappa(X_i^R - X_0^R,X_i'^R)\right)  \nonumber\\
 &= \oint_{\dd U_i} \d z \, \left(-\kappa(X_0^L,X_i'^L) + \kappa(X_0^R,X_i'^R)\right), \label{eq:summand}
\end{align}
 where the $\kappa( X_i^{L,R}, X_i'^{L,R})$ terms vanish  because $X_i, X_i'$ are holomorphic in each $U_i$. 

To conclude, consider the symplectic pairing on the multiplicative Higgs moduli space (\ref{eq:Omega}), evaluated at a pair of tangent vectors of the form $\d H(\delta A_{0}, \delta \mc A)$ as in Lemma \ref{local_derivative_description_lemma}.  This takes the form
\begin{align*}
\Omega(\d H(\delta A_{0}, \delta \mc A), \d H(\delta A_{0}', \delta \mc A')) &= \sum_{i=1, \ldots, k, \infty} \frac 1{2\pi \imath} \oint_{\dd U_i} \d z - (\kappa(X^{L}_{i}, X^{L'}_{0}) + \kappa(X^{R}_{i}, X^{R'}_{0})).
\end{align*} 
This coincides, up to a global rescaling by $2 \pi \imath$, with the sum of the antisymmetric expressions (\ref{eq:summand}), as required.
\end{proof}

We conclude this section by describing the relationships between the the further Hamiltonian reduction by the global torus $T'$ on both sides.
\begin{corollary} \label{reduced_equiv_cor}
The equivalence of Theorem \ref{symplectic_comparison_thm} induces an equivalence of holomorphic symplectic manifolds
\[\mon_G^{\mr{red}}(\bb{CP}^1 \times S^1,D \times\{0\},\omega^\vee) \to \mhiggs_G^{\text{red}}(\bb{CP}^1,D,\omega^\vee),\]
by taking the Hamiltonian reduction by the constant $T'$ action on both sides, where $T'$ is the centralizer of the framing value $g_\infty$ in the semisimple part of the gauge group.  In particular, the holomorphic symplectic structure on the right-hand side is canonically promoted to a hyperk\"ahler structure.
\end{corollary}

\section{Comparison of Hyperk\"ahler Structures and Twistor Rotation} \label{hyperkahler_section}
The comparison theorem between multiplicative Higgs bundles and periodic monopoles \ref{reduced_equiv_cor} tells us that the symplectic structure on $\mhiggs_G^{\text{red}}(\bb{CP}^1,D,\omega^\vee)$ extends canonically to an $\mathbb{R}_{>0}$-family of hyperk\"ahler structures, depending on a parameter $R$, the radius of the circle.  In this section we'll compare the twistor rotation with the deformation to the moduli space of $\eps$-connections. We'll show that in the $R \to \infty$ limit, by passing to a point $\eps$ in the twistor sphere we obtain a complex manifold equivalent to the moduli space of $\eps$-connections as a deformation of the moduli space of multiplicative Higgs bundles.

Recall from Section \ref{HK_monopole_section} that we inherit a hyperk\"ahler structure on the reduced moduli space of periodic monopoles on $\bb{CP}^1 \times S^1_R$ from the hyperk\"ahler structure on the infinite-dimensional reduced space of fields \ref{hyperkahler_structure_fields}.  This description for the hyperk\"ahler structure also tells us what should happen when we perform a rotation in the twistor sphere.  The following statement follows by identifying the holomorphic symplectic structure at $\eps$ in the twistor sphere by applying the corresponding rotation in $\SO(3)$ to the coordinates $x$, $y$ and $t$ on $\RR^3$ in the expression for the holomorphic symplectic structure at $0$.  Note that when we calculate the symplectic form as an integral as above, it's enough to take the integral over just $\RR^2 \times S^1_R$ instead of $\bb{CP}^1 \times S^1_R$.

Notationally, from now on we'll write the holomorphic symplectic pairing in the abbreviated form $\kappa(A_{0,1}, \mc A') - \kappa(\mc A, A_{0,1}') = \kappa(\delta^{(1)} \mc A, \delta^{(2)} \mc A)$, and we'll write $\dvol$ for the volume form $\d x \wedge \d y \wedge \d t$ on $\CC \times (0,2\pi R)\sub \bb{CP}^1 \times S^1_R$.

\begin{prop}
Let $\zeta$ be a point in the twistor sphere.  The holomorphic symplectic form on the moduli space $\mon_G^{\fr}(\bb{CP}^1 \times S^1_R, D, \omega^\vee)$ of periodic monopoles is given by the formula 
\[\omega_\eps(\delta^{(1)}\mc A, \delta^{(2)}\mc A) = \int_{\RR^2 \times S^1_R} \kappa(\delta^{(1)} \mc A, \delta^{(2)} \mc A) \zeta(\dvol),\]
where we identify $\zeta$ with an element of $\SO(3)$ and apply this rotation to the volume form $\dvol$.  Equivalently we can identify $\RR^2 \times \bb{CP}^1$ with the quotient $\RR^3 \times L$ where $L$ is the rank one lattice $\{0\}^2 \times 2\pi R\ZZ$ and identify the rotated holomorphic symplectic form with 
\[\omega_\zeta(\delta^{(1)}\mc A, \delta^{(2)}\mc A) = \int_{\RR^3/(\zeta(L))} \kappa(\delta^{(1)} \mc A, \delta^{(2)} \mc A) \dvol.\]
\end{prop}

Now, we'll perform our identification of the twistor deformation on the multiplicative Higgs side by rewriting this twistor rotation in a somewhat different way in the large $R$ limit.  In this limit we can identify the twistor rotation with the deformation obtained by replacing the product $\bb{CP}^1 \times S^1_R$ with the product twisted by an automorphism of $\RR^2$.  More precisely, the rotated symplectic structure can then be described in the following way.

\begin{theorem} \label{HK_rotation_thm}
Let $\left(\mhiggs^{\text{red}}_G(\bb{CP}^1,D,\omega^\vee)\right)_{\zeta,R}$ be the family of complex manifolds obtained by pulling back the holomorphic symplectic structure $\omega_{\zeta,R}$ on $\mon_G^{\mr{red}}(\bb{CP}^1 \times S^1_R,D,\omega^\vee)$ at $\zeta$ in the twistor sphere to the moduli space of polystable framed multiplicative Higgs bundles.  In the limit where $R \to \infty$ with $2 \pi \zeta R = \eps$ fixed, the canonical smooth isomorphism
\[\left(\mhiggs^{\text{red}}_G(\bb{CP}^1,D,\omega^\vee)\right)_{\zeta,R} \to \epsconn^{\text{red}}_G(\bb{CP}^1,D,\omega^\vee)\]
becomes holomorphic symplectic, in the sense that the holomorphic symplectic structures converge pointwise over the tangent bundle to the moduli space of monopoles. 
\end{theorem}

\begin{proof}
First let us note that the 3-manifolds $\bb{CP}^1 \times_\eps S^1_R$ for varying values of $R$ and $\eps$ are all canonically isometric to untwisted products of the form $\bb{CP}^1 \times S^1_{R'}$.  Any two points in the $(\eps,R)$ family of 3-manifolds can therefore be identified by a canonical conformal transformation
\[f_{(\eps,R),(\eps',R')} \colon \bb{CP}^1 \times_\eps S^1_R \to \bb{CP}^1 \times_{\eps'} S^1_{R'}\]
This means that we can canonically smoothly identify the moduli spaces of monopoles
\[\mon_G^{\red}(\bb{CP}^1 \times_\eps S^1_R,D,\omega^\vee) \iso \mon_G^{\red}(\bb{CP}^1 \times_{\eps'} S^1_{R'},f_{(\eps,R),(\eps',R')}(D),\omega^\vee).\]

Now, according to Theorem \ref{monopole_qconn_comparison_thm} it's enough to show that the complex structure on $\mon_G^{\text{red}}(\bb{CP}^1 \times S^1_R,D,\omega^\vee)$ at the point $\zeta$ converges to the complex structure on the $\eps$-deformation $\mon_G^{\red}(\bb{CP}^1 \times_\eps S^1_R,D,\omega^\vee)$.  The key observation that we'll use is that, fixing $2 \pi \zeta R = \eps$ we can identify the rotated lattice with a \emph{sheared} lattice.  Namely, when we rotate the rank one lattice $L_{R,0} = (0,0,2\pi R)\ZZ$ by $\zeta$ we obtain the lattice
\[\zeta(L_{R,0}) = \frac {1}{1+|\zeta/2|^2} (\mr{Re}(\eps), \mr{Im}(\eps), 2\pi R (1-|\zeta/2|^2) )\ZZ.\]
We can view the quotient of $\RR^3$ by the rotated lattice as a sheared lattice with different radius and shear.  Specifically, we identify
\[\RR^3 / \zeta(L_{R,0}) \iso \RR^2 \times_{\frac{\eps}{1+|\zeta/2|^2}} S^1_{R \frac {1-|\zeta/2|^2}{1+|\zeta/2|^2}}.\]

In order to identify the desired complex structures it's enough to identify the full family of symplectic structures on the $\eps$-deformation with the appropriately rotated family of symplectic structures without deforming.  So, if we apply the twistor rotation by $\zeta$ to the holomorphic symplectic structure $\omega_2 + i\omega_3$ on $\mon_G^{\text{red}}(\bb{CP}^1 \times S^1_R,D,\omega^\vee)$ (the argument will be identical for the other points in the $\bb{CP}^1$ family of holomorphic symplectic structures) we obtain the pairing
\[\omega_\zeta(\delta^{(1)}\mc A, \delta^{(2)}\mc A) = \int_{\RR^3/(\zeta(L_{R,0}))} \kappa(\delta^{(1)} \mc A, \delta^{(2)} \mc A) \dvol.\]
On the other hand we can identify the holomorphic symplectic structure on $\mon_G^{\red}(\bb{CP}^1 \times_\eps S^1_R,D,\omega^\vee)$ as
\[\omega_\zeta(\delta^{(1)}\mc A, \delta^{(2)}\mc A) = \int_{\RR^3/(L_{R,\eps})} \kappa(\delta^{(1)} \mc A, \delta^{(2)} \mc A) \dvol.\]
In particular, by tuning the radius and the shear parameter the above calculation means there's a hyperk\"ahler equivalence between the moduli spaces $\mon_G^{\text{red}}(\bb{CP}^1 \times S^1_R,D,\omega^\vee)$ -- after rotating the twistor sphere by $\zeta$ -- and $\mon_G^{\red}(\bb{CP}^1 \times_{\frac{\eps}{1+|\zeta/2|^2}} S^1_{R \frac {1-|\zeta/2|^2}{1+|\zeta/2|^2}},f(D),\omega^\vee)$, where the divisors $D$ and $f(D)$ will coincide in the limit $R \to \infty$.  

Now, let us consider the large $R$ limit.  We'd like to note that the following difference of symplectic pairings converges to zero pointwise as $R\to \infty$.
\[ \left\lvert \int_{\RR^3/(L_{R \frac {1-|\zeta/2|^2}{1+|\zeta/2|^2},\frac{\eps}{1+|\zeta/2|^2}})} \kappa(\delta^{(1)} \mc A, \delta^{(2)} \mc A) \dvol - \int_{\RR^3/(L_{R,\eps})} \kappa(\delta^{(1)} \mc A, \delta^{(2)} \mc A) \dvol \right \rvert.\]
Note that here we're abusing notation, and using the canonical isomorphism of moduli spaces to identify tangent vectors $\delta \mc A$ to monopoles on $\RR^3/(L_{R,\eps})$, and their images under the canonical isomorphism to the moduli space of monopoles on $\RR^3/(L_{R \frac {1-|\zeta/2|^2}{1+|\zeta/2|^2}})$.

Rescaling the radii of the circles only rescales the pairing by an overall constant, so we can rewrite this as 
\[ R\left\lvert \frac {1-|\zeta/2|^2}{1+|\zeta/2|^2} \int_{\RR^3/(L_{1,\frac{\eps}{1+|\zeta/2|^2}})} \kappa(\delta^{(1)} \mc A, \delta^{(2)} \mc A) \dvol - \int_{\RR^3/(L_{1,\eps})} \kappa(\delta^{(1)} \mc A, \delta^{(2)} \mc A) \dvol \right \rvert\]
in which we recall that $\zeta = \frac \eps {2\pi R}$.  Write $\delta^{(i)}\mc A_\zeta$ for the image of the deformation $\delta^{(i)}\mc A$ under the diffeomorphism induced by identifying the $\frac{\eps}{1+|\zeta/2|^2}$-twisted product with the $\eps$-twisted product, so $\delta^{(i)}\mc A_\zeta(z,t) = \delta^{(i)}\mc A(z-\frac{\eps}{1+|\zeta/2|^2},t)$.  Note that when we perform the integral the pairing between $\delta^{(1)}\mc A$ and $\delta^{(2)}\mc A$ and the pairing between $\delta^{(1)}\mc A_\zeta$ and $\delta^{(2)}\mc A_\zeta$ coincide.  In other words, our difference of pairings becomes
\begin{align*} 
&R\left\lvert \int_{\bb{CP}^1 \times_\eps S^1_1} \frac {1-|\zeta/2|^2}{1+|\zeta/2|^2} \kappa(\delta^{(1)} \mc A_\zeta, \delta^{(2)} \mc A_\zeta)  -  \kappa(\delta^{(1)} \mc A, \delta^{(2)} \mc A) \dvol \right \rvert \\
&\quad = R \left\lvert \frac {1-|\zeta/2|^2}{1+|\zeta/2|^2} - 1 \right\rvert \left\lvert \int_{\bb{CP}^1 \times_\eps S^1_1} \kappa(\delta^{(1)} \mc A, \delta^{(2)} \mc A) \dvol \right \rvert \\
&\quad = \left\lvert\frac {2R |\eps|^2}{16 \pi^2 R^2 + |\eps|^2} \right \rvert \left\lvert \int_{\bb{CP}^1 \times_\eps S^1_1} \kappa(\delta^{(1)} \mc A, \delta^{(2)} \mc A) \dvol \right \rvert.
\end{align*}
This converges to zero pointwise as $R \to \infty$.  The same calculation proves convergence for the other symplectic structures in the hyperk\"ahler family, and therefore convergence for the difference of the two complex structures.
\end{proof}

\begin{example} \label{Nahm_example}
In the case where the group $G=\GL_n$ it's possible to describe the multiplicative Hitchin system on $\bb{CP}^1$ in terms of an \emph{ordinary} Hitchin system on $\CC^\times$, but for a different group $\GL_k$ and with different singularity data.  This relationship is given by the \emph{Nahm transform}.  

Let us give an informal explanation.  From a gauge-theoretic point of view, the Nahm transform arises from the Fourier-Mukai transform for anti-self-dual connections on a 4-torus $T$, relating an of ASD connection on a vector bundle to an ASD connection on the bundle of kernels for the associated Dirac operator, a vector bundle on the dual torus $T^\vee$ (see the book of Donaldson--Kronheimer \cite[Section 3.2]{DonaldsonKronheimer}).  There is a Nahm transform relating the ordinary Hitchin system to the moduli space of periodic monopoles obtained in an appropriate limit: consider the torus $T = S^1_{r_1} \times S^1_{r_2} \times S^1_{r_3} \times S^1_{R}$.  The moduli space of periodic monopoles on $\CC$ arises in the limit where $r_1$ and $r_2$ go to $\infty$ and $r_3$ goes to zero.  On the other hand, the Hitchin system on $\CC^\times$ arises in the dual limit where $r_1$ and $r_2$ go to zero and $r_3$ goes to $\infty$.  The Nahm transform in this setting -- i.e. for periodic monopoles -- was developed by Cherkis and Kapustin \cite{CherkisKapustin2}, and worked out in detail in the case $n=2$ (they also give a brane description of the transform: see \cite[Section 2]{CherkisKapustin2}).

One then has to match up the singularity data.  In brief, the multiplicative Hitchin system for the group $\GL_n$ with $k$ singularities at $z_1, \ldots, z_k$ is equivalent to the ordinary Hitchin system on $\CC^\times$ for the group $\GL_k$ with $n$ regular singularities, which residues determined by the original positions $z_i$ and positions determined by the original local data $\omega^\vee_{z_i}$.  We refer to \cite[Section 7.1]{NekrasovPestun} for more details.

In the example where we have the Nahm transform available, we conjecture that the multiplicative geometric Langlands equivalence recovers the ordinary geometric Langlands equivalence with tame ramification.  That is, we expect the following.

\begin{claim}
Under the Nahm transform, Pseudo-Conjecture \ref{multLanglands} in the rational case for the group $\GL(n)$ becomes the ordinary geometric Langlands conjecture on $\bb{CP}^1$ with tame ramification.
\end{claim}
\end{example}

\begin{remark}
  The limit $R \to \infty$ with $2 \pi \zeta R = \eps$ appearing in the hyperk\"ahler rotation
  theorem \ref{HK_rotation_thm} was first suggested by Gaiotto \cite{GaiottoTBA} on the other side of the Nahm transform.
\end{remark}

\section{$q$-Opers and $q$-Characters} \label{qchar_section}

In this final section we will discuss the space of $q$-Opers in more depth.  In particular we will connect the geometric setup described in this paper, in terms of multiplicative Higgs bundles, to the gauge theoretic story studied by the second author and collaborators \cite{NekrasovPestunShatashvili,Kimura:2015rgi,Nekrasov:2015wsu}.  The main goal of this subsection will be to describe and motivate a connection between $q$-Opers and the $q$-character maps from the theory of quantum groups \cite{FrenkelReshetikhinSTS,FrenkelReshetikhin2,FrenkelReshetikhin1,Sevostyanov,STSSevostyanov,Sevostyanov1}.  In order to make our statements as concrete as possible it will be useful to first describe the Steinberg section \cite{Steinberg} of a semisimple group explicitly.

Throughout this section, assume that $G$ is a simple, simply-laced and simply-connected Lie group with Lie algebra $\gg$.  Let $\Delta = \{\alpha_1, \ldots, \alpha_r\}$ be the set of simple roots of $\gg$.  In order to define the Steinberg section uniquely we'll fix a \emph{pinning} on $G$.  That is, choose a Borel subgroup $B \sub G$ with maximal torus $T$ and unipotent radical $U$, and choose a generator $e_i$ for each simple root space $\gg_{\alpha_i}$.

We'll also choose an element $\sigma_i \in N(T)$ in the normalizer of $T$ representing each element of the Weyl group $W = N(T)/T$.  The Steinberg section will be independent of this choice up to conjugation by a unique element of $T$, and independent of the ordering on the set of simple roots.

\begin{definition} \label{Steinberg_section_def}
The \emph{Steinberg section} of $G$ associated to a choice of pinning is the image of the injective map $\sigma \colon T/W \to G$ defined by
\[\sigma(t_1, \ldots, t_r) = \prod_{i=1}^r \exp(t_i e_i) \sigma_i.\]
Steinberg proved \cite[Theorem 1.4]{Steinberg} that, after restriction to the regular locus in $G$, the map $\sigma$ defines a section of the Chevalley map $\chi \colon G \to T/W$.
\end{definition}

\begin{definition} \label{mhitch_section_def}
Fix a coloured divisor $(D,\omega^\vee)$ The \emph{multiplicative Hitchin section} of the map $\pi \colon \mhiggs^\fr_G(\bb{CP}^1,D,\omega^\vee) \to \mc B(D,\omega^\vee)$ is the image $\mhitch^\fr_G(\bb{CP}^1, D, \omega^\vee)$ of the map defined by post-composing a meromorphic $T/W$-valued function on $\bb{CP}^1$ with the Steinberg map $\sigma$.
\end{definition}

\begin{remark}
The multiplicative Hitchin section is indeed a section of the map $\pi$ after restricting to the connected component in $\mhiggs^\fr_G(\bb{CP}^1,D,\omega^\vee)$ corresponding to the trivial bundle, provided one chooses a value for the framing within the Steinberg section.  For example, if we choose the identity framing on the multiplicative Hitchin basis then the multiplicative Hitchin section lands in multiplicative Higgs bundles with framing $c = \sigma(1)$ at infinity, i.e. framing given by a Coxeter element.
\end{remark}

Now, let's introduce the key idea in this section: the notion of \emph{triangularization} for the multiplicative Hitchin section.
\begin{definition} \label{gen_evals_def}
Let $g(z)$ be an element of $\mhitch^\fr_G(\bb{CP}^1, D, \omega^\vee)$.  We'll abusively identify $g(z)$ with its image under the restriction map $r_\infty \colon \mhiggs_G^\fr(\bb{CP}^1,D,\omega^\vee) \to G_c[[z^{-1}]]$ to a formal neighbourhood of $\infty$.  Say that $g(z)$ has \emph{generalized eigenvalues} $y(z) \in T[[z^{-1}]]$ if there exists an element $u(z)$ of $U[[z^{-1}]]$ such that $u(z)g(z)u(z)^{-1}$ is an element of $B_-[[z^{-1}]]$, where $B_-$ is the opposite Borel subgroup to $B$, which maps to $y(z)$ under the canonical projection.

We say that $g(z)$ has \emph{$q$-generalized eigenvalues} $y(z) \in T[[z^{-1}]]$ if there exists an element $u$ of $U[[z^{-1}]]$ such that $u(q^{-1}z)g(z)u(z)^{-1}$ is an element of $B_-[[z^{-1}]]$ that maps to $y(z)$ under the canonical projection. 
\end{definition}

\begin{remark}
In this setup $q$ is the automorphism of the formal disk obtained by restricting an automorphism of $\CC$ or $\CC^\times$ to its formal punctured neighbourhood.  In particular we'll use this notation for the additive deformation we denoted by $\eps$ in previous sections, even though we'll use the multiplicative notation.  We make this notational choice in order to allow more direct comparison with the previous gauge theory literature.
\end{remark}

It's sometimes useful to use a slightly different representation of the multiplicative Hitchin section, packaging the singularity datum in a more uniform way. 

Choose a point $b(z)$ in the multiplicative Hitchin base $B(D,\omega^\vee)$. By clearing denominators, we can identify $b(z)$ with a canonical polynomial $t'(z)$ in $T[z]/W$ of fixed degree $d$, with fixed top degree term.  Specifically, one finds
\[d_i = \sum_{z_j \in D} \omega^\vee_{z_j}(\lambda_i)\]
where $\lambda_i$ is the $i^\text{th}$ fundamental weight.  We'll denote the locus of polynomials of this form by $T[z]_d$.

We define an embedding from $T[z]_d/W$ into $G[[z^{-1}]]$ as follows.  First, we encode the data of the coloured divisor $(D, \omega^\vee)$ in terms of a $T$-valued polynomial.  That is, we set 
\[p_i(z) = \prod_{z_j \in D} (z-z_j)^{\omega^\vee_{z_j}(\omega_i)}\]
where $\omega_j$ is the $j^{\text{th}}$ fundamental weight.

\begin{definition}
The \emph{$p$-twisted multiplicative Hitchin section} is the image of the map $\sigma' \colon T[z]_d/W \to G(z)$ defined by
\begin{equation}
\label{eq:steinberg}
\sigma'(t(z)) = \prod_{i=1}^r \exp\left(t'_{i}(z) e_i\right) \sigma_i p_i(z)^{-\omega^\vee_i}
\end{equation}
where $\omega^\vee_i$ is the $i^\text{th}$ fundamental coweight.
\end{definition}

This $p$-twisted section is obtained from the ordinary multiplicative Hitchin section after conjugation by an element of $G[z]$.  The result is no longer directly related to the multiplicative Hitchin moduli space because this conjugation breaks the framing at $\infty$.  It is, however, sometimes useful for computation. One can define the $q$-triangularization of a point in this $p$-twisted section exactly as in the ordinary case.

Now, let's introduce the other main object we'll be discussing in this section.
 
\begin{definition}
The \emph{$q$-character} is an algebra homomorphism 
\[\chi_q \colon \mr{Rep}(Y(\gg)) \to \OO(T_1[[z^{-1}]])\]
from the ring of finite-dimensional representations of the Yangian to the ring of functions on $T_1[[z^{-1}]]$, first defined by Knight \cite{Knight}.  Given a choice of singularity datum $(D,\omega^\vee)$ one can define a linear map $\mc B(D, \omega^\vee)^* \to \mr{Rep}(Y(\gg))$ generated by the map identifying a coordinate functional on the multiplicative Hitchin base space with the irreducible representation with the given highest weight.  We'll abusively also denote the composite map $\mc B(D, \omega^\vee)^* \to \OO(T_1[[z^{-1}]])$ by $\chi_q$.
\end{definition}

The $p$-twisted fundamental $q$-characters $t'_i(z)$ are obtained by recursive Weyl reflections of the form \cite{NekrasovPestun,NekrasovPestunShatashvili,Nekrasov:2015wsu} (following the conventions from \cite{Kimura:2015rgi} for shifts)
\begin{equation}
  w \cdot_q y_i(z)) = y_{i}(z q^{-1})^{-1} \Bigg(\prod_{e: i \to j} y_j(z) \Bigg) \Bigg(\prod_{e: j\to i} y_j(q^{-1} z)\Bigg)  p_{i}( q^{-1} z).
  \label{iWeyl_action}
\end{equation}
starting from the highest weight monomial $t'_i(z) = y_i(z) +  w \cdot_q y_i(z)) + \dots $. 

These formulae appeared previously in \cite{NekrasovPestun, NekrasovPestunShatashvili} and
should be interpreted as $p$-twisted version of the $q$-Weyl reflections of \cite{FrenkelReshetikhin2,FrenkelReshetikhinSTS,FrenkelReshetikhin1,Frenkel2001}.

In the undeformed case, the $y(z)$-functions have a very geometric meaning: they are equivalent to a sequence of algebraic functions defining the cameral cover at a point $b(z)$ in the Hitchin base.  The cameral cover associated to a point $b(z)$ in the multiplicative Hitchin base is, to put it briefly, obtained away from the singular locus as the fiber product $\mc C_b^\circ = (\bb{CP}^1_z \bs D) \times_{T/W} T$, where the map $\bb{CP}^1_z \bs D \to T/W$ is the meromorphic map corresponding to $b(z)$.  In other words, the closed points of $\mc C^\circ_b$ are simply pairs $(z, y) \in (\bb{CP}^1 \bs D) \times T$ such that $b(z) = [y]$ in $T/W$.  One can obtain the cameral curve as the graph of a $|W|$-valued
meromorphic function $\bb{CP}^1 \to T$ by
taking all images $y$ under the projection $T \to T/W$, $y \mapsto [y]$ where elements of $T$ are
coordinatized as  $\prod_{i = 1}^{r} y_i^{\alpha_i^{\vee}}$ where $\alpha_i^{\vee}$ are simple coroots.
Then, on a contractible local patch in the complement of the ramification locus define $t_i(z)$ to be the sum 
\[t_i(z) = \sum_{y \text{ over } b} \omega_i(y(z)),\]
where the sum is over all local lifts $y$ of the $T/W$-valued function $b(z)$, and where $\omega_i \colon T \to \mathbb{C}^{\times}$ is the $u^\text{th}$ fundamental weight. Because of the $W$-invariance, for each $i$ these local descriptions glue together to define a global $\mathbb{C}$-valued function on the curve $\bb{CP}^1_z \bs D$.

In the classical limit $q=1$ the procedure of taking generalized eigenvalues and the character map are closely related, simply because the multiplicative Hitchin section is a section of the Hitchin fibration.  The character map, evaluated in the fundamental representations associated to the fundamental weights $\omega_1, \ldots, \omega_r$, is essentially the Chevalley map sending a matrix to its characteristic polynomial.  In other words it is essentially the multiplicative Hitchin fibration itself.  Therefore, we find that 
\[\chi(t_i)(\sigma(b(z))) = t_i(z)\]
up to an affine isomorphism.  The same holds when we replace $\sigma(b(z))$ by $y(z)$, its triangularization, because the maps $\chi(t_i)$ are adjoint-invariant.

We conjecture that this relationship survives after $q$-deformation, yielding the following relationship between $q$-opers and the $q$-character.  This expectation is motivated in part by the multiplicative geometric Langlands Conjecture \ref{multLanglands} applied to the canonical coisotropic A-brane -- quantizing the ring of Poisson commuting functions on the multiplicative Hitchin system -- whose Langlands dual is expected to be precisely the brane of $q$-opers.
 
\begin{conjecture} \label{qchar_conjecture}
\begin{enumerate}
\item For any $q$, every element of $g(z)$ of $\mhitch^\fr_G(\bb{CP}^1, D, \omega^\vee)$ has a unique $q$ generalized eigenvalue, and therefore there is a well-defined map
\[E_q \colon \mc B(D,\omega^\vee) \to T[[z^{-1}]] \sub G[[z^{-1}]]\]
given by applying the multiplicative Hitchin section then computing its $q$-generalized eigenvalues.  

\item The composite $E_q^* \circ \chi_q \colon \mc B(D,\omega^\vee)^* \to \OO(\mc B(D,\omega^\vee))$ is an affine isomorphism onto the space of linear functionals $\mc B(D,\omega^\vee)^* \sub \OO(\mc B(D,\omega^\vee))$.
\end{enumerate}
\end{conjecture}

Let's demonstrate this conjecture in a few examples.  In each case we'll describe the multiplicative Hitchin section, and then calculate an inverse to the map sending a point $t(z)$ in the base space $\mc B(D,\omega^\vee)$ to the corresponding generalized eigenvalue map for the Hitchin section $y(z)$.  If $y_i(t(z))$ is the composite of the $q$-generalized eigenvalue map with the $i^\text{th}$ fundamental weight $\omega_i$, we must verify that $\chi_q(y_i(t)) = t_i$, the $i^\text{th}$ coordinate map (up to an affine isomorphism).

\begin{examples}
\begin{enumerate}
 \item In type $A_1$ we can calculate everything very explicitly, and so verify the conjecture in this example.  We've already described the multiplicative Hitchin section in Section \ref{GL2_example_section}: for $G = \SL_2$, for instance, it consists of matrices of the form
\begin{equation*}
  g^{t} =
  \begin{pmatrix}
    t   & - 1 \\
    1 & 0
  \end{pmatrix},
\end{equation*}
where $t$ is a rational function with fixed denominator of degree $d$, and arbitrary numerator of degree less than $d$.  We would like to triangularize this to obtain a matrix of the form
\begin{equation*}
  g^{y} =
  \begin{pmatrix}
    y  & 0 \\
    1 & y^{-1} 
  \end{pmatrix}.
\end{equation*}
It's easy to solve the equation $g^t = u(q^{-1}z)g^y(z) u^{-1}(z)$ explicitly.  One finds a solution with $t(z) = y(z) + y(q^{-1}z)^{-1}$, after conjugation by the element $u(z) = - y(z)^{-1}$.  As the conjecture tells us to expect, $t(z)$ is identified with a $q$-twisted Weyl invariant polynomial in $\mathbb{C}[y, y^{-1}]$ which starts from the highest weight monomial $y(z)$.   This is enough to identify it with the $q$-character as required.

 \item We can also make concrete calculations for type $A_2$, again verifying Conjecture \ref{qchar_conjecture}.  For more direct comparison with the formulae in the literature we'll use the $p$-twisted formulation discussed above involving polynomials $p_i(z)$ encoding the singularity datum $(D, \omega^\vee)$.  We label positive roots as $\alpha_1, \alpha_2$ and $\alpha_3:=\alpha_1 + \alpha_2$ and parametrize a $U[[z^{-1}]]$-valued
  gauge transformation $u(z)$ by the collection of functions  
 $(u_{i}(z))_{\alpha_i  \in \Delta^{+}}$ 
  \begin{equation*}
    u(z) = \prod_{3,2,1} \exp( u_{i}(z) e_{\alpha_i})
  \end{equation*}
Then solving the equation 
\begin{equation*}
g^t(z) =  u(q^{-1} z)^{-1} g^y(z) u(z)
\end{equation*}
for $u_{1}(z), u_{2}(z), u_{3}(z)$ and $t'_{1}(z), t'_{2}(z)$  we find that
\begin{equation*}
\begin{aligned}
& u_{1}(z) = p_{1}(z) u_{2}(q^{-1} z)-p_{1}(z) y_{2}(z) y_{1}(z)^{-1} \\
& u_{2}(z) =  -p_{2}(z) y_{2}(z)^{-1} \\
& u_{3}(z) = -p_{1}(z) p_{2}(z) y_{1}(z)^{-1} \\
& t'_{1}(z) = y_{1}(z)-u_{1}(q^{-1}z) \\
& t'_{2}(z) = y_{2}(z) - y_{1}(z) u_{2}(q^{-1}z)-u_{3}(q^{-1}z) \\
\end{aligned}
\end{equation*}
which implies in turn that
\begin{equation*}
  \begin{aligned}
    t'_{1}(z) =y_{1}(z)  +  \frac{p_{1}(q^{-1}z) y_{2}(q^{-1}z)}{y_{1}(q^{-1}z)} + \frac{p_{1}(q^{-1}z) p_{2}(q^{-2} z)}{ y_{2}(q^{-2} z)}\\
    t'_{2}(z) = y_{2}(z)  +\frac{y_{1}(z) p_{2}(q^{-1}z)}{y_{2}(q^{-1}z)}+   \frac{p_{1}(q^{-1}z) p_{2}(q^{-1}z)}{y_{1}(q^{-1}z)}
  \end{aligned}
\end{equation*}
and that indeed coincides with the expression for the $q$-characters for the $A_2$ quiver appearing in
\cite{Nekrasov:2015wsu,NekrasovPestunShatashvili,NekrasovPestun,Kimura:2015rgi}.

\item The first example for which the condition that the composite map in Conjecture \ref{qchar_conjecture} is a non-trivial affine isomorphism occurs for the $D_4$ Dynkin diagram, for instance for the group $G = \SO(8)$.  In that case, even for $q=1$ we find the composite $E^* \circ \chi$, i.e. the result of applying the character map in the fundamental representations given by fundamental weights $(\omega_1, \omega_2,\omega_2,\omega_4)$ to the Steinberg section, is not quite the identity, but rather than map sending $(t_1,t_2,t_3,t_4)$ to $(t_1,t_2+1,t_3,t_4)$.  One sees this by considering the $G$-modules associated to irreducible finite-dimensional highest weight modules of the Yangian $Y(\gg)$ in the limit $q \to 1$: the 29-dimensional module induced in this way from the fundamental weight $\omega_2$ splits as the sum of the 28 dimensional adjoint representation associated to the fundamental weight $\omega_2$ and the trivial representation.  We include some calculations verifying the conjecture for $G=\SO(8)$ in Appendix \ref{D4_appendix}.
\end{enumerate}
\end{examples}

\begin{remark}
After triangularizing, an element of the multiplicative Hitchin section transforms into a function on $\bb{CP}^1$ with apparent singularities (even away from the divisor $D$) where the functions $y_{i}(z)$ have zeroes and poles.  However, when we apply the $q$-character map these singularities are cancelled in pairs between the monomial summands of Equation \ref{iWeyl_action}, as we sum over the transformation by each element of the Weyl group.  This cancellation property has been called ``regularity of the $q$-character'' in the quantum group representation theory literature \cite{FrenkelReshetikhin1,FrenkelReshetikhin2} as well as in the gauge theoretic construction of \cite{NekrasovPestunShatashvili,Nekrasov:2015wsu,NekrasovPestun,Kimura:2015rgi}.
 
The meromorphic functions $y_{i}(z)$ can be expressed as ratios of the form $y_{i}(z) = Q_{i}(z)/Q_i(q^{-1}z)$, and the zeroes of the functions denoted $Q_i$ are known as \emph{Bethe roots} in the context of the Bethe ansatz equations.  In this language, the Bethe ansatz equations are precisely the equations which ensures that poles in $t_{i}(z)$ are cancelled. 
\end{remark}

 \begin{remark}
   Some of the results in a recent paper by Koroteev, Sage and Zeitlin \cite{KoroteevSageZeitlin} seem to overlap
   with the specialization of the main conjecture of this section to the group $G  = \SL(N)$, in the case where the the
   coweight coloured divisor $(D,\omega^\vee)$ takes a particular form: where the functions $p_i(z)$ encoding $(D,\omega^\vee)$ can be related to the ratio of Drinfeld polynomials $\rho_i(z)/\rho_i(q^{-1} z)$ in the Bethe ansatz equations. 
   This special form for the singularity datum means that the Yangian module obtained by quantization of the symplectic leaf $\mhiggs_G(C,D, \omega^\vee)$ contains (as a quotient) the finite-dimensional Drinfeld module specified by the Drinfeld polynomials $\rho_{i}(z)$. In the relation to the occurence of the Bethe ansatz equations in quiver gauge theory \cite{NekrasovPestunShatashvili}, this specialization is known as 4d to 2d specialization, and was first studied by Chen, Dorey, Hollowood and Lee \cite{ChenDoreyHollowoodLee, DoreyHollowoodLee}. 
\end{remark}

\appendix
\section{$q$-Character Calculations for $G=\SO(8)$} \label{D4_appendix}

\allowdisplaybreaks

We will compute the $q$-triangulization of the $p$-twisted multiplicative Hitchin section in the example of where $G = \SO(8)$.  We label the simple roots as $\alpha_i$ with $i = 1, \dots, 4$, and we fix the Cartan matrix to be 
\[
\langle \alpha_i^{\vee}, \alpha_j \rangle  =
\begin{pmatrix}
 2 & -1 & 0 & 0 \\
 -1 & 2 & -1 & -1 \\
 0 & -1 & 2 & 0 \\
 0 & -1 & 0 & 2 \\
\end{pmatrix}
\]
We then label the positive roots as $\alpha_1, \dots, \alpha_{12}$ where the non-simple roots can be decomposed as
\begin{align*}
 &\alpha _5= \alpha _1+\alpha _2 &\alpha _9= \alpha _1+\alpha _2+\alpha _4 \\
 &\alpha _6= \alpha _2+\alpha _3 &\alpha _{10}= \alpha _2+\alpha _3+\alpha _4 \\
 &\alpha _7= \alpha _2+\alpha _4 &\alpha _{11}= \alpha _1+\alpha _2+\alpha _3+\alpha _4 \\
 &\alpha _8= \alpha _1+\alpha _2+\alpha _3 &\alpha _{12}= \alpha _1+2 \alpha _2+\alpha _3+\alpha _4. \\
\end{align*}
Solving for the $q$-adjoint transformation that triangualizes the $p$-twisted multiplicative Hitchin section (\ref{eq:steinberg}) to bring it to the $B_{-}$ valued $q$-connection of the form $ \prod_{i} \exp(f_i y_i^{-1}) y_i^{\alpha_i^{\vee}} p_i^{-\omega_i^{\vee}}$ (where $(e_i, f_i)$ label standard Chevalley generators in positive and negative simple root spaces respectively) we will first identify the element $u(z)$ of $U[[z^{-1}]]$ by which we will act.  We write it as a sum over Chevalley generators of elements of the form $\exp(u_i(z) e_i)$ where the $u_i(z)$ are given as follows.
\begin{align*}
 u_{1}= &-\frac{p_{1,-2} p_{1} p_{2,-2} p_{2,-1} p_{3,-2} p_{4,-2}}{y_{1,-2}}-\frac{p_{1} p_{2,-2} p_{2,-1} p_{3,-2} p_{4,-2} y_{1,-1}}{y_{2,-2}}-\frac{p_{1} p_{2,-1} p_{4,-2} y_{3,-1}}{y_{4,-2}}\\&-\frac{p_{1} p_{2,-1}
   p_{3,-2} p_{4,-2} y_{2,-1}}{y_{3,-2} y_{4,-2}}-\frac{p_{1} y_{2}}{y_{1}}-\frac{p_{1} p_{2,-1} y_{3,-1} y_{4,-1}}{y_{2,-1}}-\frac{p_{1} p_{2,-1} p_{3,-2} y_{4,-1}}{y_{3,-2}} \\
 u_{2}= &-\frac{p_{1,-1} p_{2,-1} p_{2} p_{3,-1} p_{4,-1}}{y_{1,-1}}-\frac{p_{2,-1} p_{2} p_{3,-1} p_{4,-1} y_{1}}{y_{2,-1}}-\frac{p_{2} p_{4,-1} y_{3}}{y_{4,-1}}-\frac{p_{2} p_{3,-1} p_{4,-1} y_{2}}{y_{3,-1}
   y_{4,-1}}\\&-\frac{p_{2} y_{3} y_{4}}{y_{2}}-\frac{p_{2} p_{3,-1} y_{4}}{y_{3,-1}} \\
 u_{3}= &-\frac{p_{3}}{y_{3}}, \qquad  u_{4}= -\frac{p_{4}}{y_{4}} \\
 u_{5}= &-\frac{p_{1,-1} p_{1} p_{2} p_{3,-1} p_{4,-1} y_{3,-1} y_{4,-1} p_{2,-1}^2}{y_{1,-1} y_{2,-1}}-\frac{p_{1} p_{2} p_{3,-1} p_{4,-1} y_{1} y_{3,-1} y_{4,-1} p_{2,-1}^2}{y_{2,-1}^2}\\&-\frac{p_{1,-1} p_{1}
   p_{2} p_{3,-2} p_{3,-1} p_{4,-1} y_{4,-1} p_{2,-1}^2}{y_{1,-1} y_{3,-2}}-\frac{p_{1} p_{2} p_{3,-2} p_{3,-1} p_{4,-1} y_{1} y_{4,-1} p_{2,-1}^2}{y_{2,-1} y_{3,-2}}\\&-\frac{p_{1,-1} p_{1} p_{2,-2} p_{2} p_{3,-2}
   p_{3,-1} p_{4,-2} p_{4,-1} p_{2,-1}^2}{y_{2,-2}}-\frac{p_{1,-1} p_{1} p_{2} p_{3,-1} p_{4,-2} p_{4,-1} y_{3,-1} p_{2,-1}^2}{y_{1,-1} y_{4,-2}}\\&-\frac{p_{1} p_{2} p_{3,-1} p_{4,-2} p_{4,-1} y_{1} y_{3,-1}
   p_{2,-1}^2}{y_{2,-1} y_{4,-2}}-\frac{p_{1} p_{2} p_{3,-2} p_{3,-1} p_{4,-2} p_{4,-1} y_{1} p_{2,-1}^2}{y_{3,-2} y_{4,-2}}\\&-\frac{p_{1,-1} p_{1} p_{2} p_{3,-2} p_{3,-1} p_{4,-2} p_{4,-1} y_{2,-1} p_{2,-1}^2}{y_{1,-1}
   y_{3,-2} y_{4,-2}}-\frac{p_{1,-1} p_{1} p_{2} p_{3,-1} p_{4,-1} y_{2} p_{2,-1}}{y_{1,-1} y_{1}}\\&-\frac{2 p_{1} p_{2} p_{3,-1} p_{4,-1} y_{2} p_{2,-1}}{y_{2,-1}}-\frac{p_{1} p_{2} p_{4,-1} y_{3,-1} y_{3}
   p_{2,-1}}{y_{2,-1}}-\frac{p_{1} p_{2} p_{3,-2} p_{4,-1} y_{3} p_{2,-1}}{y_{3,-2}}-\frac{p_{1} p_{2} p_{3,-1} y_{4,-1} y_{4} p_{2,-1}}{y_{2,-1}}\\&-\frac{p_{1} p_{2} p_{3,-1} p_{4,-2} y_{4}
   p_{2,-1}}{y_{4,-2}}-\frac{p_{1} p_{2} p_{3,-2} p_{3,-1} p_{4,-1} y_{2} p_{2,-1}}{y_{3,-2} y_{3,-1}}-\frac{p_{1} p_{2} p_{3,-1} p_{4,-2} p_{4,-1} y_{2} p_{2,-1}}{y_{4,-2} y_{4,-1}}-\frac{p_{1} p_{2} y_{3}
   y_{4}}{y_{1}}\\&-\frac{p_{1} p_{2} p_{3,-1} y_{2} y_{4}}{y_{1} y_{3,-1}}-\frac{p_{1} p_{2} p_{4,-1} y_{2} y_{3}}{y_{1} y_{4,-1}}-\frac{p_{1} p_{2} p_{3,-1} p_{4,-1} y_{2}^2}{y_{1} y_{3,-1} y_{4,-1}}
   \\
 u_{6}= &-\frac{p_{2} p_{3} p_{4,-1}}{y_{4,-1}}-\frac{p_{2} p_{3} y_{4}}{y_{2}} \\
 u_{7}= &-\frac{p_{2} p_{4} y_{3}}{y_{2}}-\frac{p_{2} p_{3,-1} p_{4}}{y_{3,-1}} \\
  u_{8}= &-\frac{p_{1} p_{2,-1} p_{2} p_{3} p_{4,-1} y_{3,-1}}{y_{2,-1}}-\frac{p_{1} p_{2,-1} p_{2} p_{3,-2} p_{3} p_{4,-1}}{y_{3,-2}}-\frac{p_{1} p_{2} p_{3} p_{4,-1} y_{2}}{y_{1} y_{4,-1}}-\frac{p_{1}
   p_{2} p_{3} y_{4}}{y_{1}} \\
 u_{9}= &-\frac{p_{1} p_{2} p_{4} y_{3}}{y_{1}}-\frac{p_{1} p_{2,-1} p_{2} p_{3,-1} p_{4} y_{4,-1}}{y_{2,-1}}-\frac{p_{1} p_{2} p_{3,-1} p_{4} y_{2}}{y_{1} y_{3,-1}}-\frac{p_{1} p_{2,-1} p_{2}
   p_{3,-1} p_{4,-2} p_{4}}{y_{4,-2}} \\
  u_{10}= &-\frac{p_{2} p_{3} p_{4}}{y_{2}} \\
 u_{11}= &-\frac{p_{1} p_{2} p_{3} p_{4}}{y_{1}} \\
 u_{12}= &-\frac{p_{1} p_{3} p_{4} p_{2}^2 y_{3} y_{4}}{y_{1} y_{2}}-\frac{p_{1} p_{3,-1} p_{3} p_{4} p_{2}^2 y_{4}}{y_{1} y_{3,-1}}-\frac{p_{1} p_{2,-1} p_{3,-1} p_{3} p_{4,-1} p_{4}
   p_{2}^2}{y_{2,-1}}-\frac{p_{1} p_{3} p_{4,-1} p_{4} p_{2}^2 y_{3}}{y_{1} y_{4,-1}}\\&-\frac{p_{1} p_{3,-1} p_{3} p_{4,-1} p_{4} p_{2}^2 y_{2}}{y_{1} y_{3,-1} y_{4,-1}}.
\end{align*}

and consequently we can calculate the functions $t_i'(z)$ in terms of the generalized eigenvalues, to find the following:
\begin{align*}
 t'_{1} = &\frac{p_{1,-3} p_{1,-1} p_{2,-3} p_{2,-2} p_{3,-3} p_{4,-3}}{y_{1,-3}}+\frac{p_{1,-1} p_{2,-3} p_{2,-2} p_{3,-3} p_{4,-3} y_{1,-2}}{y_{2,-3}}+\frac{p_{1,-1} p_{2,-2} p_{4,-3} y_{3,-2}}{y_{4,-3}}\\&+\frac{p_{1,-1} p_{2,-2}
   p_{3,-3} p_{4,-3} y_{2,-2}}{y_{3,-3} y_{4,-3}}+\frac{p_{1,-1} y_{2,-1}}{y_{1,-1}}+\frac{p_{1,-1} p_{2,-2} y_{3,-2} y_{4,-2}}{y_{2,-2}}\\&+\frac{p_{1,-1} p_{2,-2} p_{3,-3} y_{4,-2}}{y_{3,-3}}+y_{1} \\
 t'_{2}= &\frac{p_{1,-2} p_{1,-1} p_{2,-1} p_{3,-2} p_{4,-2} y_{3,-2} y_{4,-2} p_{2,-2}^2}{y_{1,-2} y_{2,-2}}+\frac{p_{1,-1} p_{2,-1} p_{3,-2} p_{4,-2} y_{1,-1} y_{3,-2} y_{4,-2} p_{2,-2}^2}{y_{2,-2}^2}\\&+\frac{p_{1,-2} p_{1,-1}
   p_{2,-1} p_{3,-3} p_{3,-2} p_{4,-2} y_{4,-2} p_{2,-2}^2}{y_{1,-2} y_{3,-3}}+\frac{p_{1,-1} p_{2,-1} p_{3,-3} p_{3,-2} p_{4,-2} y_{1,-1} y_{4,-2} p_{2,-2}^2}{y_{2,-2} y_{3,-3}}\\&+\frac{p_{1,-2} p_{1,-1} p_{2,-3} p_{2,-1} p_{3,-3}
   p_{3,-2} p_{4,-3} p_{4,-2} p_{2,-2}^2}{y_{2,-3}}+\frac{p_{1,-2} p_{1,-1} p_{2,-1} p_{3,-2} p_{4,-3} p_{4,-2} y_{3,-2} p_{2,-2}^2}{y_{1,-2} y_{4,-3}}+\\&\frac{p_{1,-1} p_{2,-1} p_{3,-2} p_{4,-3} p_{4,-2} y_{1,-1} y_{3,-2}
   p_{2,-2}^2}{y_{2,-2} y_{4,-3}}+\frac{p_{1,-1} p_{2,-1} p_{3,-3} p_{3,-2} p_{4,-3} p_{4,-2} y_{1,-1} p_{2,-2}^2}{y_{3,-3} y_{4,-3}}\\&+\frac{p_{1,-2} p_{1,-1} p_{2,-1} p_{3,-3} p_{3,-2} p_{4,-3} p_{4,-2} y_{2,-2}
   p_{2,-2}^2}{y_{1,-2} y_{3,-3} y_{4,-3}}+\frac{p_{1,-2} p_{2,-1} p_{3,-2} p_{4,-2} y_{1} p_{2,-2}}{y_{1,-2}}\\&+\frac{p_{1,-2} p_{1,-1} p_{2,-1} p_{3,-2} p_{4,-2} y_{2,-1} p_{2,-2}}{y_{1,-2} y_{1,-1}}+\frac{2 p_{1,-1} p_{2,-1}
   p_{3,-2} p_{4,-2} y_{2,-1} p_{2,-2}}{y_{2,-2}}\\&+\frac{p_{1,-1} p_{2,-1} p_{4,-2} y_{3,-2} y_{3,-1} p_{2,-2}}{y_{2,-2}}+\frac{p_{1,-1} p_{2,-1} p_{3,-3} p_{4,-2} y_{3,-1} p_{2,-2}}{y_{3,-3}}+\frac{p_{1,-1} p_{2,-1} p_{3,-2}
   y_{4,-2} y_{4,-1} p_{2,-2}}{y_{2,-2}}\\&+\frac{p_{1,-1} p_{2,-1} p_{3,-2} p_{4,-3} y_{4,-1} p_{2,-2}}{y_{4,-3}}+\frac{p_{2,-1} p_{3,-2} p_{4,-2} y_{1,-1} y_{1} p_{2,-2}}{y_{2,-2}}+\frac{p_{1,-1} p_{2,-1} p_{3,-3} p_{3,-2}
   p_{4,-2} y_{2,-1} p_{2,-2}}{y_{3,-3} y_{3,-2}}\\&+\frac{p_{1,-1} p_{2,-1} p_{3,-2} p_{4,-3} p_{4,-2} y_{2,-1} p_{2,-2}}{y_{4,-3} y_{4,-2}}+y_{2}+\frac{p_{1,-1} p_{2,-1} y_{3,-1} y_{4,-1}}{y_{1,-1}}+\frac{p_{2,-1} y_{1}
   y_{3,-1} y_{4,-1}}{y_{2,-1}}\\&+\frac{p_{2,-1} p_{3,-2} y_{1} y_{4,-1}}{y_{3,-2}}+\frac{p_{1,-1} p_{2,-1} p_{3,-2} y_{2,-1} y_{4,-1}}{y_{1,-1} y_{3,-2}}+\frac{p_{2,-1} p_{4,-2} y_{1} y_{3,-1}}{y_{4,-2}}+\frac{p_{1,-1} p_{2,-1}
   p_{4,-2} y_{2,-1} y_{3,-1}}{y_{1,-1} y_{4,-2}}\\&+\frac{p_{1,-1} p_{2,-1} p_{3,-2} p_{4,-2} y_{2,-1}^2}{y_{1,-1} y_{3,-2} y_{4,-2}}+\frac{p_{2,-1} p_{3,-2} p_{4,-2} y_{1} y_{2,-1}}{y_{3,-2} y_{4,-2}} \\
 t'_{3}= &\frac{p_{1,-1} p_{2,-2} p_{2,-1} p_{3,-1} p_{4,-2} y_{3,-2}}{y_{2,-2}}+\frac{p_{1,-1} p_{2,-2} p_{2,-1} p_{3,-3} p_{3,-1} p_{4,-2}}{y_{3,-3}}+\frac{p_{2,-1} p_{3,-1} p_{4,-2} y_{1}}{y_{4,-2}}\\&+\frac{p_{1,-1} p_{2,-1}
   p_{3,-1} p_{4,-2} y_{2,-1}}{y_{1,-1} y_{4,-2}}+\frac{p_{1,-1} p_{2,-1} p_{3,-1} y_{4,-1}}{y_{1,-1}}+\frac{p_{2,-1} p_{3,-1} y_{1} y_{4,-1}}{y_{2,-1}}+\frac{p_{3,-1} y_{2}}{y_{3,-1}}+y_{3} \\
 t'_{4}= &\frac{p_{1,-1} p_{2,-1} p_{4,-1} y_{3,-1}}{y_{1,-1}}+\frac{p_{2,-1} p_{4,-1} y_{1} y_{3,-1}}{y_{2,-1}}+\frac{p_{1,-1} p_{2,-2} p_{2,-1} p_{3,-2} p_{4,-1} y_{4,-2}}{y_{2,-2}}\\&+\frac{p_{2,-1} p_{3,-2} p_{4,-1}
   y_{1}}{y_{3,-2}}+\frac{p_{1,-1} p_{2,-1} p_{3,-2} p_{4,-1} y_{2,-1}}{y_{1,-1} y_{3,-2}}+\frac{p_{1,-1} p_{2,-2} p_{2,-1} p_{3,-2} p_{4,-3} p_{4,-1}}{y_{4,-3}}+\frac{p_{4,-1} y_{2}}{y_{4,-1}}+y_{4}.
\end{align*}

where we use the shorthand notation
\begin{equation*}
\begin{aligned}
y_{i,k} = y_i(q^{k}z) \\ p_{i,k} = p_i(q^{k} z)\\ t'_{i,k} = t'_{i}(q^{k} z)
\end{aligned}
\end{equation*}
 with $i, k \in \mathbb{Z}$ throughout. 

The formulae for these $t'_{i}(z)$ match the formulae for $p$-twisted $q$-characters computed in quiver gauge theory \cite{NekrasovPestunShatashvili}, see for instance the example of $qq$-characters in the $q_2 \to 1$ limit calculated in \cite{Nekrasov:2015wsu}. 

\bibliographystyle{alpha}
\bibliography{Mult_Hitchin}

\textsc{Institut des Hautes \'Etudes Scientifiques}\\
\textsc{35 Route de Chartres, Bures-sur-Yvette, 91440, France}\\
\texttt{celliott@ihes.fr}\\ 
\texttt{pestun@ihes.fr}
 
\end{document}